\documentclass[smallcondensed]{article}
\usepackage{amsfonts}
\usepackage{amsmath}
\usepackage{amssymb}
\usepackage{graphicx}%
\providecommand{\U}[1]{\protect\rule{.1in}{.1in}}
\newtheorem{theorem}{Theorem}

\newtheorem{corollary}[theorem]{Corollary}

\newtheorem{definition}[theorem]{Definition}

\newtheorem{lemma}[theorem]{Lemma}

\newtheorem{proposition}[theorem]{Proposition}
\newtheorem{remark}[theorem]{Remark}

\newenvironment{proof}[1][Proof]{\noindent\textbf{#1.} }{\ \rule{0.5em}{0.5em}}

\renewcommand{\thefootnote}{\fnsymbol{footnote}}
\begin{document}

\author{Lucian Maticiuc$\;^{a,1}$, Tianyang Nie$\;^{b,\ast,2}$, \bigskip\\{\small $^{a}$~Faculty of Mathematics, \textquotedblleft Alexandru Ioan
Cuza\textquotedblright\ University \&}\\{\small Department of Mathematics, \textquotedblleft Gheorghe
Asachi\textquotedblright\ Technical University,}\\{\small Carol I Blvd., no. 11, Ia\c{s}i, 700506, Romania}$\bigskip$\\{\small $^{b}$~School of Mathematics, Shandong University, Jinan, Shandong
250100, China \&}\\{\small Laboratoire de Math\'{e}matiques, CNR-UMR 6205 Universit\'{e} de
Bretagne Occidentale,}\\{\small 29285 Brest C\'{e}dex 3, France \&}\\{\small ~School of Mathematics and Statistics, University of Sydney, Sydney, NSW 2006, Australia}}
\title{Fractional backward stochastic differential equations and fractional backward
variational inequalities}
\maketitle
\date{}

\begin{abstract}
In the framework of fractional stochastic calculus, we study the existence
and the uniqueness of the solution for a backward stochastic differential equation, formally written as:
\[
\left\{
\begin{array}
[c]{l}
-dY(t)= f(t,\eta(t),Y(t),Z(t))dt-Z(t)\delta B^{H}\left(
t\right)  ,\quad t\in[0,T],\\
Y(T)=\xi,
\end{array}
\right.
\]
where $\eta$ is a stochastic process given by $\eta(t)=\eta(0) +\int_{0}^{t}\sigma(s) \delta B^{H}(s)$, $t\in[0,T]$, and $B^{H}$ is a fractional
Brownian motion with Hurst parameter greater than $1/2$. The stochastic integral used in above equation is the divergence-type integral.
Based on Hu and Peng's paper, \textit{BDSEs driven by fBm}, SIAM J Control Optim. (2009), we develop a rigorous approach for this equation. Moreover, we study the existence of the solution for the multivalued backward stochastic differential equation
\[
\left\{
\begin{array}
[c]{l}
-dY(t)+\partial\varphi(Y(t))dt\ni f(t,\eta(t),Y(t),Z(t))dt-Z(t)\delta B^{H}\left(
t\right)  ,\quad t\in[0,T],\\
Y(T)=\xi,
\end{array}
\right.
\]
where $\partial\varphi$ is a multivalued operator of subdifferential type associated with the convex function $\varphi$.

\end{abstract}

\footnotetext[1]{{\scriptsize Corresponding author.}}

\renewcommand{\thefootnote}{\arabic{footnote}} \footnotetext[1]%
{{\scriptsize The work of this author was carried out at Faculty of
Mathematics, \textquotedblleft Alexandru Ioan Cuza\textquotedblright%
\ University of Ia\c{s}i under project POSDRU/89/1.5/S/49944.}}
\footnotetext[2]{{\scriptsize The work of this author was supported by the
Marie Curie ITN Project, \textquotedblleft Controlled
Systems\textquotedblright, no. 213841.}}
\renewcommand{\thefootnote}{\fnsymbol{footnote}}
\footnotetext{\textit{{\scriptsize E-mail addresses:}}
{\scriptsize lucian.maticiuc@ymail.com (Lucian\ Maticiuc), nietianyang@163.com
(Tianyang\ Nie)}}

\medskip


\section{Introduction}
General backward stochastic differential equations (BSDEs) driven by a Brownian motion were first studied by Pardoux and Peng \cite{PP-92},
where they also gave a probabilistic interpretation for the viscosity solution of semilinear partial differential equations (PDEs).
Pardoux and R\u{a}\c{s}canu \cite{PR-98} studied backward stochastic differential equations involving a subdifferential
operator [which are often called backward stochastic variational inequalities (BSVIs)],
and they used them in order to generalize the Feymann$-$Kac type formula to represent the solution of multivalued parabolic PDEs
[also called parabolic variational inequalities (PVIs)].

Backward stochastic differential equations driven by a fractional Brownian motion  with Hurst parameter $H\in(1/2,1)$ were first considered by Biagini, Hu, {\O}ksendal and Sulem \cite{BHOS-02}, when they studied the stochastic maximal principle in the framework of a fractional Brownian motion. By adapting the four-step scheme introduced by Ma et al. \cite{MPY-94} and the so-called $S$-transform, Bender \cite{B-05} studied BSDEs driven by a fractional Brownian motion  with Hurst parameter $H\in(0,1)$. Indeed, through a backward parabolic PDE, he constructed an explicit solution of a kind of linear fractional BSDE. Hu and Peng \cite{HP-09} were the first to study nonlinear BSDEs governed by a fractional Brownian motion (fBm). Our work, based on \cite{HP-09}, has the objective to develop a rigorous approach for such BSDEs driven by a fBm and to extend the discussion to fractional BSVIs. Our paper is, to our best knowledge, the first one to study fractional BSVIs.

Let us recall that, for $H\in(0,1)$, a (one-dimensional) fBm
$(B^{H}\left(  t\right)  )_{t\geq0}$ with Hurst parameter $H$ is a continuous and
centered Gaussian process with covariance
\[
\mathbb{E}\left[  B^{H}(t)B^{H}(s)\right]=\frac{1}{2}(t^{2H}+s^{2H}-|t-s|^{2H}),\quad t,s\ge0.
\]
For $H=1/2$, the fBm is a standard Brownian motion. If $H>1/2$, then $B^{H}\left(  t\right)  $ has a long-range dependence,
which means that for $r(n):=\mathrm{cov}(B^{H}\left(  1\right)  ,B^{H}\left(  n+1\right)-B^{H}\left(  n\right)  )$,
we have $\sum_{n=1}^{\infty}r(n)=\infty$.
Moreover, $B^{H}$ is self-similar, i.e., $B^{H}\left(  at\right)  $ has the same law as $a^{H}B^{H}\left(  t\right)  $ for any $a>0$.
Since there are many models of physical phenomena and finance which exploit the self-similarity and
the long-range dependence, fBms are a very useful tool to characterize such type of problems.

However, since fBms are not semimartingales nor
Markov processes when $H\neq1/2$, we cannot use the classical theory of stochastic calculus to define the fractional stochastic integral.
In essence, two different integration theories with respect to fractional Brownian motion have been defined and studied.
The first one, originally due to Young \cite{Y-36}, concerns the pathwise (with $\omega$ as a parameter) Riemann$-$Stieltjes integral,
which exists if the integrand has H\"{o}lder continuous paths of order $\alpha>1-H$. But
it turns out that this integral has the properties comparable to the Stratonovich integral,
which leads to difficulties in applications.

The second one concerns the divergence operator (or Skorohod integral),
defined as the adjoint of the derivative operator in the framework of the Malliavin calculus.
This approach was introduced by Decreusefond and \"{U}st\"{u}nel \cite{DU-98},
and it was very intensely studied, e.g., in Al\`{o}s and Nualart \cite{AN-03} for $H>1/2$,
and in Al\`{o}s et al. \cite{AMN1-00} for $H<1/2$ as well as in Al\`{o}s et al. \cite{AMN-01} for Gaussian processes.

An equivalent approach consists in defining, for $H\in(1/2,1)$,
the stochastic integral based on Wick product (introduced by Duncan et al. \cite{DHP-00}), as the limit of Riemann sums.
We mention that in contrast to the pathwise integral, the expectation of this integral is zero
for a large class of integrands.

Concerning the study of BSDEs in the fractional framework, the major problem is
the absence of a martingale representation type theorem with respect to a fBm.
For the first time, Hu and Peng \cite{HP-09} overcome this problem, in the case $H>1/2$. For this, they used the notion of
quasi-conditional expectation $\hat{E}$ (introduced in Hu and {\O }ksendal
\cite{HO-03}).  In our paper, we consider the BSDE
\begin{equation}\label{BSDE in introduction}
\left\{
\begin{array}
[c]{l}%
-dY(t)= f(t,\eta(t),Y(t),Z(t))dt-Z(t)\delta B^{H}\left(
t\right)  ,\quad t\in[0,T],\\
Y(T)=g(\eta(T)),
\end{array}
\right.
\end{equation}
driven by a fBm $B^{H}$ and governed by the process $\eta(t)=\eta(0) +\int_{0}^{t}\sigma(s) \delta B^{H}(s)$, $t\in[0,T]$,
where $\sigma:[0,T]\rightarrow\mathbb{R}$ is deterministic, continuous function.

A special care has to be paid here to the stochastic integral in the BSDE (\ref{BSDE in introduction}). In \cite{HP-09} this stochastic integral is the Wick product one, but the It\^{o} formula and the integration by part formula they used were established for the It\^{o}$-$Skorohod type integral (see Definition 6.11  \cite{H-05} ).  In our approach, we use as stochastic integral the divergence operator.

Concerning the coefficient $\sigma$ of the driving process $\eta$, Hu and Peng \cite{HP-09} supposed that
\[
\text{ there exists } c_{0}>0\text{ such that }\inf_{t\in\left[  0,T\right]  }\dfrac
{\hat{\sigma}(t)}{\sigma(t)}\geq c_{0},
\]
for $\hat{\sigma}(t):={\int_{0}^{t}}\phi(t-r)\sigma(r)dr,~t\in[0,T]$.
Here, in our manuscript, we work without such a condition. Let us mention that in \cite{HP-09},
it is assumed that
$\eta(t)=\eta(0)+\int_{0}^{t}b(s)ds+\int_{0}^{t}\sigma(s)\delta B^{H}(s),~t\in[0,T]$.
In our paper, we adopt the form $\eta(t)=\eta(0)+\int_{0}^{t}\sigma(s)\delta B^{H}(s),~t\in[0,T]$, since we will use the proof of Theorem 3.8 from \cite{HP-09}, especially the quasi-conditional expectation formula
\[
\hat{E}[f(\eta(T))|\mathcal{F}_{t}]=P_{\|\sigma\|_{T}^{2}-\|\sigma\|_{t}^{2}}f(\eta(t)).
\]

Based on the above-described framework, we prove the existence and the uniqueness for BSDE (\ref{BSDE in introduction}). This approach includes, in particular, first a discussion of the equation
\[
Y(t)=g(\eta(T))+{\displaystyle\int_{t}^{T}}f\left(s,\eta(s)\right)ds-{\displaystyle\int_{t}^{T}}Z(s)\delta B^{H}(s),\qquad t\in[0,T].
\]
After, the existence for BSDE (\ref{BSDE in introduction}) is proved by using a fixed point theorem over an appropriate Banach space.

Based on our results on BSDE driven by a fBm and on Pardoux and R\u{a}\c{s}canu \cite{PR-98} on BSVI governed by a standard Brownian motion, we consider the following fractional BSVI
\[
\left\{
\begin{array}
[c]{l}%
-dY(t)+\partial\varphi(Y(t))dt\ni f(t,\eta(t),Y(t),Z(t))dt-Z(t)\delta B^{H}\left(
t\right)  ,\;\;t\in[0,T],\\
Y(T)=g(\eta(T)),
\end{array}
\right.
\]
where $\partial\varphi$ is the subdifferential of a convex lower semicontinuous (l.s.c.) function $\varphi:\mathbb{R}\rightarrow(-\infty,+\infty]$.
The existence of the solution will be proved.

Now, we give the outline of our paper: In Section 2 we recall some
definitions and results about fractional stochastic integrals and the related It\^{o} formula.
We present the assumptions and some auxiliary results including the It\^{o} formula w.r.t. the divergence-type integral in Section 3.
Section 4 is devoted to  prove the existence and the uniqueness result for BSDE driven by a fBm. In Section 5, we study the existence for fractional BSVI governed by a fBm. Finally, in the Appendix, we prove a more general It\^{o} formula based on Theorem 8 \cite{AN-03} and an auxiliary lemma.
\section{Preliminaries: Fractional stochastic calculus}

In this section, we shall recall some important definitions and results
concerning the Malliavin calculus, the stochastic integral with respect to a fBm, and
It\^{o}'s formula. For a deeper discussion, we refer the reader to \cite{AMN-01,AN-03,BHOZ-06,DHP-00,H-05} and \cite{N-95}.

Throughout our paper, we assume that the Hurst parameter $H$ always satisfies $H>1/2$.
Define
\[
\phi(x)=H(2H-1)|x|^{2H-2},\qquad x\in\mathbb{R}.
\]
Let us denote by $|\mathcal{H}|$ the Banach space of measurable functions $f:[0,T]\rightarrow\mathbb{R}$ such that
\[
\Vert f\Vert_{|\mathcal{H}|}^{2}:=\int_{0}^{T}\int_{0}^{T}\phi(u-v)|f(u)||f(v)|dudv<+\infty.
\]
Given $\xi,\eta\in|\mathcal{H}|$, we put
\begin{equation}\label{inner product}
\langle\xi,\eta\rangle_{T}=\int_{0}^{T}\int_{0}^{T}\phi(u-v)\xi\left(
u\right)  \eta\left(  v\right)  dudv\text{ \ and\ }\Vert\xi\Vert_{T}%
^{2}:=\langle\xi,\xi\rangle_{T}.
\end{equation}
Then $\langle\xi,\eta\rangle_{T}$ is a Hilbert scalar product.
Let $\mathcal{H}$ be the completion of the space of step functions in $|\mathcal{H}|$ under this scalar product.
We emphasize that the elements of $\mathcal{H}$ can be distributions.
Moreover, from \cite{MMV-01} we have the continuous embedding $L^{2}([0,T])\subset L^{\frac{1}{H}}([0,T])\subset|\mathcal{H}|\subset\mathcal{H}$.

Let $B^{H}$ be a fractional Brownian motion defined on a complete probability space $(\Omega,\mathcal{F},P)$, and that $\mathcal{F}$ is generated by $B^{H}$. Let $\mathcal{P}_{T}$ be the set of elementary random variables of the form
\[
F=f\left(  \int_{0}^{T}\xi_{1}(t)dB^{H}\left(
t\right)  ,\ldots,\int_{0}^{T}\xi_{n}(t)dB^{H}\left(  t\right)  \right)  ,
\]
where $f$ is a polynomial function of $n$ variables and $\xi_{1},\xi_{2},\ldots,\xi_{n}\in\mathcal{H}$.
The Malliavin derivative $D^{H}$ of an elementary variable $F\in\mathcal{P}_{T}$ is defined by
\[
D_{s}^{H}F=\sum_{i=1}^{n}\dfrac{\partial f}{\partial x_{i}}\left(\int_{0}^{T}\xi_{1}(t)dB^{H}(t),\ldots,\int_{0}^{T}\xi_{n}(t)dB^{H}(t)\right)  \xi_{i}(s),\;s\in\lbrack0,T].
\]
We denote by $\mathbb{D}_{1,2}$ the Banach space defined as the completion of $\mathcal{P}_{T}$ w.r.t. the following norm
\[
\Vert F\Vert_{1,2}=\left[  \mathbb{E}(\vert F\vert ^{2})+\mathbb{E}\left(  \Vert D_{s}^{H}F\Vert_{T}^{2}\right)\right]^{1/2},~ F\in \mathcal{P}_{T}.
\]
Hence, $\mathbb{D}_{1,2}$ consists of all $F\in L^{2}(\Omega,\mathcal{F},P)$ such that
there exists a sequence $F_{n}\in\mathcal{P}_{T}$, $n\ge 1$, which satisfies
\[
\begin{array}
[c]{l}
F_{n}\longrightarrow F\text{ in }L^{2}(\Omega,\mathcal{F},P),\medskip\\
\left(D^{H}F_{n}\right)_{n\ge1}\text{ is convergent in }L^{2}(\Omega,\mathcal{F},P;\mathcal{H}).
\end{array}
\]
Moreover, from Proposition 1.2.1 \cite{N-95} we have that $D^{H}=\left(D_{s}^{H}\right)_{s\in[0,T]}$ is a
closable operator from $L^{2}(\Omega,\mathcal{F},P)$ to $L^{2}(\Omega,\mathcal{F},P;\mathcal{H})$.
Thus, $D^{H}F=\lim\limits_{n\rightarrow\infty}D^{H}G_{n}$ in $L^{2}(\Omega,\mathcal{F},P;\mathcal{H})$,
for every sequence $G_{n}\in\mathcal{P}_{T},~n\ge1$,
which satisfies
\[
\begin{array}
[c]{l}
G_{n}\longrightarrow F\text{ in } L^{2}(\Omega,\mathcal{F},P),\medskip\\
\left(D^{H}G_{n}\right)_{n\ge1}\text{ is convergent in }L^{2}(\Omega,\mathcal{F},P;\mathcal{H}).
\end{array}
\]
Let us introduce also another derivative
\begin{equation}\label{another derivative}
\mathbb{D}_{t}^{H}F=\int_{0}^{T}\phi(t-v)D_{v}^{H}Fdv,\quad t\in[0,T].
\end{equation}

We also need the adjoint operator of the derivative $D^{H}$. This operator
is called divergence operator, and it is denoted by $\delta(\cdot)$ and represents the
divergence-type integral with respect to a fBm (see, e.g., \cite{AMN-01},  \cite{AN-03} and \cite{BHOZ-06} for
more details).

\begin{definition}\label{divergence operator}
We say that a process $u\in L^{2}(\Omega,\mathcal{F},P;\mathcal{H})$
belongs to the domain $Dom(\delta)$, if there exists $\delta(u)\in L^{2}(\Omega,\mathcal{F},P)$, such that the following duality relationship is satisfied
\begin{equation}\label{equation for divergence operator}
\mathbb{E}(F\delta(u))=\mathbb{E}(\langle D_{\cdot}^{H}F,u\rangle_{T}%
),\quad\text{ for all } F\in\mathcal{P}_{T}.
\end{equation}
\end{definition}
\begin{remark}
In (\ref{equation for divergence operator}), the class $\mathcal{P}_{T}$ can be replaced by $\mathbb{D}_{1,2}$ (see \cite{BHOZ-06} Definition 2.2.2 and 2.2.3). If $u\in Dom(\delta)$, $\delta(u)$ is unique, and we define the divergence-type integral of $u\in Dom(\delta)$ w.r.t. fBm $B^{H}$ by putting $\int_{0}^{T}u(s)\delta B^{H}(s):=\delta(u)$.
\end{remark}

Let us recall a result about a sufficient condition for the existence of the divergence-type integral.
For this, we use the It\^{o}$-$Skorohod type stochastic integral introduced in Definition 6.11 \cite{H-05}, which is defined in the spirit of the anticipative Skorohod integral w.r.t. Brownian motion in \cite{NP-88}.

\begin{theorem}\label{Th.1_Hu 2005}
[Proposition 6.25, \cite{H-05}] We denote by $\mathbb{L}^{1,2}_{H}$ the space of  all stochastic processes $u:\left(  \Omega,\mathcal{F},P\right)\rightarrow \mathcal{H}$ such that
\begin{equation}\label{isometry equation}
\mathbb{E}\left(  \Vert u\Vert_{T}^{2}+\int_{0}^{T}\int_{0}^{T}|\mathbb{D}
_{s}^{H}u\left(  t\right)  |^{2}dsdt\right)  <\infty.
\end{equation}
If $u\in \mathbb{L}^{1,2}_{H}$,  then the It\^{o}$-$Skorohod type stochastic integral $\int_{0}^{T}u\left(  s\right)
dB^{H}\left(  s\right)  $ defined by Definition
6.11 \cite{H-05} exists and coincides with the divergence-type integral (see Theorem 6.23 \cite{H-05}). Moreover,
\[
\left\{
\begin{array}
[c]{l}
\mathbb{E}\left[  {\displaystyle\int_{0}^{T}}u\left(  s\right)  dB^{H}\left(
s\right)  \right]  =0,\medskip\\
\mathbb{E}\left[  {\displaystyle\int_{0}^{T}}u\left(  s\right)  dB^{H}\left(
s\right)  \right]  ^{2}=\mathbb{E}\left(  \Vert u\Vert_{T}^{2}
+{\displaystyle\int_{0}^{T}}{\displaystyle\int_{0}^{T}}\mathbb{D}_{s}
^{H}u(t)\mathbb{D}_{t}^{H}u(s)dsdt\right)  .
\end{array}
\right.
\]
\end{theorem}

Let us finish this section by giving an It\^{o} formula for the divergence-type integral. Due to Theorem 8 \cite{AN-03},
the following It\^{o} formula holds.
\begin{theorem}\label{Ito formula for the divergence integral}
Let $\psi$ be a function of class $C^{2}(\mathbb{R})$. Assume that the process $(u_{t})_{t\in[0,T]}$ belongs to $\mathbb{D}^{2,2}_{loc}(|\mathcal{H}|)$ and that the integral $X_{t}=\int_{0}^{t}u_{s}\delta B^{H}(s)$ is almost surely continuous. Assume that $\left(\mathbb{E}|u|^{2}\right)^{1/2}$ belong to $\mathcal{H}$. Then, for each $t\in[0,T]$, the following formula holds
\[
\begin{array}
[c]{l}
\psi(X_{t})=\psi(0)+{\displaystyle\int_{0}^{t}}{\displaystyle\dfrac{\partial }{\partial x}}\psi(X_{s})u_{s}\delta B^{H}(s)\medskip\\
\qquad\qquad+H(2H-1){\displaystyle\int_{0}^{t}}{\displaystyle\dfrac{\partial^{2}}{\partial x^{2}}}\psi(X_{s})u_{s}{\displaystyle\int_{0}^{T}}
|s-r|^{2H-2}\left({\displaystyle\int_{0}^{s}}D_{r}u_{\theta}\delta B^{H}(\theta)dr\right)ds\medskip\\
\qquad\qquad+H(2H-1){\displaystyle\int_{0}^{t}}{\displaystyle\dfrac{\partial^{2}}{\partial x^{2}}}\psi(X_{s})u_{s}\left({\displaystyle\int_{0}^{s}}
u_{\theta}|s-\theta|^{2H-2}d\theta \right)ds.
\end{array}
\]
\end{theorem}
\begin{remark}
The above theorem can be generalized, see Theorem \ref{general Ito formula for the divergence integral} in the Appendix. In particular, Corollary
\ref{particular for general ito formula} tells us:

Let $f:[0,T]\rightarrow\mathbb{R}$ and $g:[0,T]\rightarrow\mathbb{R}$ be deterministic continuous  functions.
If
\[
X_{t}=X_{0}+\int_{0}^{t}g_{s}ds+\int_{0}^{t}f_{s}\delta B^{H}(s),\;t\in[0,T],
\]
and $\psi\in C^{1,2}([0,T]\times\mathbb{R})$, we have
\begin{equation}\label{deterministic Ito's formula}
\begin{array}
[c]{l}
\psi(t,X_{t})=\psi(0,X_{0})+{\displaystyle\int_{0}^{t}}\dfrac{\partial}{\partial s} \psi(s,X_{s})ds
+{\displaystyle\int_{0}^{t}}\dfrac{\partial }{\partial x}\psi(s,X_{s})dX_{s}\medskip\\
\qquad\qquad\qquad
+\dfrac{1}{2}{\displaystyle\int_{0}^{t}}\dfrac{\partial^{2}}{\partial x^{2}}\psi(s,X_{s})~\left( \dfrac{d}{ds} \Vert f\Vert_{s}^{2}\right)ds,\;t\in[0,T].
\end{array}
\end{equation}
\end{remark}

\section{Assumptions and auxiliary results}
\subsection{Assumptions}
Let us consider the It\^{o}-type process
\begin{equation}\label{definition of eta}
\eta(t)=\eta(0)+\int_{0}^{t}\sigma(s)\delta B^{H}\left(
s\right)  ,~t\in[0,T],
\end{equation}
where the coefficients $\eta(0)$ and $\sigma$ satisfy:
\begin{itemize}
\item[$(H_{1})$] $\eta(0)\in\mathbb{R}$ is a given constant;
\item[$(H_{2})$] $\sigma:\mathbb{R}\rightarrow\mathbb{R}$ is a deterministic continuous function such that $\sigma(t)\neq0$, for all $t\in [0,T]$.
\end{itemize}
Let
\begin{equation} \label{sigma tilde_hat}
\hat{\sigma}(t):={\int_{0}^{t}}\phi(t-r)\sigma(r)dr,~t\in[0,T].
\end{equation}
We recall that (see (\ref{inner product}))
\[
\Vert\sigma\Vert_{t}^{2}=H(2H-1)\int_{0}^{t}\int_{0}^{t}|u-v|^{2H-2}%
\sigma\left(  u\right)  \sigma\left(  v\right)  dudv.
\]

\begin{remark}\label{remark for sigma}
The function $\hat{\sigma}$ defined by (\ref{sigma tilde_hat}) can be written in the following form:
\[
\hat{\sigma}(t)=H(2H-1)t^{2H-1}
{\displaystyle\int_{0}^{1}}
\left(  1-u\right)  ^{2H-2}\sigma\left(  tu\right)  du, ~t\in[0,T].
\]
Moreover, we observe that $\Vert\sigma\Vert_{t}^{2}$ is continuously differentiable with respect to $t$, and
\begin{equation}\label{properties for sigma}
\begin{array}
[c]{ll}
(a)  & \dfrac{d}{dt}\left(  \Vert\sigma\Vert_{t}^{2}\right)
=2\sigma\left(  t\right)  \hat{\sigma}(t)>0,\; t\in
(0,T],\medskip\\
(b)  & \text{ for a suitable constant } M>0, \dfrac{1}{M}t^{2H-1}%
\leq\dfrac{\hat{\sigma}(t)}{\sigma(t)}\leq Mt^{2H-1},\; t\in\left[
0,T\right]  .
\end{array}
\end{equation}
\end{remark}
\smallskip

Our objective is to study the following BSDE driven by the fBm $B^{H}$ and the above introduced stochastic process $\eta$:
\[
\left\{
\begin{array}
[c]{l}%
-dY(t)= f(t,\eta(t),Y(t),Z(t))dt-Z(t)\delta B^{H}\left(
t\right),~t\in[0,T],\medskip\\
Y(T)=\xi.
\end{array}
\right.
\]
Here, the stochastic integral is understood as the divergence operator.
We make the following assumptions on the function $f$ and the terminal condition $\xi$:
\begin{itemize}
\item[$(H_{3})$] The function $f:[0,T]\times\mathbb{R}^{3}\longrightarrow\mathbb{R}$ belongs to the space $C_{pol}^{0,1}\left(
[0,T]\times\mathbb{R}^{3}\right)^{\ast} $, and there
exists a constant $L$ such that, for all $t\in[0,T]$, $x,y_{1},y_{2}
,z_{1},z_{2}\in\mathbb{R}$,
\[
\begin{array}
[c]{rl}
\left\vert f(t,x,y_{1},z_{1})-f(t,x,y_{2},z_{2})\right\vert \leq
L(\left\vert y_{1}-y_{2}\right\vert
+\left\vert z_{1}-z_{2}\right\vert ).
\end{array}
\]
\item[$\left(  H_{4}\right)  $] $\xi=g(\eta_{T})$, where $g:\mathbb{R\rightarrow\mathbb{R}}$ is a differentiable function with polynomial growth.
\end{itemize}
\footnotetext[1]{{\scriptsize $C^{k,l}_{pol}([0,T]\times\mathbb{R}^{m})$} is the space of all $C^{k,l}$-functions over $[0,T]\times\mathbb{R}^{m}$, which together with their derivatives, are of polynomial growth.}
\smallskip

Before giving the definition of the solution for the above BSDE and investigating its wellposedness (see Section 4),
we introduce, motivated by \cite{HP-09}, the following space
\[
\mathcal{V}_{T}:=\left\{  Y=\phi(\cdot,\eta(\cdot)):\phi\in C_{pol}^{1,3}(
[0,T]\times\mathbb{R}) \text{ with } \dfrac{\partial\phi}{\partial t}\in C_{pol}^{0,1}(
[0,T]\times\mathbb{R})\right\}
\]
as well as its completion $\mathcal{\bar{V}}_{T}^{\alpha}$ under the following $\alpha$-norm
\begin{equation}\label{def V^alpha}
\Vert Y\Vert_{\alpha}=\left(  \int_{0}^{T}t^{2\alpha-1}\mathbb{E}|Y\left(
t\right)  |^{2}dt\right)  ^{1/2}=\left(  \int_{0}^{T}t^{2\alpha-1}%
\mathbb{E}|\phi(t,\eta\left(  t\right)  )|^{2}dt\right)  ^{1/2},
\end{equation}
where $\alpha\geq1/2$. Let us study some auxiliary results concerning these spaces.

\subsection{An It\^{o} formula}
We begin with the following result concerning the space $\mathcal{V}_{T}$.
\begin{lemma}\label{property of VT}
We have $\mathcal{V}_{T}\subset\mathbb{L}^{1,2}_{H}\subset Dom(\delta)$.
\end{lemma}
\begin{proof}
Let $u\in \mathcal{V}_{T}$. In order to show (\ref{isometry equation}), we first prove that $E\|u\|^{2}_{T}<\infty$. From  $L^{2}([0,T])\subset\mathcal{H}$ we see that it is sufficient to show that $\mathbb{E}{\displaystyle\int_{0}^{T}}|u(s,\eta(s))|^{2}ds<\infty$, where the latter property can be deduced from $\mathbb{E}|u(s,\eta(s))|^{2}\leq C$, $s\in[0,T]$, for some suitable $C\in \mathbb{R}$. Indeed, since $u\in \mathcal{V}_{T}$, we have, for some $C>0,~k\ge1$,
\[
\mathbb{E}|u(s,\eta(s))|^{2}\leq C\mathbb{E}(1+|\eta(s)|+\cdots+|\eta(s)|^{k}), ~s\in[0,T].
\]
On the other hand, from (\ref{definition of eta}) and Theorem 7.10 \cite{H-05}, we see that for any $p\ge1$, there exists $C_{p}>0$ such that
\begin{equation}\label{Lp estimates}
\mathbb{E}|\eta(s)|^{p}\leq C\mathbb{E}\left(1+\left|{\displaystyle\int_{0}^{s}}\sigma(v)\delta B^{H}(v)\right|^{p}\right)
\leq C_{p}+C_{p}\|\sigma\|_{s}^{p/2}.
\end{equation}
Hence, $E\|u\|^{2}_{T}<\infty$.

In a second step, we show that $\mathbb{D}^{H}_{s}u(t,\eta(t))$ exists for all $s,t\in[0,T]$ and
\begin{equation}\label{boundedness of DH operator}
\mathbb{E}{\displaystyle\int_{0}^{T}}{\displaystyle\int_{0}^{T}}|\mathbb{D}_{s}^{H}u(t,\eta(t))|^{2}dsdt<\infty.
\end{equation}
In fact, a straightforward computation shows that
\[
\mathbb{D}^{H}_{s}u(t,\eta(t))=\dfrac{\partial}{\partial x}u(t,\eta(t)){\displaystyle\int_{0}^{t}}\phi(s-v)\sigma(v)dv.
\]
From the polynomial growth of $\dfrac{\partial u}{\partial x}$ and the continuity of $\sigma$, we conclude
\[
\mathbb{E}|\mathbb{D}^{H}_{s}u(t,\eta(t))|^{2}\leq C,\quad s,t\in[0,T], \text{ for a suitable } C\in\mathbb{R},
\]
which yields (\ref{boundedness of DH operator}).
Consequently, the process $u$ satisfies (\ref{isometry equation}) and belongs to $\mathbb{L}^{1,2}_{H}$.\hfill
\end{proof}
\smallskip

Let us give now a statement for the It\^{o} formula in the framework of the divergence-type integral, which is for our purposes better adapted than Theorem 4.5 \cite{DHP-00}. We mention that the formula in the following theorem is a particular case of our generalized It\^{o} formula (\ref{1 general Ito's formula for the divergence integral}) (see Theorem \ref{general Ito formula for the divergence integral} in the Appendix), but here, we use a different approach.
\begin{theorem}\label{Ito formula for divergence type integral}
Let $u\in \mathcal{V}_{T}$ and $f\in C_{pol}^{0,1}([0,T]\times\mathbb{R})$. We put $f_{s}=f(s,\eta(s)), ~s\in[0,T]$. Then for the It\^{o} process
\[
X_{t}=X_{0}  +\int_{0}^{t}f_{s}ds+\int_{0}^{t}u_{s}\delta B^{H}(s),~t\in[0,T],
\]
we have
\[
\begin{array}
[c]{l}
(i)\quad uXI_{[0,t]}\in Dom(\delta),~t\in[0,T],\medskip\\
(ii)\quad X_{s}\in \mathbb{D}_{1,2},~s\in[0,T],\medskip\\
(iii)\quad \left(\mathbb{D}^{H}_{s}X_{s}\right)_{s\in[0,T]}\in L^{2}([0,T]\times\Omega)
\end{array}
\]
and
\[
\begin{array}
[c]{l}
X_{t}^{2}=X_{0}^{2} +2{\displaystyle\int_{0}^{t}}X_{s}f_{s}ds+2{\displaystyle\int_{0}^{t}}X_{s}u_{s}\delta B^{H}(s)+2{\displaystyle\int_{0}^{t}}u_{s}\mathbb{D}^{H}_{s}X_{s}ds  , ~a.s. \qquad t\in[0,T].
\end{array}
\]
\end{theorem}

Before proving this theorem, we give the following lemma.
\begin{lemma}\label{square ito formula}
Let $u\in\mathcal{V}_{T}$ and $X_{t}={\displaystyle\int_{0}^{t}}u_{s}\delta B^{H}(s), ~t\in[0,T]$. Then $uXI_{[0,t]}\in Dom(\delta)$, $t\in[0,T]$, $\left(\mathbb{D}^{H}_{s}X_{s}\right)_{s\in[0,T]}\in L^{2}([0,T]\times\Omega)$, and
\[X_{t}^{2}=2{\displaystyle\int_{0}^{t}}X_{s}u_{s}\delta B^{H}(s)+2{\displaystyle\int_{0}^{t}}u_{s}\mathbb{D}^{H}_{s}X_{s}ds  , ~t\in[0,T].
\]
\end{lemma}
\begin{proof}
Let $F\in\mathcal{P}_{T}$. Then, since obviously $X_{t}F\in\mathbb{D}_{1,2}$, we have from Definition \ref{divergence operator}
\[
\begin{array}
[c]{l}
\mathbb{E}\left[X_{t}^{2}F\right]=\mathbb{E}\left[X_{t}(X_{t}F)\right]=\mathbb{E}\left[{\displaystyle\int_{0}^{T}}u_{s}I_{[0,t](s)}\delta B^{H}(s)(X_{t}F)\right]\medskip\\
=\mathbb{E}\left[{\displaystyle\int_{0}^{T}}u_{s}I_{[0,t](s)}\mathbb{D}^{H}_{s}(X_{t}F)ds\right]
=\mathbb{E}\left[\left({\displaystyle\int_{0}^{t}}u_{s}\mathbb{D}^{H}_{s}Fds\right)X_{t}\right]
+\mathbb{E}\left[\left({\displaystyle\int_{0}^{t}}u_{s}\mathbb{D}^{H}_{s}X_{t}ds\right)F\right].
\end{array}
\]
Here, we have used $X\in\mathbb{L}^{1,2}_{H}$  and, hence also $XF\in\mathbb{L}^{1,2}_{H}$. In particular, we observe that
\begin{equation}\label{Malliavin derivative of X}
\mathbb{D}^{H}_{s}X_{t}={\displaystyle\int_{0}^{t}}\phi(s-r)u_{r}dr+{\displaystyle\int_{0}^{t}}\mathbb{D}^{H}_{s}u_{r}\delta B^{H}(r),
~s,t\in[0,T].
\end{equation}
Moreover, since ${\displaystyle\int_{0}^{t}}u_{s}\mathbb{D}^{H}_{s}Fds\in\mathbb{D}_{1,2}$, we get again from Definition \ref{divergence operator}
\[
\begin{array}
[c]{l}
\mathbb{E}\left[\left({\displaystyle\int_{0}^{t}}u_{s}\mathbb{D}^{H}_{s}Fds\right)X_{t}\right]
=\mathbb{E}\left[\left({\displaystyle\int_{0}^{t}}u_{s}\mathbb{D}^{H}_{s}Fds\right){\displaystyle\int_{0}^{t}}u_{s}\delta B^{H}(s)\right]\medskip\\
\qquad\qquad\qquad\qquad\qquad\quad
=\mathbb{E}\left[{\displaystyle\int_{0}^{t}}{\displaystyle\int_{0}^{t}}\mathbb{D}^{H}_{r}\left(u_{s}\mathbb{D}^{H}_{s}F\right)u_{r}drds\right].
\end{array}
\]
On the other hand, using (\ref{Malliavin derivative of X}) it follows
\[
\mathbb{E}\left[\left({\displaystyle\int_{0}^{t}}u_{s}\mathbb{D}^{H}_{s}X_{t}ds\right)F\right]=
\mathbb{E}\left[{\displaystyle\int_{0}^{t}}{\displaystyle\int_{0}^{t}}\phi(s-r)u_{s}u_{r}dsdr\cdot F\right]+
{\displaystyle\int_{0}^{t}}\mathbb{E}\left[{\displaystyle\int_{0}^{t}}\mathbb{D}^{H}_{r}(u_{s}F)\mathbb{D}^{H}_{s}u_{r}dr\right]ds.
\]
Therefore, by combining the above relations, we obtain
\[
\begin{array}
[c]{l}
\mathbb{E}\left[X_{t}^{2}F\right]=\mathbb{E}\left[\left({\displaystyle\int_{0}^{t}}u_{s}\mathbb{D}^{H}_{s}Fds\right)X_{t}\right]
+\mathbb{E}\left[\left({\displaystyle\int_{0}^{t}}u_{s}\mathbb{D}^{H}_{s}X_{t}ds\right)F\right]\medskip\\
=\mathbb{E}\left[{\displaystyle\int_{0}^{t}}{\displaystyle\int_{0}^{t}}\mathbb{D}^{H}_{r}\left(u_{s}\mathbb{D}^{H}_{s}F\right)u_{r}drds\right]
+\mathbb{E}\left[{\displaystyle\int_{0}^{t}}{\displaystyle\int_{0}^{t}}\phi(s-r)u_{s}u_{r}dsdr\cdot F\right]\medskip\\
\qquad\qquad\qquad\qquad\qquad
+\mathbb{E}\left[{\displaystyle\int_{0}^{t}}{\displaystyle\int_{0}^{t}}\mathbb{D}^{H}_{r}(u_{s}F)\mathbb{D}^{H}_{s}u_{r}drds\right].
\end{array}
\]
By noticing that the right-hand side of the above equality is symmetric in $(s,r)$ we deduce
\begin{equation}\label{I1I2I3}
\begin{array}
[c]{l}
\mathbb{E}\left[X_{t}^{2}F\right]
=2\mathbb{E}\left[{\displaystyle\int_{0}^{t}}{\displaystyle\int_{0}^{s}}\mathbb{D}^{H}_{r}\left(u_{s}\mathbb{D}^{H}_{s}F\right)u_{r}drds\right]
+2\mathbb{E}\left[{\displaystyle\int_{0}^{t}}{\displaystyle\int_{0}^{s}}\phi(s-r)u_{s}u_{r}dsdr\cdot F\right]\medskip\\
\qquad\qquad\qquad\qquad\qquad
+2\mathbb{E}\left[{\displaystyle\int_{0}^{t}}{\displaystyle\int_{0}^{s}}\mathbb{D}^{H}_{r}(u_{s}F)\mathbb{D}^{H}_{s}u_{r}drds\right]\medskip\\
\qquad\qquad:=2I_{1}+2I_{2}+2I_{3}.
\end{array}
\end{equation}
Let us begin with the evaluation of $I_{1}$.
Obviously, by using that $uI_{[0,s]}\in\mathbb{L}^{1,2}_{H}\subset Dom(\delta)$ and $u_{s}\mathbb{D}_{s}^{H}F\in\mathbb{D}_{1,2}$,
we have from Fubini's Theorem and Definition \ref{divergence operator}
\begin{equation}\label{I1}
\begin{array}
[c]{l}
I_{1}=\mathbb{E}\left[{\displaystyle\int_{0}^{t}}{\displaystyle\int_{0}^{s}}\mathbb{D}^{H}_{r}\left(u_{s}\mathbb{D}^{H}_{s}F\right)u_{r}drds\right]
={\displaystyle\int_{0}^{t}}\mathbb{E}
\left[{\displaystyle\int_{0}^{T}}\mathbb{D}^{H}_{r}\left(u_{s}\mathbb{D}^{H}_{s}F\right)u_{r}I_{[0,s](r)}dr\right]ds\medskip\\
={\displaystyle\int_{0}^{t}}\mathbb{E}
\left[{\displaystyle\int_{0}^{T}}u_{r}I_{[0,s](r)}\delta B^{H}(r)u_{s}\mathbb{D}^{H}_{s}F\right]ds
=\mathbb{E}
\left[{\displaystyle\int_{0}^{t}}u_{s}X_{s}\mathbb{D}^{H}_{s}Fds\right].
\end{array}
\end{equation}
On the other hand, since also $\mathbb{D}_{s}^{H}uI_{[0,s]}\in\mathbb{L}^{1,2}_{H}\subset Dom(\delta)$ and $u_{s}F\in\mathbb{D}_{1,2}$, $s\in[0,t]$,
we obtain again from Fubini's Theorem as well as Definition \ref{divergence operator} that
\[
\begin{array}
[c]{l}
\mathbb{E}\left[{\displaystyle\int_{0}^{t}}{\displaystyle\int_{0}^{s}}\mathbb{D}^{H}_{r}(u_{s}F)\mathbb{D}^{H}_{s}u_{r}drds\right]
={\displaystyle\int_{0}^{t}}\mathbb{E}\left[{\displaystyle\int_{0}^{s}}\mathbb{D}^{H}_{r}(u_{s}F)\mathbb{D}^{H}_{s}u_{r}dr\right]ds\medskip\\
={\displaystyle\int_{0}^{t}}\mathbb{E}\left[{\displaystyle\int_{0}^{s}}\mathbb{D}^{H}_{s}u_{r}\delta B^{H}(r)u_{s}\cdot F\right]ds
=\mathbb{E}\left[{\displaystyle\int_{0}^{t}}\left({\displaystyle\int_{0}^{s}}\mathbb{D}^{H}_{s}u_{r}\delta B^{H}(r)\right)u_{s}\cdot Fds\right].
\end{array}
\]
Thus, due to (\ref{Malliavin derivative of X})
\begin{equation}\label{I2 plus I3}
\begin{array}
[c]{l}
I_{2}+I_{3}=\mathbb{E}\left[{\displaystyle\int_{0}^{t}}{\displaystyle\int_{0}^{s}}\phi(s-r)u_{s}u_{r}dsdr\cdot F\right]+\mathbb{E}\left[{\displaystyle\int_{0}^{t}}{\displaystyle\int_{0}^{s}}\mathbb{D}^{H}_{r}(u_{s}F)\mathbb{D}^{H}_{s}u_{r}drds\right]\medskip\\
=\mathbb{E}\left[{\displaystyle\int_{0}^{t}}u_{s}
\left({\displaystyle\int_{0}^{s}}\phi(s-r)u_{r}dr+{\displaystyle\int_{0}^{s}}\mathbb{D}^{H}_{s}u_{r}\delta B^{H}(r)\right)ds\cdot F\right]
=\mathbb{E}\left[F{\displaystyle\int_{0}^{t}}u_{s}\mathbb{D}^{H}_{s}X_{s}ds\right].
\end{array}
\end{equation}
Consequently, from (\ref{I1I2I3})-(\ref{I2 plus I3}),
\begin{equation}\label{adjoint equation}
\mathbb{E}\left[2{\displaystyle\int_{0}^{T}}u_{s}X_{s}I_{[0,t]}(s)\mathbb{D}^{H}_{s}Fds\right]
=\mathbb{E}
\left[\left(X_{t}^{2}-2{\displaystyle\int_{0}^{t}}u_{s}\mathbb{D}^{H}_{s}X_{s}ds\right)F\right],~\text{ for all } F\in\mathcal{P}_{T}.
\end{equation}
On the other hand, from Theorem 7.10 \cite{H-05} and the fact that $u\in\mathcal{V}_{T}$, it follows that there exists $C>0$ such that
\[
\begin{array}
[c]{l}
\mathbb{E}\left({\displaystyle\int_{0}^{t}}u_{s}\delta B^{H}(s)\right)^{4}
\leq C\mathbb{E}\|u\|_{T}^{4}
+C\mathbb{E}\left({\displaystyle\int_{0}^{T}}{\displaystyle\int_{0}^{T}}|\mathbb{D}^{H}_{t}u_{s}|^{2}dsdt\right)^{2}\medskip\\
\leq C\mathbb{E}\left({\displaystyle\int_{0}^{T}}|u_{s}|^{2}ds\right)^{2}
+C\mathbb{E}{\displaystyle\int_{0}^{T}}{\displaystyle\int_{0}^{T}}|\mathbb{D}^{H}_{t}u_{s}|^{4}dsdt\leq C,\text{ for all }t\in[0,T],
\end{array}
\]
as well as
\[
\begin{array}
[c]{l}
\mathbb{E}\left({\displaystyle\int_{0}^{t}}\mathbb{D}^{H}_{s}u_{r}\delta B^{H}(r)\right)^{4}
\leq C\mathbb{E}\|\mathbb{D}^{H}_{s}u\|_{T}^{4}
+C\mathbb{E}\left({\displaystyle\int_{0}^{T}}{\displaystyle\int_{0}^{T}}|\mathbb{D}^{H}_{t}(\mathbb{D}^{H}_{s}u_{r})|^{2}drdt\right)^{2}\medskip\\
\leq C+C\mathbb{E}{\displaystyle\int_{0}^{T}}{\displaystyle\int_{0}^{T}}|u_{xx}(r,\eta(r))|^{4}dr\leq C, \text{ for all } t\in[0,T].
\end{array}
\]
Taking into account the definition of the process $X$, we deduce from the above two estimates and Theorem \ref{Th.1_Hu 2005} that
\[
uX\in L^{2}(\Omega,\mathcal{F},P;\mathcal{H}) \text{ and } X_{t}^{2}-2{\displaystyle\int_{0}^{t}}u_{s}\mathbb{D}^{H}_{s}X_{s}ds\in L^{2}(\Omega,\mathcal{F},P).
\]
Therefore, from (\ref{adjoint equation}) and Definition \ref{divergence operator} it follows that $uXI_{[0,t]}\in Dom(\delta)$
and
\[
2{\displaystyle\int_{0}^{t}}u_{s}X_{s}\delta B^{H}(s)=X_{t}^{2}-2{\displaystyle\int_{0}^{t}}u_{s}\mathbb{D}^{H}_{s}X_{s}ds.
\]
\hfill
\end{proof}

\begin{proof}[Proof of Theorem \ref{Ito formula for divergence type integral}]
Let
\[
Y_{t}:={\displaystyle\int_{0}^{t}}u_{s}\delta B^{H}(s)\text{ and }Z_{t}:=X_{0}  +{\displaystyle\int_{0}^{t}}f_{s}ds,~t\in[0,T].
\]
From the previous lemma, we know that $uYI_{[0,t]}\in Dom(\delta)$, for all $t\in[0,T]$, and
\[
Y_{t}^{2}=2{\displaystyle\int_{0}^{t}}u_{s}Y_{s}\delta B^{H}(s) +2{\displaystyle\int_{0}^{t}}u_{s}\mathbb{D}^{H}_{s}Y_{s}ds,~t\in[0,T].
\]
On the other hand, it is obvious that $Z_{t}^{2}=X_{0}^{2}+2{\displaystyle\int_{0}^{t}}f_{s}Z_{s}ds, ~t\in[0,T].$
Moreover, we assert that $uZI_{[0,t]}\in Dom(\delta)$, for all $t\in[0,T]$, and
\[
Y_{t}Z_{t}={\displaystyle\int_{0}^{t}}u_{s}Z_{s}\delta B^{H}(s)+{\displaystyle\int_{0}^{t}}f_{s}Y_{s}ds
+{\displaystyle\int_{0}^{t}}u_{s}\mathbb{D}_{s}^{H}Z_{s}ds,~t\in[0,T].
\]
Indeed, since $Z_{t}F\in\mathbb{D}_{1,2}$ and $\mathbb{D}^{H}_{s}(Z_{t}F)=\mathbb{D}^{H}_{s}(Z_{s}F)+{\displaystyle\int_{s}^{t}}\mathbb{D}^{H}_{s}(f_{r}F)dr$, $s\in[0,t]$, we have
\[
\begin{array}
[c]{l}
\mathbb{E}\left[Y_{t}Z_{t}F\right]=\mathbb{E}\left[\left({\displaystyle\int_{0}^{T}}u_{s}I_{[0,t]}(s)\delta B^{H}(s)\right)Z_{t}F\right]
=\mathbb{E}\left[{\displaystyle\int_{0}^{t}}u_{s}\mathbb{D}^{H}_{s}(Z_{t}F)ds\right]\medskip\\
=\mathbb{E}\left[{\displaystyle\int_{0}^{t}}u_{s}\mathbb{D}^{H}_{s}(Z_{s}F)ds\right]
+\mathbb{E}\left[{\displaystyle\int_{0}^{t}}{\displaystyle\int_{s}^{t}}u_{s}\mathbb{D}^{H}_{s}(f_{r}F)drds\right]\medskip\\
=\mathbb{E}\left[{\displaystyle\int_{0}^{t}}u_{s}\mathbb{D}^{H}_{s}Z_{s}ds\cdot F\right]
+\mathbb{E}\left[{\displaystyle\int_{0}^{t}}u_{s}Z_{s}\mathbb{D}^{H}_{s}Fds\right]
+{\displaystyle\int_{0}^{t}}\mathbb{E}\left[{\displaystyle\int_{0}^{r}}u_{s}\mathbb{D}^{H}_{s}(f_{r}F)ds\right]dr\medskip\\
=\mathbb{E}\left[{\displaystyle\int_{0}^{t}}u_{s}\mathbb{D}^{H}_{s}Z_{s}ds\cdot F\right]
+\mathbb{E}\left[{\displaystyle\int_{0}^{t}}u_{s}Z_{s}\mathbb{D}^{H}_{s}Fds\right]
+\mathbb{E}\left[{\displaystyle\int_{0}^{t}}Y_{r}f_{r}Fdr\right].
\end{array}
\]
Therefore,
\[
\mathbb{E}\left[{\displaystyle\int_{0}^{T}}u_{s}Z_{s}I_{[0,t]}(s)\mathbb{D}^{H}_{s}Fds\right]
=\mathbb{E}\left[\left(Y_{t}Z_{t}-{\displaystyle\int_{0}^{t}}u_{s}\mathbb{D}^{H}_{s}Z_{s}ds
-{\displaystyle\int_{0}^{t}}Y_{r}f_{r}dr\right)F\right], ~F\in\mathcal{P}_{T},
\]
and since $uZI_{[0,t]}\in L^{2}(\Omega,\mathcal{F},P;\mathcal{H})$ as well as $Y_{t}Z_{t}-{\displaystyle\int_{0}^{t}}u_{s}\mathbb{D}^{H}_{s}Z_{s}ds-{\displaystyle\int_{0}^{t}}Y_{s}f_{s}ds\in L^{2}(\Omega,\mathcal{F},P)$,
we conclude from Definition \ref{divergence operator} that $uZI_{[0,t]}\in Dom(\delta)$ and
\[
{\displaystyle\int_{0}^{t}}u_{s}Z_{s}\delta B^{H}(s)=Y_{t}Z_{t}-{\displaystyle\int_{0}^{t}}u_{s}\mathbb{D}^{H}_{s}Z_{s}ds-{\displaystyle\int_{0}^{t}}Y_{s}f_{s}ds.
\]
Consequently, using the above notation as well as the linearity of $Dom(\delta)$, we have $X_{t}=Y_{t}+Z_{t}$, $uXI_{[0,t]}\in Dom(\delta)$ and
\[
\begin{array}
[c]{l}
X_{t}^{2}=Y_{t}^{2}+2Y_{t}Z_{t}+Z_{t}^{2}\medskip\\
=X_{0}^{2}+2{\displaystyle\int_{0}^{t}}X_{s}f_{s}ds+2{\displaystyle\int_{0}^{t}}X_{s}u_{s}\delta B^{H}(s)+2{\displaystyle\int_{0}^{t}}u_{s}\mathbb{D}^{H}_{s}X_{s}ds  , ~ t\in[0,T].
\end{array}
\]
\hfill
\end{proof}

Emphasizing that the It\^{o}$-$Skorohod integral and the divergence-type integral coincide for all $u\in\mathbb{L}^{1,2}_{H}$. Then,
from Hu and Peng Lemma 4.2 \cite{HP-09} the following lemma holds true:
\begin{lemma}\label{Lemma 4.2_HP-09}
Let $a,b\in C^{0,1}_{pol}([0,T]\times\mathbb{R})$. If
\[
\int_{0}^{t}b(s,\eta(s))ds+\int_{0}^{t}a(s,\eta(s))\delta B^{H}\left(  s\right)
=0,\text{ for all } t\in[0,T],
\]
then
\[
b(s,x)=a(s,x)=0,\text{ for all } t\in[0,T], ~x\in\mathbb{R}.
\]
\end{lemma}
\subsection{ Quasi-Conditional Expectation }
In this subsection, we recall the quasi-conditional expectation which was introduced by Hu and {\O}ksendal \cite{HO-03}.

For any $n\ge 1$, we introduce the set $\mathcal{H}^{\otimes n}$ of all real symmetric Borel functions $f_{n}$ of $n$ variables such that
\[
\|f_{n}\|^{2}_{\mathcal{H}^{\otimes n}}:=
{\displaystyle\int_{\mathbb{R}_{+}^{n}\times\mathbb{R}_{+}^{n}}}\prod\limits_{i=1}\limits^{n}\phi(s_{i}-r_{i})
f_{n}(s_{1},\ldots,s_{n})f_{n}(r_{1},\ldots,r_{n})ds_{1}\ldots ds_{n}dr_{1}\ldots dr_{n}<\infty.
\]
Then one can define the iterated integral (see \cite{BHOZ-06,H-05,HO-03} )
\[
I_{n}(f_{n})=n!{\displaystyle\int_{0\leq t_{1}<\ldots<t_{n}}}f_{n}(t_{1},\ldots,t_{n})dB^{H}(t_{1})\cdots dB^{H}(t_{n}).
\]
in the sense of It\^{o}$-$Skorohod.
For $n=0$ and $f=f_{0}$ be a constant, we set $I_{0}(f_{0})=f_{0}$ and $\|f_{0}\|^{2}_{\mathcal{H}^{\otimes 0}}=f_{0}^{2}$.

We recall the following theorem, see Theorem 3.9.9 \cite{BHOZ-06} or \cite{DHP-00} (Theorem 6.11) or \cite{HO-03} (Theorem 3.22).
\begin{theorem}\label{chaos expansion}
Let $F\in L^{2}(\Omega,\mathcal{F},P)$. Then there exist $f_{n}\in  \mathcal{H}^{\otimes n}$, $n\ge 0$ such that
\[
F=\sum\limits_{n=0}\limits^{\infty}I_{n}(f_{n}).
\]
Moreover,
\[
\mathbb{E}|F|^{2}=\sum\limits_{n=0}\limits^{\infty}n!\|f_{n}\|^{2}_{\mathcal{H}^{\otimes n}}<\infty.
\]
The convergence in this chaos expansion of $F$ is understood in the sense of $L^{2}(\Omega,\mathcal{F},P)$.
\end{theorem}

Let $\{\mathcal{F}_{t}\}_{t\ge0}$ be the filtration generated by $B^{H}$. We now recall the definition of quasi-conditional expectation (see \cite{HO-03} and \cite{HP-09}):
\begin{definition}
If $F\in L^{2}(\Omega,\mathcal{F},P)$, then the quasi-conditional expectation is defined as
\begin{equation}\label{quasi conditional expectation}
\hat{E}[F|\mathcal{F}_{t}]=\sum\limits_{n=0}\limits^{\infty}I_{n}(f_{n}I_{[0,t]}^{\otimes n}),~t\in[0,T],
\end{equation}
if the series on the right side converges in $L^{2}(\Omega,\mathcal{F},P)$. Here
\[
I_{[0,t]}^{\otimes n}(t_{1},\ldots,t_{n})=I_{[0,t]}(t_{1})\ldots I_{[0,t]}(t_{n}).
\]
\end{definition}
\begin{remark}\label{remark for quasi conditional expectation}
$\hat{E}\left[\hat{E}[F|\mathcal{F}_{t}]\big|\mathcal{F}_{s}\right]=\hat{E}[F|\mathcal{F}_{s}]$, for $0\leq s\leq t\leq T$.
\end{remark}
\begin{lemma}[Lemma 3.3 \cite{HP-09}]\label{Lemma 3.3 HU PENG}
For all $h\in \mathbb{L}^{1,2}_{H}$ and all $t\in[0,T]$, P-a.s.
\[
\hat{E}\left[{\displaystyle\int_{t}^{T}}h(u)\delta B^{H}(u)\Big| \mathcal{F}_{t}\right]=0.
\]
\end{lemma}
The following lemma is inspired by Theorem 3.9 \cite{HP-09}.
\begin{lemma}\label{Lemma for expectation and quasi conditional expectation}
Let $F=h(\eta(T))$, where $h:\mathbb{R}\rightarrow\mathbb{R}$ is a continuous function of polynomial growth. Then $F\in L^{2}(\Omega,\mathcal{F},P)$ and
\[
\mathbb{E}\left[\hat{E}[F|\mathcal{F}_{t}]\right]=\mathbb{E}F,~t\in[0,T].
\]
\end{lemma}
\begin{proof}
First, from the polynomial growth of $f$ and
\[
\mathbb{E}|\eta(T)|^{p}\leq C\mathbb{E}\left(1+\left|{\displaystyle\int_{0}^{T}}\sigma(v)dB^{H}(v)\right|^{p}\right)
\leq C_{p}+C_{p}\|\sigma\|_{T}^{p/2}\leq M_{p},~ p\ge1,
\]
we obtain $F\in L^{2}(\Omega,\mathcal{F},P)$.

We now put $p_{t}(x)=\dfrac{1}{\sqrt{2\pi t}}e^{-\frac{x^{2}}{2t}}$, $t\in(0,T],x\in\mathbb{R}$, and
\[
P_{t}h(x)={\displaystyle\int_{\mathbb{R}}}p_{t}(x-y)h(y)dy.
\]
Applying (\ref{deterministic Ito's formula}) to $P_{\|\sigma\|_{T}^{2}-\|\sigma\|_{t}^{2}}h(\eta(t))$ and noticing that
$\dfrac{\partial}{\partial t}P_{t}h(x)=\dfrac{1}{2}\dfrac{\partial^{2}}{\partial x^{2}}P_{t}h(x)$, we have
\begin{equation}\label{Ito formula equation}
\begin{array}
[c]{l}
h(\eta(T))=P_{\|\sigma\|_{T}^{2}-\|\sigma\|_{t}^{2}}h(\eta(t))
+{\displaystyle\int_{t}^{T}}{\displaystyle\dfrac{\partial}{\partial x}}P_{\|\sigma\|_{T}^{2}-\|\sigma\|_{s}^{2}}h(\eta(s))\sigma(s)\delta B^{H}(s),
\end{array}
\end{equation}
and hence,
\begin{equation}\label{equation 2}
\mathbb{E}h(\eta(T))=\mathbb{E}\left\{P_{\|\sigma\|_{T}^{2}-\|\sigma\|_{t}^{2}}h(\eta(t))\right\}.
\end{equation}
On the other hand, from the proof of Theorem 3.8 \cite{HP-09}, it follows that
\begin{equation}\label{quasi linear conditional expectation}
\hat{E}[F|\mathcal{F}_{t}]=P_{\|\sigma\|_{T}^{2}-\|\sigma\|_{t}^{2}}h(\eta(t)).
\end{equation}

For the reader's convenience, we give a justification for (\ref{quasi linear conditional expectation}) here.
By taking $t=0$ in (\ref{Ito formula equation}) we obtain
\[
\begin{array}
[c]{l}
h(\eta(T))=P_{\|\sigma\|_{T}^{2}}h(\eta(0))+{\displaystyle\int_{0}^{T}}\dfrac{\partial}{\partial x}P_{\|\sigma\|_{T}^{2}-\|\sigma\|_{s}^{2}}h(\eta(s))\sigma(s)\delta B^{H}(s).
\end{array}
\]
Thus, due to Lemma \ref{Lemma 3.3 HU PENG} and Remark 4.10 \cite{HO-03},
\begin{equation}\label{equation 20}
\hat{E}\left[F|\mathcal{F}_{t}\right]=P_{\|\sigma\|_{T}^{2}}h(\eta(0))+{\displaystyle\int_{0}^{t}}\dfrac{\partial}{\partial x}P_{\|\sigma\|_{T}^{2}-\|\sigma\|_{s}^{2}}h(\eta(s))\sigma(s)\delta B^{H}(s).
\end{equation}
On the other hand, by applying (\ref{deterministic Ito's formula}) to $P_{\|\sigma\|_{t}^{2}-\|\sigma\|_{s}^{2}}h(\eta(s))$ over time interval $[0,t]$, we get
\[
\begin{array}
[c]{l}
h(\eta(t))=P_{\|\sigma\|_{t}^{2}}h(\eta(0))+{\displaystyle\int_{0}^{t}}\dfrac{\partial}{\partial x}P_{\|\sigma\|_{t}^{2}-\|\sigma\|_{s}^{2}}h(\eta(s))\sigma(s)\delta B^{H}(s).
\end{array}
\]
Thus, from semigroup property
\[
P_{\|\sigma\|_{T}^{2}-\|\sigma\|_{s}^{2}}h(x)=P_{\|\sigma\|_{T}^{2}-\|\sigma\|_{t}^{2}}P_{\|\sigma\|_{t}^{2}-\|\sigma\|_{s}^{2}}h(x),
~0\leq s\leq t\leq T
\]
and
\[
\dfrac{\partial}{\partial x}P_{\|\sigma\|_{T}^{2}-\|\sigma\|_{s}^{2}}h(x)=P_{\|\sigma\|_{T}^{2}-\|\sigma\|_{t}^{2}}\dfrac{\partial}{\partial x}P_{\|\sigma\|_{t}^{2}-\|\sigma\|_{s}^{2}}h(x),
\]
then, we get form (\ref{equation 20}) that
\[
\begin{array}
[c]{l}
\hat{E}\left[F|\mathcal{F}_{t}\right]=P_{\|\sigma\|_{T}^{2}}h(\eta(0))+{\displaystyle\int_{0}^{t}}\dfrac{\partial}{\partial x}P_{\|\sigma\|_{T}^{2}-\|\sigma\|_{s}^{2}}h(\eta(s))\sigma(s)\delta B^{H}(s)\medskip\\
=P_{\|\sigma\|_{T}^{2}-\|\sigma\|_{t}^{2}}\left[P_{\|\sigma\|_{t}^{2}}h(\eta(0))+{\displaystyle\int_{0}^{t}}\dfrac{\partial}{\partial x}P_{\|\sigma\|_{t}^{2}-\|\sigma\|_{s}^{2}}h(\eta(s))\sigma(s)\delta B^{H}(s)\right]\medskip\\
=P_{\|\sigma\|_{T}^{2}-\|\sigma\|_{t}^{2}}h(\eta(t)).
\end{array}
\]
This means that we have proved (\ref{quasi linear conditional expectation}). Consequently, from (\ref{equation 2}) and (\ref{quasi linear conditional expectation}), we have
\[
\mathbb{E}\left\{\hat{E}[F|\mathcal{F}_{t}]\right\}=\mathbb{E}\left\{P_{\|\sigma\|_{T}^{2}-\|\sigma\|_{t}^{2}}h(\eta(t))\right\}=\mathbb{E}F.
\]
\hfill
\end{proof}

\section{BSDEs Driven By $B^{H}$}
The objective of this section is to study the BSDE
\begin{equation}  \label{BSDE}
\left\{
\begin{array}
[c]{l}%
-dY(t)= f(t,\eta(t),Y(t),Z(t))dt-Z(t)\delta B^{H}\left(
t\right)  ,\quad t\in[0,T],\medskip\\
Y(T)=\xi.
\end{array}
\right.
\end{equation}
We now give the definition of the solution for the above BSDE.
\begin{definition}
A pair $(Y,Z)$ is called a solution of BSDE (\ref{BSDE}), if the following
conditions are satisfied:
\[
\begin{array}
[c]{rl}
(a_{1}) & Y\in\mathcal{\bar{V}}_{T}^{1/2}\;\text{ and } Z\in\mathcal{\bar{V}}_{T}^{H} ~(\text{Recall } (\ref{def V^alpha}));\medskip\\
(a_{2}) & Y(t)=\xi+{\displaystyle\int_{t}^{T}}f\left(  s,\eta(s),Y(s),Z(s)\right)ds
-{\displaystyle\int_{t}^{T}}Z(s)\delta B^{H}(s),~a.s., ~t\in(0,T].
\end{array}
\]
\end{definition}
Let us begin by discussing the existence of a solution for BSDE (\ref{BSDE}).
\subsection{Existence}
We begin with considering the following equation:
\begin{equation}\label{fBSDE}
Y(t)=\xi+{\displaystyle\int_{t}^{T}}f\left(  s,\eta(s),\chi(s,\eta(s)),\psi(s,\eta(s))\right)ds
-{\displaystyle\int_{t}^{T}}Z(s)\delta B^{H}(s),~t\in[0,T],
\end{equation}
where $\chi,\psi\in C_{pol}^{1,3}\left(
[0,T]\times\mathbb{R}\right)$ with $\dfrac{\partial \chi}{\partial t},\dfrac{\partial \psi}{\partial t}\in C_{pol}^{0,1}([0,T]\times\mathbb{R})$. Observe that (\ref{fBSDE}) is a special case of BSDE (\ref{BSDE}).

In Proposition 4.5 \cite{HP-09}, the existence problem of a solution for an equation of type (\ref{fBSDE}) was not explicitly specified. Therefore, we shall give the following proposition:
\begin{proposition}\label{Proposition 4.5_4.3_HP-09}
Under the assumptions $(H_{1})$-$(H_{4})$, BSDE (\ref{fBSDE}) has a unique solution $(Y,Z)\in \mathcal{V}_{T}\times \mathcal{V}_{T}$. This solution has the form
\[
\begin{array}
[c]{l}
(i)  Y(t)=u(t,\eta(t)), \quad Z(t)=v(t,\eta(t)),\medskip\\
(ii)  v(t,x)=\sigma(t)\dfrac{\partial }{\partial x}u(t,x).
\end{array}
\]
where $u,v\in C_{pol}^{1,3}([0,T]\times\mathbb{R})$ with $\dfrac{\partial u}{\partial t},\dfrac{\partial v}{\partial t}\in C_{pol}^{0,1}([0,T]\times\mathbb{R})$.
\end{proposition}
Before giving the proof, we show the following auxiliary result:
\begin{lemma}\label{Fubini theorem}
Assume that $h\in C^{0,1}_{pol}([0,T]\times\mathbb{R})$ and put $h_{s}=h(s,\eta(s))$, $s\in[0,T]$. Then
\[
\hat{E}\left[{\displaystyle\int_{t}^{T}}h_{s}ds\big|\mathcal{F}_{t}\right]
={\displaystyle\int_{t}^{T}}\hat{E}\left[h_{s}\big|\mathcal{F}_{t}\right]ds, ~P-a.s., ~t\in[0,T].
\]
\end{lemma}
\begin{proof}
Since $h\in C^{0,1}_{pol}([0,T]\times\mathbb{R})$, we know that $h_{s}\in L^{2}(\Omega,\mathcal{F},P)$, for all $s\in[0,T]$.
From Theorem \ref{chaos expansion} there exist $h_{n,s}\in \mathcal{H}^{\otimes n}$, $n\ge0$ such that $h_{s}=\sum\limits_{n=0}\limits^{\infty}I_{n}(h_{n,s})$, with $I_{0}(h_{0,s})=\mathbb{E}h_{s}$.
Recall that the series converges in $L^{2}(\Omega,\mathcal{F},P)$.
From the proof of Theorem 3.9.9 \cite{BHOZ-06} we deduce that $h_{n,s}$ is measurable w.r.t. $s$,  for $n\ge0$.
Similarly, for $G:={\displaystyle\int_{t}^{T}}h_{s}ds\in L^{2}(\Omega,\mathcal{F},P)$ there exist $G_{n}\in \mathcal{H}^{\otimes n}$, $n\ge0$, such that, $G=\sum\limits_{n=0}\limits^{\infty}I_{n}(G_{n})$ with $I_{0}(G_{0})=\mathbb{E}G$.
Let us show that we can choose $G_{n}={\displaystyle\int_{t}^{T}} h_{n,s}ds$, $n\ge0$. For this, we observe that
\[
\mathbb{E}{\displaystyle\int_{t}^{T}}\left|\sum\limits_{n=0}\limits^{N}I_{n}(h_{n,s})\right|^{2}ds
=\mathbb{E}{\displaystyle\int_{t}^{T}}\sum\limits_{n=0}\limits^{N}\left|I_{n}(h_{n,s})\right|^{2}ds
\leq \mathbb{E}{\displaystyle\int_{t}^{T}}\sum\limits_{n=0}\limits^{\infty}\left|I_{n}(h_{n,s})\right|^{2}ds
=\mathbb{E}{\displaystyle\int_{t}^{T}}|h_{s}|^{2}ds<\infty,
\]
and hence $I_{n}(h_{n,s})$ is square integrable w.r.t. $s$.

Now, for arbitrary $F\in L^{2}(\Omega,\mathcal{F},P)$ with the chaos expansion $F=\sum\limits_{n=0}\limits^{\infty}I_{n}(l_{n})$, we have from Fubini's Theorem
\[
\begin{array}
[c]{l}
\mathbb{E}\left[F\cdot I_{n}(G_{n})\right]=\mathbb{E}\left[I_{n}(l_{n})\cdot I_{n}(G_{n})\right]
=\mathbb{E}\left[I_{n}(l_{n})\cdot {\displaystyle\int_{t}^{T}}h_{s}ds\right]
={\displaystyle\int_{t}^{T}}\mathbb{E}\left[I_{n}(l_{n})\cdot h_{s}\right]ds\medskip\\
={\displaystyle\int_{t}^{T}}\mathbb{E}\left[I_{n}(l_{n})\cdot I_{n}(h_{n,s})\right]ds
={\displaystyle\int_{t}^{T}}\mathbb{E}\left[F\cdot I_{n}(h_{n,s})\right]ds
=\mathbb{E}\left[F\cdot{\displaystyle\int_{t}^{T}} I_{n}(h_{n,s})ds\right].
\end{array}
\]
It follows that $I_{n}(G_{n})={\displaystyle\int_{t}^{T}} I_{n}(h_{n,s})ds, ~n\ge0$,
and the stochastic Fubini Theorem (see Theorem 1.13.1 \cite{M-07}) yields
$I_{n}(G_{n})={\displaystyle\int_{t}^{T}} I_{n}(h_{n,s})ds=I_{n}\left({\displaystyle\int_{t}^{T}} h_{n,s}ds\right)$.
Thus, we can indeed choose  $G_{n}={\displaystyle\int_{t}^{T}} h_{n,s}ds$, $n\ge0$
and $I_{[0,t]}^{\otimes n}G_{n}=I_{[0,t]}^{\otimes n}{\displaystyle\int_{t}^{T}} h_{n,s}ds$.
Consequently, we have
\begin{equation*}\label{change order}
\lim\limits_{N\rightarrow\infty}\sum\limits_{n=0}\limits^{N}I_{n}({\displaystyle\int_{t}^{T}} h_{n,s}ds)=\lim\limits_{N\rightarrow\infty}\sum\limits_{n=0}\limits^{N}I_{n}(G_{n})=G
={\displaystyle\int_{t}^{T}} h_{s}ds,~\text{ in } L^{2}(\Omega,\mathcal{F},P).
\end{equation*}
Now we are going to show that
$\hat{E}\left[{\displaystyle\int_{t}^{T}}h_{s}ds\big|\mathcal{F}_{t}\right]={\displaystyle\int_{t}^{T}}\hat{E}\left[h_{s}\big|\mathcal{F}_{t}\right]ds$.
In fact
\[
\begin{array}
[c]{l}
\hat{E}\left[{\displaystyle\int_{t}^{T}}h_{s}ds\big|\mathcal{F}_{t}\right]
=\hat{E}\left[\sum\limits_{n=0}\limits^{\infty}I_{n}(G_{n})\big|\mathcal{F}_{t}\right]
=\sum\limits_{n=0}\limits^{\infty}I_{n}(I^{\otimes n}_{[0,t]}G_{n})\medskip\\
\qquad\qquad\qquad
=\sum\limits_{n=0}\limits^{\infty}I_{n}(I^{\otimes n}_{[0,t]}{\displaystyle\int_{t}^{T}} h_{n,s}ds)
=\sum\limits_{n=0}\limits^{\infty}I_{n}({\displaystyle\int_{t}^{T}} I^{\otimes n}_{[0,t]}h_{n,s}ds)=\sum\limits_{n=0}\limits^{\infty}{\displaystyle\int_{t}^{T}} I_{n}(I^{\otimes n}_{[0,t]}h_{n,s})ds,
\end{array}
\]
where, for the later equality in the above equation, we have used the stochastic Fubini Theorem.
Moreover, from (\ref{quasi linear conditional expectation}), we know that for $s\ge t$, $\hat{E}\left[h_{s}|\mathcal{F}_{t}\right]$ exists and
$\hat{E}\left[h_{s}|\mathcal{F}_{t}\right]= \sum\limits_{n=0}\limits^{\infty}I_{n}(I^{\otimes n}_{[0,t]}h_{n,s})$. Then from $h\in C^{0,1}_{pol}([0,T]\times\mathbb{R})$, Theorem 3.9 of \cite{HP-09} and  dominate convergence theorem, we have
\[
\sum\limits_{n=0}\limits^{\infty}{\displaystyle\int_{t}^{T}} I_{n}(I^{\otimes n}_{[0,t]}h_{n,s})ds
={\displaystyle\int_{t}^{T}} \sum\limits_{n=0}\limits^{\infty}I_{n}(I^{\otimes n}_{[0,t]}h_{n,s})ds
={\displaystyle\int_{t}^{T}} \hat{  E}\left[h_{s}|\mathcal{F}_{t}\right]ds, ~P-a.s.
\]
which completes our proof.
\hfill
\end{proof}
\medskip

After the above auxiliary result, we can now prove our Proposition \ref{Proposition 4.5_4.3_HP-09}.
\medskip

\begin{proof}[Proof of Proposition \ref{Proposition 4.5_4.3_HP-09}]
Using similar arguments to those of the proof of Lemma \ref{property of VT}, we get
\[
g(\eta(T))+{\displaystyle\int_{0}^{T}}f(s,\eta(s),\chi(s,\eta
(s)),\psi(s,\eta(s)))ds\in L^{2}(\Omega,\mathcal{F},P).
\]
Then, we define
\[
M(t)=\hat{E}\left[g(\eta(T))+{\displaystyle\int_{0}^{T}}f(s,\eta(s),\chi(s,\eta
(s)),\psi(s,\eta(s)))ds\big|\mathcal{F}_{t}\right], t\in[0,T].
\]
Using (\ref{quasi linear conditional expectation}) and Lemma \ref{Fubini theorem}, we obtain
\[
\begin{array}
[c]{l}
M(t)=P_{\|\sigma\|_{T}^{2}-\|\sigma\|_{t}^{2}}g(\eta(t))
+{\displaystyle\int_{t}^{T}}P_{\|\sigma\|_{u}^{2}-\|\sigma\|_{t}^{2}}f(u,\eta(t),\chi(u,\eta(t)),\psi(u,\eta(t)))du\medskip\\
\qquad\qquad\qquad\qquad\qquad+{\displaystyle\int_{0}^{t}}f(u,\eta(u),\chi(u,\eta(u)),\psi(u,\eta(u)))du.
\end{array}
\]
Moreover, $\hat{E}[M(t)|\mathcal{F}_{s}]=M(s)$, $0\leq s\leq t\leq T$ (see Remark \ref{remark for quasi conditional expectation}).
Recall the definition of the Malliavin derivative, it follows that if $F\in \mathbb{D}_{1,2}$ is $\mathcal{F}_{t}$-measurable, then $D_{s}^{H}F=0$, $ds$-a.e. on $[t,T]$. Therefore, for $s\in[0,t]$
\begin{equation*}\label{Malliavin derivative of M}
D_{s}^{H}M(t)=\sigma_{s}P_{\|\sigma\|_{T}^{2}-\|\sigma\|_{t}^{2}}g^{\prime}(\eta(t))
+\sigma_{s}{\displaystyle\int_{t}^{T}}P_{\|\sigma\|_{u}^{2}-\|\sigma\|_{t}^{2}}\tilde{g}(u,\eta(t))du
+\sigma_{s}{\displaystyle\int_{s}^{t}}\tilde{g}(u,\eta(u))du
\end{equation*}
where
\[
\begin{array}
[c]{l}
\tilde{g}(u,x)=\dfrac{\partial}{\partial x}f(u,x,\chi(u,x),\psi(u,x))+\dfrac{\partial}{\partial y}f(u,x,\chi(u,x),\psi(u,x))\chi_{x}(u,x)\medskip\\
\qquad\qquad\qquad+\dfrac{\partial}{\partial z}f(u,x,\chi(u,x),\psi(u,x))\psi_{x}(u,x).
\end{array}
\]
Moreover, (\ref{quasi linear conditional expectation}) and the semigroup property of $P_{u}$ yields
\begin{equation}\label{equation 3}
\begin{array}
[c]{l}
\hat{E}[D_{s}^{H}M(t)|\mathcal{F}_{s}]=\sigma_{s}P_{\|\sigma\|_{T}^{2}-\|\sigma\|_{t}^{2}}P_{\|\sigma\|_{t}^{2}-\|\sigma\|_{s}^{2}}g^{\prime}(\eta(s))
\medskip\\
\qquad\quad+\sigma_{s}{\displaystyle\int_{t}^{T}}P_{\|\sigma\|_{u}^{2}-\|\sigma\|_{t}^{2}}P_{\|\sigma\|_{t}^{2}-\|\sigma\|_{s}^{2}}\tilde{g}(u,\eta(s))du
+\sigma_{s}{\displaystyle\int_{s}^{t}}P_{\|\sigma\|_{u}^{2}-\|\sigma\|_{s}^{2}}\tilde{g}(u,\eta(s))du\medskip\\
\qquad=\sigma_{s}P_{\|\sigma\|_{T}^{2}-\|\sigma\|_{s}^{2}}g^{\prime}(\eta(s))
+\sigma_{s}{\displaystyle\int_{s}^{T}}P_{\|\sigma\|_{u}^{2}-\|\sigma\|_{s}^{2}}\tilde{g}(u,\eta(s))du.\medskip\\
\end{array}
\end{equation}
Now, we are going to prove that $\mathbb{E}|M(t)|^{2}<\infty$. Indeed,
\[
\begin{array}
[c]{l}
\mathbb{E}|M(t)|^{2}\leq
3\mathbb{E}\left|\hat{E}[g(\eta(T))|\mathcal{F}_{t}]\right|^{2}+3\mathbb{E}\left|\hat{E}\left[{\displaystyle\int_{t}^{T}}f(s,\eta(s),\chi(s,\eta
(s)),\psi(s,\eta(s)))ds\big|\mathcal{F}_{t}\right]\right|^{2}\medskip\\
\qquad\qquad\qquad+3\mathbb{E}\left|{\displaystyle\int_{0}^{t}}f(s,\eta(s),\chi(s,\eta
(s)),\psi(s,\eta(s)))ds\right|^{2},
\end{array}
\]
and, similarly to Theorem 3.9 \cite{HP-09}, we obtain
\[
\mathbb{E}\left|\hat{E}[g(\eta(T))|\mathcal{F}_{t}]\right|^{2}\leq \mathbb{E}|g(\eta(T))|^{2}
\]
On the other hand, from Lemma \ref{Fubini theorem}, we have
\[
\begin{array}
[c]{l}
\mathbb{E}\left|\hat{E}\left[{\displaystyle\int_{t}^{T}}f(s,\eta(s),\chi(s,\eta
(s)),\psi(s,\eta(s)))ds\big|\mathcal{F}_{t}\right]\right|^{2}\medskip\\
=\mathbb{E}\left|{\displaystyle\int_{t}^{T}}\hat{E}\left[f(s,\eta(s),\chi(s,\eta
(s)),\psi(s,\eta(s)))\big|\mathcal{F}_{t}\right]ds\right|^{2}\medskip\\
\leq T{\displaystyle\int_{t}^{T}}\mathbb{E}\left|\hat{E}\left[f(s,\eta(s),\chi(s,\eta
(s)),\psi(s,\eta(s)))\big|\mathcal{F}_{t}\right]\right|^{2}ds\medskip\\
\leq T{\displaystyle\int_{t}^{T}}\mathbb{E}|f(s,\eta(s),\chi(s,\eta(s)),\psi(s,\eta(s)))|^{2}ds.
\end{array}
\]
Consequently, $\mathbb{E}|M(t)|^{2}<\infty$. Then by using fractional Clark formula (see \cite{H-05} and \cite{HO-03}) we get
\begin{equation}\label{fractional Clark formula}
\begin{array}
[c]{l}
M(t)=\mathbb{E}M(t)+{\displaystyle\int_{0}^{t}}\hat{E}[D_{s}^{H}M(t)|\mathcal{F}_{s}]\delta B^{H}(s).
\end{array}
\end{equation}
From Lemmas \ref{Lemma for expectation and quasi conditional expectation} and \ref{Fubini theorem} we have $\mathbb{E}\left\{\hat{E}[g(\eta(T))|\mathcal{F}_{t}]\right\}=\mathbb{E}g(\eta(T))$ and
\[
\begin{array}
[c]{l}
\mathbb{E}\left\{\hat{E}\left[{\displaystyle\int_{t}^{T}}f(s,\eta(s),\chi(s,\eta(s)),\psi(s,\eta(s)))ds\big|\mathcal{F}_{t}\right]\right\}\medskip\\
={\displaystyle\int_{t}^{T}}\mathbb{E}\left\{\hat{E}\left[f(s,\eta(s),\chi(s,\eta(s)),\psi(s,\eta(s)))\big|\mathcal{F}_{t}\right]\right\}ds\medskip\\
=\mathbb{E}{\displaystyle\int_{t}^{T}}f(s,\eta(s),\chi(s,\eta(s)),\psi(s,\eta(s)))ds.
\end{array}
\]
Consequently, from the definition of $M(t)$ and Remark 4.10 \cite{HO-03}, we have
\[
M(T)=g(\eta(T))+{\displaystyle\int_{0}^{T}}f(s,\eta(s),\chi(s,\eta(s)),\psi(s,\eta(s)))ds
\]
and
\begin{equation}\label{equation 4}
\begin{array}
[c]{l}
\mathbb{E}M(t)=\mathbb{E}\left\{\hat{E}[g(\eta(T))|\mathcal{F}_{t}]\right\}
+\mathbb{E}{\displaystyle\int_{0}^{t}}f(s,\eta(s),\chi(s,\eta(s)),\psi(s,\eta(s)))ds\medskip\\
\qquad\qquad\qquad\qquad
+\mathbb{E}\left\{\hat{E}\left[{\displaystyle\int_{t}^{T}}f(s,\eta(s),\chi(s,\eta(s)),\psi(s,\eta(s)))ds\big|\mathcal{F}_{t}\right]\right\}
\medskip\\
=\mathbb{E}g(\eta(T))+\mathbb{E}{\displaystyle\int_{0}^{T}}f(s,\eta(s),\chi(s,\eta(s)),\psi(s,\eta(s)))ds
=\mathbb{E}M(T).
\end{array}
\end{equation}
On the other hand, from (\ref{equation 3}) we obtain
\begin{equation}\label{equation 5}
\hat{E}[D_{s}^{H}M(t)|\mathcal{F}_{s}]=\hat{E}[D_{s}^{H}M(T)|\mathcal{F}_{s}],~P-a.s.,~s\in[0,t].
\end{equation}
We deduce from (\ref{fractional Clark formula}), (\ref{equation 4}) and (\ref{equation 5}) that
\[
M(t)=\mathbb{E}M(T)+{\displaystyle\int_{0}^{t}}\hat{E}[D_{s}^{H}M(T)|\mathcal{F}_{s}]\delta B^{H}(s).
\]
Let us now introduce the process $Z(s)=\hat{E}[D_{s}^{H}M(T)|\mathcal{F}_{s}],~s\in[0,T]$. Similarly as Proposition 4.5 \cite{HP-09}, using the property of the operator $P_{\|\sigma\|_{T}^{2}-\|\sigma\|_{s}^{2}}$, we can prove that $Z\in \mathcal{V}_{T}$. Moreover, in virtue of the latter relation, we have
\[
M(t)=\mathbb{E}M(T)+{\displaystyle\int_{0}^{t}}Z(s)\delta B^{H}(s),~P-a.s.~t\in[0,T].
\]
Now, we define
\[
Y(t)=M(t)-{\displaystyle\int_{0}^{t}}f(s,\eta(s),\chi(s,\eta
(s)),\psi(s,\eta(s)))ds, ~  t\in[0,T].
\]
Then,
\[
\begin{array}
[c]{l}
Y(T)-Y(t)=M(T)-M(t)-{\displaystyle\int_{t}^{T}}f(s,\eta(s),\chi(s,\eta(s)),\psi(s,\eta(s)))ds\medskip\\
={\displaystyle\int_{t}^{T}}Z(s)dB^{H}(s)-{\displaystyle\int_{t}^{T}}f(s,\eta(s),\chi(s,\eta(s)),\psi(s,\eta(s)))ds,
\end{array}
\]
Moreover, it's not hard to check that $Y(T)=M(T)-{\displaystyle\int_{0}^{T}}f(s,\eta(s),\chi(s,\eta(s)),\psi(s,\eta(s)))ds=g(\eta(T))=\xi$, so we have
\[
Y(t)=\xi+{\displaystyle\int_{t}^{T}}f\left(  s,\eta(s),\chi(s,\eta
(s)),\psi(s,\eta(s))\right)  ds-{\displaystyle\int_{t}^{T}}Z(s)\delta B^{H}\left(
s\right).
\]
Moreover, from Remark 4.10 \cite{HO-03} and (\ref{quasi linear conditional expectation})
\[
\begin{array}
[c]{l}
Y(t)=\hat{E}[Y(t)|\mathcal{F}_{t}]=\hat{E}\left[g(\eta(T))+{\displaystyle\int_{t}^{T}}f\left(  s,\eta(s),\chi(s,\eta
(s)),\psi(s,\eta(s))\right)  ds\Big|\mathcal{F}_{t}\right]\medskip\\
=P_{\|\sigma\|_{T}^{2}-\|\sigma\|_{t}^{2}}g(\eta(t))+{\displaystyle\int_{t}^{T}}P_{\|\sigma\|_{s}^{2}-\|\sigma\|_{t}^{2}}f\left(  s,\eta(t),\chi(s,\eta
(t)),\psi(s,\eta(t))\right)  ds,
\end{array}
\]
so that it can be easily shown that also $Y\in \mathcal{V}_{T}$. Consequently, we have constructed a solution $(Y,Z)\in \mathcal{V}_{T}\times\mathcal{V}_{T}$
for BSDE (\ref{fBSDE}). Moreover, $Y$ is continuous since $Y\in\mathcal{V}_{T}$.

Finally, using that $Y,Z\in\mathcal{V}_{T}$, we can find $u,v\in C_{pol}^{1,3}([0,T]\times\mathbb{R})$ with $\dfrac{\partial u}{\partial t},\dfrac{\partial v}{\partial t}\in C_{pol}^{0,1}([0,T]\times\mathbb{R})$ such that $Y(t)=u(t,\eta(t)), Z(t)=v(t,\eta(t))$, $t\in[0,T]$. Then $ v(t,x)=\sigma(t)\dfrac{\partial}{\partial x}u(t,x)$. Indeed, by applying
(\ref{deterministic Ito's formula}) we have
\[
\begin{array}
[c]{l}
du(t,\eta(t))={\displaystyle\dfrac{\partial}{\partial t}}u(t,\eta(t))dt+\sigma(t){\displaystyle\dfrac{\partial}{\partial x}}u(t,\eta
(t))\delta B^{H}\left(  t\right) +\dfrac{1}{2}\tilde{\sigma}(t){\displaystyle\dfrac{\partial^{2}}{\partial x^{2}}}u(t,\eta(t))dt\medskip\\
 =\left[ {\displaystyle\dfrac{\partial}{\partial t}}u(t,\eta(t))+\dfrac
{1}{2}\tilde{\sigma}(t){\displaystyle\dfrac{\partial^{2}}{\partial x^{2}}}u(t,\eta(t))\right]  dt+\sigma
(t){\displaystyle\dfrac{\partial}{\partial x}}u(t,\eta(t))\delta B^{H}\left(  t\right),
\end{array}
\]
where $\tilde{\sigma}(t):=\dfrac{d}{dt}(\Vert\sigma\Vert_{t}^{2})$. Consequently
\[
u(t,\eta(t)) =\xi-\int_{t}^{T}\left[\dfrac{\partial}{\partial s}u(s,\eta(s))
+\dfrac{1}{2}\tilde{\sigma}(s)\dfrac{\partial^{2}}{\partial x^{2}}u(s,\eta(s))\right] ds
-\int_{t}^{T}\sigma(s)\dfrac{\partial}{\partial x}u(s,\eta(s))\delta B^{H}\left(  s\right).
\]
From (\ref{fBSDE}) it can be concluded that
\[
\begin{array}
[c]{l}
{\displaystyle\int_{t}^{T}}\left[{\displaystyle\dfrac{\partial}{\partial s}}u(s,\eta(s))
+\dfrac{1}{2}\tilde{\sigma}(s){\displaystyle\dfrac{\partial^{2}}{\partial x^{2}}}u(s,\eta(s))\right]ds
+{\displaystyle\int_{t}^{T}}\sigma(s){\displaystyle\dfrac{\partial}{\partial x}}u(s,\eta(s))\delta B^{H}\left(  s\right)\medskip\\
=-{\displaystyle\int_{t}^{T}}f(s,\eta(s),\chi(s,\eta(s)),\psi(s,\eta(s)))ds+{\displaystyle\int_{t}^{T}}v(s,\eta(s)) \delta B^{H}\left( s\right).
\end{array}
\]
Using Lemma \ref{Lemma 4.2_HP-09} we deduce that
\begin{equation}\label{relation between u and v}
v(t,x)=\sigma(t)\dfrac{\partial}{\partial x}u(t,x),\text{ for all } t\in[0,T], ~x\in\mathbb{R}.
\end{equation}
It remains to prove that the above solution is the unique one in $\mathcal{V}_{T}\times\mathcal{V}_{T}$ for BSDE (\ref{fBSDE}). Indeed, we suppose that there is another solution $(\tilde{Y},\tilde{Z})\in\mathcal{V}_{T}\times\mathcal{V}_{T}$.
Then, by applying Theorem \ref{Ito formula for divergence type integral}, using (\ref{relation between u and v}) (which yields $\mathbb{D}_{t}^{H}Y(t)=\dfrac{\hat{\sigma}(t)}{\sigma(t)}Z(t)$) and taking expectation, we have
\[
\begin{array}
[c]{l}%
\mathbb{E}|Y(t)-\tilde{Y}(t)|^{2}+\dfrac{2}{M}{\displaystyle\int_{t}^{T}}s^{2H-1}\mathbb{E}|Z(s)-\tilde{Z}(s)|^{2}ds\medskip\\
\leq\mathbb{E}|Y(t)-\tilde{Y}(t)|^{2}+2{\displaystyle\int_{t}^{T}}\dfrac{\hat{\sigma}(s)}{\sigma(s)}\mathbb{E}|Z(s)-\tilde{Z}(s)|^{2}ds\medskip\\
=2\mathbb{E}{\displaystyle\int_{t}^{T}}\left(Y(t)-\tilde{Y}(t) \right)\left(Z(t)-\tilde{Z}(t)\right) \delta B^{H}(s)=0,~\text{ for all }t\in [0,T],
\end{array}
\]
where the latter equality follows the fact that $(Y-\tilde{Y})(Z-\tilde{Z})\in\mathcal{V}_{T}$. Therefore, taking into account the continuity of $Y-\tilde{Y}$, the uniqueness follows.\hfill
\end{proof}
\begin{proposition}\label{estimate for main result Hu&Peng}
Let the assumptions $(H_{1})$-$(H_{4})$ be satisfied. For $(U,V)\in\mathcal{V}_{T}\times\mathcal{V}_{T}$,
let $(Y,Z)\in\mathcal{V}_{T}\times\mathcal{V}_{T}$ be the unique solution of the following BSDE
\[
Y(t)=\int_{t}^{T}f\left(  s,\eta(s),U(s),V(s)\right)  ds-\int_{t}^{T}Z(s)\delta B^{H}(s), ~t\in[0,T] .
\]
Then, for all $\beta>0$, there exists $C\left(  \beta\right)\in\mathbb{R}$ (depending also on $L$ and
$T$) such that
\begin{equation}\label{a priori estimate}
\begin{array}
[c]{l}
\sup\limits_{t\in[0,T]}e^{2\beta t}\mathbb{E}\left\vert Y(t)\right\vert^{2}
+{\displaystyle\int_{0}^{T}}e^{2\beta s}\mathbb{E}\left\vert Y(s)\right\vert^{2}ds
+{\displaystyle\int_{0}^{T}}s^{2H-1}e^{2\beta s}\mathbb{E}|Z\left(  s\right)  |^{2}ds\medskip\\
\leq C\left(  \beta\right)  \left(  {\displaystyle\int_{0}^{T}
}e^{2\beta s}\mathbb{E}|U(s)|^{2}ds+{\displaystyle\int_{0}^{T}}s^{2H-1}e^{2\beta
s}\mathbb{E}|V(s)|^{2}ds+{\displaystyle\int_{0}^{T}}e^{2\beta s}\left\vert
f\left(  s,\eta(s),0,0\right)  \right\vert ^{2}ds\right).
\end{array}
\end{equation}
Moreover, $C(\beta)$ can be chosen such that $\lim\limits_{\beta\rightarrow\infty}C\left(  \beta\right)=0$.
\end{proposition}
\begin{proof}
From Proposition \ref{Proposition 4.5_4.3_HP-09}, it is not hard to check that $\mathbb{D}_{t}^{H}Y(t)=\dfrac{\hat{\sigma}(t)}{\sigma(t)}Z(t)$.
Then, from Theorem \ref{Ito formula for divergence type integral}, we deduce that
for $t\in[0,T]$,
\begin{equation*}\label{Ito for contraction}
\begin{array}
[c]{l}
\mathbb{E}|Y(t)|^{2}+\dfrac{2}{M}{\displaystyle\int_{t}^{T}}s^{2H-1}\mathbb{E}|Z(s)|^{2}ds\leq\mathbb{E}|Y(t)|^{2}
+2{\displaystyle\int_{t}^{T}}\dfrac{\hat{\sigma}(s)}{\sigma(s)}\mathbb{E}|Z(s)|^{2}ds\medskip\\
=2{\displaystyle\int_{t}^{T}}\mathbb{E}\Big[Y(s)f\left(s,\eta(s),U(s),V(s)\right)\Big]ds\medskip\\
\quad\leq2{\displaystyle\int_{t}^{T}}\mathbb{E}\left[ \left\vert Y(s)\right\vert
\left(L|U(s)|+L|V(s)|+\left\vert f\left(s,\eta(s),0,0\right)\right\vert\right)\right]ds\medskip\\
\quad\leq2{\displaystyle\int_{t}^{T}}\left[\mathbb{E}\left\vert Y(s)\right\vert^{2}\right]^{1/2}
\left[\mathbb{E}\left(L|U(s)|+L|V(s)|+\left\vert f\left(s,\eta(s),0,0\right)\right\vert \right)^{2}\right]^{1/2}ds.
\end{array}
\end{equation*}
Let $x(t)=\left[\mathbb{E}\left\vert Y(t)\right\vert ^{2}\right]^{1/2}$, $t\in[0,T]$. Then
\begin{equation}
x^{2}\left(  t\right)  \leq2\sqrt{3}{\int_{t}^{T}}x\left(  s\right)  \left[
\mathbb{E}\left(  L^{2}|U(s)|^{2}+L^{2}|V(s)|^{2}+\left\vert f\left(
s,\eta(s),0,0\right)  \right\vert ^{2}\right)  \right]  ^{1/2}ds,\;t\in\left[
0,T\right]  .\label{deterministic ineq}%
\end{equation}
To estimate $x(t)$, we will apply the following inequality:
\begin{lemma}
Let $a,\alpha,\beta:\left[  0,T\right]  \rightarrow\mathbb{R}_{+}$ be three nonnegative Borel functions such that $a$ is decreasing and $\alpha,\beta\in L^{1}_{loc}([0,\infty])$.
If $x:\left[0,T\right]\rightarrow\mathbb{R}_{+}$ is a continuous function such that%
\[
x^{2}\left(  t\right)  \leq a\left(  t\right)  +2\int_{t}^{T}\alpha\left(
s\right)  x\left(  s\right)  ds+2\int_{t}^{T}\beta\left(  s\right)
x^{2}\left(  s\right)  ds,\;t\in\left[  0,T\right]  ,
\]
then
\[
x\left(  t\right)  \leq\sqrt{a\left(  t\right)  }\exp\left(  \int_{t}^{T}
\beta\left(  s\right)  ds\right)  +\int_{t}^{T}\alpha\left(  s\right)
\exp\left(  \int_{t}^{s}\beta\left(  r\right)  dr\right)  ds,\;t\in\left[
0,T\right]  .
\]
\end{lemma}
\begin{remark}
For this lemma the reader is referred to Corollary 6.61 \cite{PR-11} .
\end{remark}
Now from (\ref{deterministic ineq}) and the above lemma, by setting
\[
\begin{array}
[c]{l}
a(t)=0,\quad \beta(s)=0, \medskip\\
\alpha(s)=\sqrt{3}\left[
\mathbb{E}\left(  L^{2}|U(s)|^{2}+L^{2}|V(s)|^{2}+\left\vert f\left(
s,\eta(s),0,0\right)  \right\vert ^{2}\right)\right]^{1/2},~s\in[0,T],
\end{array}
\]
we have
\[
x(t) \leq\sqrt{3}\int_{t}^{T}\left[  \mathbb{E}\left(
L^{2}|U(s)|^{2}+L^{2}|V(s)|^{2}+\left\vert f\left(  s,\eta(s),0,0\right)
\right\vert ^{2}\right)  \right]  ^{1/2}ds,\quad t\in[0,T],
\]
and, hence, for any $\beta>0$,
\begin{equation}\label{ineq Ito formula 1}
\begin{array}
[c]{l}
\left[\mathbb{E}\left\vert Y(t)\right\vert^{2}\right]^{1/2}
\leq\sqrt{3}{\displaystyle\int_{t}^{T}}\left(L\left[\mathbb{E}|U(s)|^{2}\right]^{1/2}+L\left[\mathbb{E}|V(s)|^{2}\right]^{1/2}
+\left[\mathbb{E}\left\vert f\left(s,\eta(s),0,0\right)\right\vert^{2}\right]^{1/2}\right)ds\medskip\\
\leq\sqrt{3}L{\displaystyle\int_{t}^{T}}\bigg(e^{-\beta s}\left[e^{2\beta s}\mathbb{E}|U(s)|^{2}\right]^{1/2}
+\dfrac{e^{-\beta s}}{s^{H-1/2}}\left[s^{2H-1}e^{2\beta s}\mathbb{E}|V(s)|^{2}\right]^{1/2}\bigg)ds\medskip\\
\quad\quad
+\sqrt{3}{\displaystyle\int_{t}^{T}}e^{-\beta s}\left[e^{2\beta s}\mathbb{E}\left\vert f\left(s,\eta(s),0,0\right)\right\vert ^{2}\right]^{1/2}ds\\
\leq\sqrt{3}L\left({\displaystyle\int_{t}^{T}}e^{-2\beta s}ds\right)^{1/2}\left({\displaystyle\int_{t}^{T}}e^{2\beta s}\mathbb{E}|U(s)|^{2}ds\right)^{1/2}\medskip\\
\quad\quad+\sqrt{3}L\left(  {\displaystyle\int_{t}^{T}}\dfrac{e^{-2\beta s}}{s^{2H-1}}ds\right)^{1/2}\left({\displaystyle\int_{t}^{T}}s^{2H-1}
e^{2\beta s}\mathbb{E}| V(s)|^{2}ds\right)^{1/2}\medskip\\
\quad\quad+\sqrt{3}\left(  {\displaystyle\int_{t}^{T}}e^{-2\beta s}ds\right)^{1/2}
\left({\displaystyle\int_{t}^{T}}e^{2\beta s}\mathbb{E}\left\vert f\left(s,\eta(s),0,0\right)\right\vert ^{2}ds\right)^{1/2}.
\end{array}
\end{equation}
Let us use the following notations:
\[
\begin{array}
[c]{l}
A_{t}:=\left(  {\displaystyle\int_{t}^{T}}e^{2\beta s}\mathbb{E}|U(s)|^{2}\right)^{1/2},
\quad B_{t}:=\left(  {\displaystyle\int_{t}^{T}}s^{2H-1}e^{2\beta s}\mathbb{E}| V(s)|^{2}ds\right)^{1/2}\medskip\\
\text{and}\quad C_{t}=\left(  {\displaystyle\int_{t}^{T}}e^{2\beta
s}\mathbb{E}\left\vert f\left(  s,\eta(s),0,0\right)  \right\vert ^{2}\right)
^{1/2},~t\in[0,T].
\end{array}
\]
Since ${\displaystyle\int_{t}^{T}}e^{-2\beta s}ds=\dfrac{1}{2\beta}\left(  e^{-2\beta t}-e^{-2\beta T}\right)$
we have for $\alpha>0$ with $0<\alpha<2-2H<1$ and $\beta>0$,
\[
e^{2\beta t}{\int_{t}^{T}}\dfrac{e^{-2\beta s}}{s^{2H-1}}ds\leq{\int_{t}^{T}}\dfrac{\left(2\beta(s-t)\right)^{-\alpha}}{s^{2H-1}}ds
\leq\dfrac{1}{\left(  2\beta\right)^{\alpha}}{\int_{0}^{T}}\dfrac{1}{s^{\alpha+2H-1}}ds<\infty.
\]
This allows to conclude from (\ref{ineq Ito formula 1}), that
\begin{equation}\label{ineq Ito formula 2}
e^{2\beta t}\mathbb{E}\left\vert Y(t)\right\vert ^{2}\leq\dfrac{9L^{2}}{2\beta}~{A}_{t}^{2}
+\dfrac{9L^{2}}{(2\beta)^{\alpha}}{\int_{t}^{T}}\dfrac{(s-t)^{-\alpha}}{s^{2H-1}}ds~{B}_{t}^{2}+\dfrac{9}{2\beta}~{C}_{t}^{2}.
\end{equation}
Consequently, there exists $C\left(  \beta\right)  $ with $\lim
\limits_{\beta\rightarrow\infty}C\left(  \beta\right)  =0$,  s.t.
\begin{equation}\label{ineq Ito formula 3}
e^{2\beta t}\mathbb{E}\left\vert Y(t)\right\vert^{2}dt
\leq C(\beta)\left( {A}_{t}^{2}+{B}_{t}^{2}+{C}_{t}^{2}\right), ~t\in[0,T].
\end{equation}
Applying the It\^{o} formula to $|Y(t)|^{2}$, taking the expectation $\mathbb{E}|Y(t)|^{2}$ and then determining the function $d\left(e^{2\beta t}\mathbb{E}|Y(t)|^{2}\right)$ and using $\mathbb{D}_{t}^{H}Y(t)=\dfrac{\hat{\sigma}(t)}{\sigma(t)}Z(t)$, the Lipschitz property of $f$ as well as (\ref{ineq Ito formula 2}) we obtain
(Recall (\ref{properties for sigma}) for the definition of $M$)
\[
\begin{array}
[c]{l}
e^{2\beta t}\mathbb{E}|Y(t)|^{2}+2\beta{\displaystyle\int_{t}^{T}}e^{2\beta s}\mathbb{E}|Y(s)|^{2}ds
+\dfrac{2}{M}{\displaystyle\int_{t}^{T}}s^{2H-1}e^{2\beta s}\mathbb{E}|Z(s)|^{2}ds\medskip\\
\quad\leq2{\displaystyle\int_{t}^{T}}e^{2\beta s}\mathbb{E}\left[\left\vert Y(s)\right\vert
\left( L|U(s)|+L|V(s)|+\left\vert f\left(s,\eta(s),0,0\right)\right\vert \right)\right]ds\medskip\\
\quad\leq2L{\displaystyle\int_{t}^{T}}\left[\mathbb{E}\left( e^{2\beta s}|Y(s)|^{2}\right)\right]^{1/2}
\left[\mathbb{E}\left(  e^{2\beta s}|U(s)|^{2}\right)\right]^{1/2}ds\medskip\\
\quad\quad+2L{\displaystyle\int_{t}^{T}}\left[\mathbb{E}\left(\dfrac{e^{2\beta s}}{s^{2H-1}}|Y(s)|^{2}\right)\right]^{1/2}
\left[\mathbb{E}\left(e^{2\beta s}s^{2H-1}|V(s)|^{2}\right)\right]^{1/2}ds\medskip\\
\quad\quad+2{\displaystyle\int_{t}^{T}}\left[\mathbb{E}\left(e^{2\beta s}|Y(s)|^{2}\right)\right]^{1/2}
\left[\mathbb{E}\left(e^{2\beta s}\left\vert f\left(s,\eta(s),0,0\right)\right\vert ^{2}\right)\right]^{1/2}ds\medskip\\
\quad\leq2L{\displaystyle\int_{t}^{T}}\left[C(\beta)\left(A_{s}^{2}+B_{s}^{2}+C_{s}^{2}\right)\right]^{1/2}
\left[\mathbb{E}\left(e^{2\beta s}|U(s)|^{2}\right)\right]^{1/2}ds\medskip\\
\quad\quad+2L{\displaystyle\int_{t}^{T}}\left[\dfrac{1}{s^{2H-1}}C(\beta)\left(A_{s}^{2}+B_{s}^{2}+C_{s}^{2}\right)\right]^{1/2}
\left[\mathbb{E}\left(  e^{2\beta s}s^{2H-1}|V(s)|^{2}\right)\right]^{1/2}ds\medskip\\
\quad\quad+2{\displaystyle\int_{t}^{T}}\left[C(\beta)\left(A_{s}^{2}+B_{s}^{2}+C_{s}^{2}\right)\right]^{1/2}
\left[\mathbb{E}\left(e^{2\beta s}\left\vert f\left(s,\eta(s),0,0\right)\right\vert^{2}\right)\right]^{1/2}ds\medskip\\
\quad\leq2L\sqrt{C(\beta)}\left(A_{t}+B_{t}+C_{t}\right)\left(\sqrt{T-t}~{A}_{t}+\sqrt{\dfrac{T^{2-2H}-t^{2-2H}}{2-2H}}~B_{t}+\sqrt{T-t}~{C}_{t}\right).
\end{array}
\]
Thus, the above inequality and (\ref{ineq Ito formula 3}) allow to conclude inequality (\ref{a priori estimate}).\hfill
\end{proof}
\begin{theorem}\label{main result Hu&Peng}
Let the assumptions $(H_{1})$-$(H_{4})$ be satisfied. Then the BSDE
\[
Y(t)=\xi+\int_{t}^{T}f\left(  s,\eta(s),Y(s),Z(s)\right) ds-\int_{t}^{T}Z(s)\delta B^{H}(s),~t\in[0,T]
\]
has a solution $(Y,Z)\in\mathcal{\bar{V}}_{T}^{1/2}\times
\mathcal{\bar{V}}_{T}^{H}$.
\end{theorem}
\begin{remark}
Let us mention that it is not clear here, if the solution $Y$ has continuous paths or not.
Indeed, since $Z$ does not necessarily belong to $\mathbb{L}^{H}_{1,2}$, the divergence integral $\int_{t}^{T}Z(s)\delta B^{H}(s)$ can eventually
be discontinuous in $t$.
\end{remark}
\begin{proof}
The existence of the solution is obtained by the Banach fixed point theorem. Let us consider the mapping
$\Gamma:\mathcal{V}_{T}\times\mathcal{V}_{T}\rightarrow\mathcal{V}_{T}\times \mathcal{V}_{T}$
given by $\left(  U,V\right)  \longrightarrow\Gamma\left(  U,V\right)=\left(  Y,Z\right)  $, where $(Y,Z)$ is the unique solution in
$\mathcal{V}_{T}\times\mathcal{V}_{T}$ for the BSDE
\[
Y(t)=\xi+\int_{t}^{T}f\left(  s,\eta(s),U(s),V(s)\right) ds-\int_{t}^{T}Z(s)\delta B^{H}\left(  s\right) ,~t\in[0,T] .
\]
First, we remark that $\Gamma$ is well defined (see Proposition \ref{Proposition 4.5_4.3_HP-09}).

Let us show that $\Gamma$ is a contraction w.r.t. the norm $\|(u,v)\|_{1/2,H}:=\|u\|_{1/2}+\|v\|_{H}$, for $(u,v)\in\overline{\mathcal{V}}^{1/2}_{T}\times\overline{\mathcal{V}}^{H}_{T}$ (for the definition of $\|\cdot\|_{\alpha}$, see (\ref{def V^alpha})).

For $(U,V),(U^{\prime},V^{\prime})\in\mathcal{V}_{T}\times\mathcal{V}_{T}$ and $(Y,Z)=\Gamma(U,V)$, $(Y^{\prime},Z^{\prime})=\Gamma(U^{\prime},V^{\prime})$,
we set $\Delta Y=Y-Y^{\prime}$, $\Delta Z=Z-Z^{\prime}$, $\Delta U=U-U^{\prime}$ and $\Delta V=V-V^{\prime}$.
Using Proposition \ref{estimate for main result Hu&Peng} we know
that there exists $C(\beta)$ which can depend on $L$ and $T$, such that $\lim\limits_{\beta\rightarrow\infty}C(\beta)=0$,
and
\begin{equation}\label{contraction equation}
\begin{array}
[c]{l}
\sup\limits_{t\in[0,T]}e^{2\beta t}\mathbb{E}\left\vert \Delta Y(t)\right\vert^{2}
+{\displaystyle\int_{0}^{T}}e^{2\beta s}\mathbb{E}\left\vert \Delta Y\left(s\right)\right\vert ^{2}ds
+{\displaystyle\int_{0}^{T}}s^{2H-1}e^{2\beta s}\mathbb{E}|\Delta Z(s)|^{2}ds\medskip\\
\quad\quad\leq C(\beta)\left(  {\displaystyle\int_{0}^{T}}e^{2\beta s}\mathbb{E}|\Delta U(s)|^{2}ds
+{\displaystyle\int_{0}^{T}}s^{2H-1}e^{2\beta s}\mathbb{E}|\Delta V(s)|^{2}ds\right).
\end{array}
\end{equation}
Taking $\beta$ large enough such that $C(\beta)\leq1/2$, then $\Gamma$ becomes a strict contraction on $\mathcal{V}_{T}\times
\mathcal{V}_{T}$ w.r.t. the norm $\|(\cdot,\cdot)\|_{1/2,H}$.

Now, we define $\left\{(Y_{k},Z_{k})\right\}_{k\in\mathbb{N}}$ recursively by putting $Y_{0}=\chi(t,\eta(t))$,
$Z_{0}=\psi(t,\eta(t))$ for $\chi,\psi\in C_{pol}^{1,3}([0,T]\times\mathbb{R})$ with
$\dfrac{\partial \chi}{\partial t},\dfrac{\partial \psi}{\partial t}\in C_{pol}^{0,1}([0,T]\times\mathbb{R})$,
and by defining $(Y_{k+1},Z_{k+1})\in\mathcal{V}_{T}\times\mathcal{V}_{T}$ through the BSDE
\begin{equation}\label{recursiv penalized fBSVI}
Y_{k+1}(t)=\xi+{\displaystyle\int_{t}^{T}}f\left(  s,\eta(s),Y_{k}(s),Z_{k}(s)\right)ds-
{\displaystyle\int_{t}^{T}}Z_{k+1}(s)dB^{H}(s),~t\in[0,T],
\end{equation}
$k\ge0$. From Proposition \ref{Proposition 4.5_4.3_HP-09}, we know that  for all $k\ge0$,
$Y_{k}(t)=u_{k}(t,\eta\left(  t\right)  )$, $Z_{k}(t)=v_{k}(t,\eta(t))$, for suitable
$u_{k},v_{k}\in C_{pol}^{1,3}\left([0,T]\times\mathbb{R}\right)$ with
$\dfrac{\partial u_{k}}{\partial t},\dfrac{\partial v_{k}}{\partial t}\in C_{pol}^{0,1}([0,T]\times\mathbb{R})$  such that
$Z_{k}(t)=\sigma(t)\dfrac{\partial}{\partial x}u_{k}(t,\eta(t))$, $t\in[0,T]$.

Since $\Gamma$ is contraction, which means that $\left\{(Y_{k},Z_{k})\right\}_{k\in\mathbb{N}}$
is a Cauchy sequence in $\mathcal{\bar{V}}_{T}^{1/2}\times\mathcal{\bar{V}}_{T}^{H}$, then
there exists $(Y,Z)\in \mathcal{\bar{V}}_{T}^{1/2}\times\mathcal{\bar{V}}_{T}^{H}$ such that
$(Y_{k},Z_{k})$ converges to $(Y,Z)$ in $\mathcal{\bar{V}}_{T}^{1/2}\times\mathcal{\bar{V}}_{T}^{H}$.
Moreover, we can prove that $(Y,Z)$ satisfies
\[
Y(t)=\xi+{\displaystyle\int_{t}^{T}}f\left(  s,\eta(s),Y(s),Z(s)\right)ds
-{\displaystyle\int_{t}^{T}}Z(s)\delta B^{H}\left(  s\right),~t\in\left(0,T\right].
\]
Indeed, from (\ref{contraction equation})
\begin{equation}\label{convergent property}
\begin{array}
[c]{ll}
\lim\limits_{k\rightarrow\infty}\mathbb{E}|Y_{k}(t)-Y(t)|^{2}=0,~t\in[0,T],\text{ and } \medskip\\
\lim\limits_{k\rightarrow\infty}\mathbb{E}{\displaystyle\int_{0}^{T}}|Y_{k}(s)-Y(s)|^{2}ds
+\mathbb{E}{\displaystyle\int_{0}^{T}}s^{2H-1}|Z_{k}(s)-Z(s)|^{2}ds=0.
\end{array}
\end{equation}
Then, for arbitrary $\rho>0$ and for all $t\in[\rho,T]$,
\begin{equation}\label{equation 7}
\begin{array}
[c]{ll}
\lim\limits_{k\rightarrow\infty}\left\{-Y_{k+1}(t)+\xi+{\displaystyle\int_{t}^{T}}f\left(  s,\eta(s),Y_{k}(s),Z_{k}(s)\right)ds\right\}\medskip\\
=-Y(t)+\xi+{\displaystyle\int_{t}^{T}}f\left(  s,\eta(s),Y(s),Z(s)\right) ds:=\theta(t),\text{ in } L^{2}(\Omega,\mathcal{F},P),
\end{array}
\end{equation}
and $Z_{k}\mathbf{1}_{[t,T]}\rightarrow Z\mathbf{1}_{[t,T]}$ in $L^{2}(\Omega,\mathcal{F},P;\mathcal{H})$.
Therefore, using Definition \ref{divergence operator}, (\ref{recursiv penalized fBSVI}) and (\ref{equation 7}), we see that
for all $F\in\mathcal{P}_{T}$,
\[
\begin{array}
[c]{l}
\mathbb{E}\left( \left\langle D_{\cdot}^{H}F,Z(\cdot)\mathbf{1}_{[t,T]}(\cdot)\right\rangle _{T}\right)
=\lim\limits_{k\rightarrow\infty}\mathbb{E}\left( \left\langle D_{\cdot}^{H}F,Z_{k+1}(\cdot)\mathbf{1}_{[t,T]}(\cdot)\right\rangle _{T}\right) \medskip\\
=\lim\limits_{k\rightarrow\infty}\mathbb{E}\left( F{{\displaystyle\int_{t}^{T}}}Z_{k+1}(s)\delta B^{H}(s)\right)=\mathbb{E}(F~\theta(t)).
\end{array}
\]
From the definition of the divergence operator $\delta$, it follows that $Z\mathbf{1}_{[t,T]}\in Dom(\delta)$ and $\delta(Z\mathbf{1}_{[t,T]})=\theta(t)$. Consequently, we have
\[
Y(t)=\xi+{\displaystyle\int_{t}^{T}}f\left(  s,\eta(s),Y(s),Z(s)\right)ds-{\displaystyle\int_{t}^{T}}Z(s)\delta B^{H}(s),
~a.s., ~\text{ for all } t\in[\rho,T].
\]
Considering that $\rho$ is arbitrary, we complete our proof.\hfill
\end{proof}
\begin{proposition}\label{connection Y, Z}
Let $(Y,Z)\in\mathcal{V}_{T}\times  \mathcal{V}_{T}$ be the solution of BSDE
(\ref{BSDE}) constructed in the proof of Theorem \ref{main result Hu&Peng}. Then for almost $ t\in(0,T]$,
\[
\mathbb{D}_{t}^{H}Y(t)=\dfrac{\hat{\sigma}(t)}{\sigma(t)}Z(t).
\]
\end{proposition}
\begin{proof}
From (\ref{recursiv penalized fBSVI}) we know that $(Y_{k},Z_{k})\in  \mathcal{V}_{T}\times \mathcal{V}_{T}$ satisfies
\[
Y_{k+1}(t)=\xi+{\displaystyle\int_{t}^{T}}f\left(  s,\eta(s),Y_{k}(s),Z_{k}(s)\right)ds
-{\displaystyle\int_{t}^{T}}Z_{k+1}(s)\delta B^{H}(s),~ t\in[0,T],~k\ge1.
\]
We recall that $Y_{k}(t)=u_{k}(t,\eta(t))$, $Z_{k}(t)=v_{k}(t,\eta(t))$, $t\in[0,T]$ and
$Z_{k}(t)=\sigma(t)\dfrac{\partial}{\partial x}u_{k}(t,\eta(t))$.
Since $(Y_{k},Z_{k})\rightarrow(Y,Z)$ in $\mathcal{\bar{V}}_{T}^{1/2}\times\mathcal{\bar{V}}_{T}^{H}$,
there exists a subsequence, by convenience still denoted by $\left\{(Y_{k},Z_{k})\right\}_{k\in\mathbb{N}}$, such that
for arbitrary $\rho>0$, we have that
\[
\lim\limits_{k\rightarrow\infty}\mathbb{E}|Y_{k}(s)-Y(s)|^{2}=0 \text{ and }
\lim\limits_{k\rightarrow\infty}\mathbb{E}|Z_{k}(s)-Z(s)|^{2}=0,\text{ for almost all } s\in[\rho,T].
\]
As a process with the parameter $r$,
\[
\begin{array}
[c]{l}
D_{r}Y_{k}(t)={\displaystyle\dfrac{\partial}{\partial x}}u(t,\eta(t))\sigma(r)\mathbf{1}_{[0,t]}(r)
={\displaystyle\dfrac{\sigma(r)}{\sigma(t)}}Z_{k}(t)\mathbf{1}_{[0,t]}(r) \xrightarrow[]{\;L^{2}([0,T]\times\Omega) \;} {\displaystyle\dfrac{\sigma(r)}{\sigma(t)}}Z(t)\mathbf{1}_{[0,t]}(r),
\medskip\\
\qquad\qquad\qquad\qquad\text{ as } k\rightarrow\infty,\quad \text{ for almost all } t\in[\rho,T].
\end{array}
\]
On the other hand, since $L^{2}([0,T])\subset\mathcal{H}$,
we conclude that the convergence also holds in $L^{2}(\Omega,\mathcal{F},P;\mathcal{H})$. Consequently, in $L^{2}(\Omega,\mathcal{F},P;\mathcal{H})$
\[
D_{r}Y(t)=\lim\limits_{k\rightarrow\infty}D_{r}Y_{k}(t)=\lim\limits_{k\rightarrow\infty}\dfrac{\sigma(r)}{\sigma(t)}
Z_{k}(t)\mathbf{1}_{[0,t]}(r)=\dfrac{\sigma(r)}{\sigma(t)}Z(t)\mathbf{1}_{[0,t]}(r),~\text{ a.e. }t\in[\rho,T],
\]
and, thus,
\[
\mathbb{D}_{t}^{H}Y(t)={\int_{0}^{T}}\phi(t-r)D_{r}Y(t)dr=\dfrac{\hat{\sigma}(t)}{\sigma(t)}Z(t),~\text{ a.e. }t\in[\rho,T],
\]
where $\hat{\sigma}(t)$ is defined by (\ref{sigma tilde_hat}). Considering that $\rho>0$ is arbitrary, we have
\[
\mathbb{D}_{t}^{H}Y(t)=\dfrac{\hat{\sigma}(t)}{\sigma(t)}Z(t),~\text{ a.e. } t\in(0,T],
\]
which completes the proof.\hfill
\end{proof}
\subsection{Uniqueness}
Before giving our uniqueness result, we introduce the following spaces:
\[
\begin{array}
[c]{l}
\mathcal{M}=\Bigg\{X\big|X(t)=X(0)-{\displaystyle\int_{0}^{t}}v_{s}ds-{\displaystyle\int_{0}^{t}}u_{s}\delta B^{H}(s),~t\in[0,T] \text{ with }
\medskip\\
\qquad\qquad\qquad~u\in\mathcal{V}_{T},~v_{s}=v(s,\eta(s)), \text{ where }v\in C^{0,1}_{pol}([0,T]\times\mathbb{R})\Bigg\}
\end{array}
\]
and $\mathcal{S}_{f}$, the set of $(Y,Z)\in  \mathcal{\bar{V}}_{T}^{1/2}\times\mathcal{\bar{V}}_{T}^{H}$ such that, for $t\in(0,T]$,
\[
\begin{array}
[c]{l}
(i)~\mathbb{E}|Y(t)|^{2}+2\mathbb{E}{\displaystyle\int_{t}^{T}}\dfrac{\hat{\sigma}(s)}{\sigma(s)}| Z(s)|^{2}ds
=\mathbb{E}|\xi|^{2}+2\mathbb{E}{\displaystyle\int_{t}^{T}}Y(s)f(s,\eta(s),Y(s),Z(s))ds,\medskip\\
(ii)~ \mathbb{E}\left[Y(t)X(t)\right]+\mathbb{E}{\displaystyle\int_{t}^{T}}
\left[\dfrac{\hat{\sigma}(s)}{\sigma(s)}u_{s}+\mathbb{D}_{s}^{H}X(s)\right]Z(s)ds=\mathbb{E}\left[\xi X(T)\right]\medskip\\
\qquad+\mathbb{E}{\displaystyle\int_{t}^{T}}\left[Y(s)v_{s}+X(s)f(s,\eta(s),Y(s),Z(s))\right]ds,
~X\in\mathcal{M}.
\end{array}
\]

\begin{theorem}\label{uniqueness theorem}
Let the assumptions $(H_{1})$-$(H_{4})$ be satisfied.
Then BSDE
\[
Y(t)=\xi+\int_{t}^{T}f\left( s,\eta(s),Y(s),Z(s)\right)  ds-\int_{t}^{T}Z(s)\delta B^{H}(s)
\]
has a unique solution $(Y,Z)\in\mathcal{S}_{f}$.
\end{theorem}
\begin{proof}
We show first that the solution $(Y,Z)$ we constructed in the proof of Theorem \ref{main result Hu&Peng} belongs to $\mathcal{S}_{f}$.
Indeed, the sequence $\left\{(Y_{k},Z_{k})\right\}_{k\in\mathbb{N}}$ introduced in the proof of Theorem \ref{main result Hu&Peng} is in $\mathcal{V}_{T}\times\mathcal{V}_{T}$ and  converges to $(Y,Z)$ in $\mathcal{\bar{V}}_{T}^{1/2}\times\mathcal{\bar{V}}_{T}^{H}$.
Applying the It\^{o} formula to $Y_{k+1}^{2}$ (see Theorem \ref{Ito formula for divergence type integral}, using $\mathbb{D}_{t}^{H}Y(t)=\dfrac{\hat{\sigma}(t)}{\sigma(t)}Z(t)$) and taking the expectation, we have
\begin{equation}\label{equation for Yk}
\mathbb{E}|Y_{k+1}(t)|^{2}+2\mathbb{E}{\displaystyle\int_{t}^{T}}\dfrac{\hat{\sigma}(s)}{\sigma(s)}| Z_{k+1}(s)|^{2}ds
=\mathbb{E}|\xi|^{2}+2\mathbb{E}{\displaystyle\int_{t}^{T}}Y_{k+1}(s)f(s,\eta(s),Y_{k}(s),Z_{k}(s))ds.
\end{equation}
Moreover, from (\ref{convergent property}) we know
\[
\begin{array}
[c]{ll}
\lim\limits_{k\rightarrow\infty}\mathbb{E}|Y_{k}(t)-Y(t)|^{2}=0,~t\in[0,T],\text{ and } \medskip\\
\lim\limits_{k\rightarrow\infty}\mathbb{E}{\displaystyle\int_{t}^{T}}|Y_{k}(s)-Y(s)|^{2}ds
+\mathbb{E}{\displaystyle\int_{t}^{T}}s^{2H-1}|Z_{k}(s)-Z(s)|^{2}ds=0,~t\in[0,T].
\end{array}
\]
Letting $k\rightarrow\infty$ in (\ref{equation for Yk}), it follows that for arbitrary $\rho>0$,
\begin{equation}\label{equation for Y}
\mathbb{E}|Y(t)|^{2}+2\mathbb{E}{\displaystyle\int_{t}^{T}}\dfrac{\hat{\sigma}(s)}{\sigma(s)}| Z(s)|^{2}ds
=\mathbb{E}|\xi|^{2}+2\mathbb{E}{\displaystyle\int_{t}^{T}}Y(s)f(s,\eta(s),Y(s),Z(s))ds,~t\in[\rho,T].
\end{equation}
On the other hand, for any $X\in\mathcal{M}$, we deduce from Theorem \ref{Ito formula for divergence type integral},
\[
\begin{array}
[c]{ll}
\mathbb{E}\left[Y_{k+1}(t)X(t)\right]+\mathbb{E}{\displaystyle\int_{t}^{T}}
\left[u_{s}\mathbb{D}^{H}_{s}Y_{k+1}(s)+Z_{k+1}(s)\mathbb{D}^{H}_{s}X(s)\right]ds\medskip\\
=\mathbb{E}\left[\xi X(T)\right]+\mathbb{E}{\displaystyle\int_{t}^{T}}\left[Y_{k+1}(s)v_{s}+X(s)f(s,\eta(s),Y_{k}(s),Z_{k}(s))\right]ds.
\end{array}
\]
Letting $k\rightarrow\infty$ in the above equation and recalling that
$\mathbb{D}_{s}^{H}Y_{k+1}(s)=\dfrac{\hat{\sigma}(s)}{\sigma(s)}Z_{k+1}(s)$, we obtain for arbitrary $\rho>0$,
\begin{equation}\label{equation for Y and X}
\begin{array}
[c]{ll}
\mathbb{E}\left[Y(t)X(t)\right]+\mathbb{E}{\displaystyle\int_{t}^{T}}
\left[{\displaystyle\dfrac{\hat{\sigma}(s)}{\sigma(s)}}u_{s}+\mathbb{D}^{H}_{s}X(s)\right]Z(s)ds\medskip\\
=\mathbb{E}\left[\xi X(T)\right]+\mathbb{E}{\displaystyle\int_{t}^{T}}\left[Y(s)v_{s}+X(s)f(s,\eta(s),Y(s),Z(s))\right]ds,~
t\in[\rho,T].
\end{array}
\end{equation}
Consequently, (\ref{equation for Y}) and (\ref{equation for Y and X}) yield that $(Y,Z)\in\mathcal{S}_{f}$.

Now, it remains to show the uniqueness in the class $\mathcal{S}_{f}$.
We suppose that $(\tilde{Y},\tilde{Z})\in \mathcal{S}_{f}$ is another solution of BSDE (\ref{BSDE}).
Then, for arbitrary $\rho>0$,
\begin{equation}\label{equation for tilde Y}
\mathbb{E}|\tilde{Y}(t)|^{2}+2\mathbb{E}{\displaystyle\int_{t}^{T}}\dfrac{\hat{\sigma}(s)}{\sigma(s)}| \tilde{Z}(s)|^{2}ds
=\mathbb{E}|\xi|^{2}+2\mathbb{E}{\displaystyle\int_{t}^{T}}\tilde{Y}(s)f(s,\eta(s),\tilde{Y}(s),\tilde{Z}(s))ds,~t\in[\rho,T].
\end{equation}
and
\[
\begin{array}
[c]{ll}
\mathbb{E}\left[\tilde{Y}(t)Y_{k+1}(t)\right]+\mathbb{E}{\displaystyle\int_{t}^{T}}
2{\displaystyle\dfrac{\hat{\sigma}(s)}{\sigma(s)}}Z_{k+1}(s)\tilde{Z}(s)ds\medskip\\
=\mathbb{E}|\xi|^{2}+\mathbb{E}{\displaystyle\int_{t}^{T}}
\left[\tilde{Y}(s)f(s,\eta(s),Y_{k}(s),Z_{k}(s))+Y_{k+1}(s)f(s,\eta(s),\tilde{Y}(s),\tilde{Z}(s))\right]ds,~
t\in[\rho,T],
\end{array}
\]
and letting $k\rightarrow\infty$, we have
\begin{equation}\label{equation for tilde Y and Y}
\begin{array}
[c]{ll}
\mathbb{E}\left[\tilde{Y}(t)Y(t)\right]+\mathbb{E}{\displaystyle\int_{t}^{T}}2\dfrac{\hat{\sigma}(s)}{\sigma(s)}Z(s)\tilde{Z}(s)ds\medskip\\
=\mathbb{E}|\xi|^{2}+\mathbb{E}{\displaystyle\int_{t}^{T}}
\left[\tilde{Y}(s)f(s,\eta(s),Y(s),Z(s))+Y(s)f(s,\eta(s),\tilde{Y}(s),\tilde{Z}(s))\right]ds,~
t\in[\rho,T].
\end{array}
\end{equation}
Thus, from (\ref{equation for Y}), (\ref{equation for tilde Y}) and (\ref{equation for tilde Y and Y}) as well as $\mathbb{E}|Y(t)-\tilde{Y}(t)|^{2}=\mathbb{E}|Y(t)|^{2}-2\mathbb{E}\left[Y(t)\tilde{Y}(t)\right]+\mathbb{E}|\tilde{Y}(t)|^{2}$ we have,
for $t\in[\rho,T]$,
\[
\begin{array}
[c]{ll}
\mathbb{E}|Y(t)-\tilde{Y}(t)|^{2}+2\mathbb{E}{\displaystyle\int_{t}^{T}}\dfrac{\hat{\sigma}(s)}{\sigma(s)}|Z(s)-\tilde{Z}(s)|^{2}ds\medskip\\
=\mathbb{E}{\displaystyle\int_{t}^{T}}
\left[Y(s)-\tilde{Y}(s)\right]\left[f(s,\eta(s),Y(s),Z(s))-f(s,\eta(s),\tilde{Y}(s),\tilde{Z}(s))\right]ds\medskip\\
\leq \mathbb{E}{\displaystyle\int_{t}^{T}}
\left(L|Y(s)-\tilde{Y}(s)|^{2}+L^{2}Ms^{1-2H}|Y(s)-\tilde{Y}(s)|^{2}+\dfrac{1}{M}s^{2H-1}|Z(s)-\tilde{Z}(s)|^{2}\right)ds.
\end{array}
\]
where $M$ is the constant introduced in Remark \ref{remark for sigma}. Then, using Remark \ref{remark for sigma}
\[
\begin{array}
[c]{l}
\mathbb{E}|Y(t)-\tilde{Y}(t)|^{2}+\dfrac{1}{M}\mathbb{E}{\displaystyle\int_{t}^{T}}s^{2H-1}|Z(s)-\tilde{Z}(s)|^{2}ds\medskip\\
\leq \mathbb{E}{\displaystyle\int_{t}^{T}}(L+L^{2}Ms^{1-2H})|Y(s)-\tilde{Y}(s)|^{2}ds
\end{array}
\]
and Gronwall's inequality yields that
\[
\begin{array}
[c]{ll}
\mathbb{E}|Y(t)-\tilde{Y}(t)|^{2}+\dfrac{1}{M}\mathbb{E}{\displaystyle\int_{t}^{T}}s^{2H-1}|Z(s)-\tilde{Z}(s)|^{2}ds=0,~t\in[\rho,T].
\end{array}
\]
Since $\rho>0$ is arbitrary, our proof is complete now.\hfill
\end{proof}
\section{Fractional backward stochastic variational inequality}
Let us now consider the following BSVI driven by a fBm:
\begin{equation}
\left\{
\begin{array}
[c]{l}
-dY(t)+\partial\varphi(Y(t))dt\ni f(t,\eta(t),Y(t),Z(t))dt-Z(t)\delta B^{H}(t),\qquad t\in[0,T],\medskip\\
Y(T)=\xi,
\end{array}
\right.  \label{fBSVI}
\end{equation}
where the coefficients satisfy $(H_{1})$-$(H_{4})$ and $\partial\varphi$ is the subdifferential of
the function $\varphi:\mathbb{R}\rightarrow(-\infty,+\infty]$ satisfying
\begin{itemize}
\item[$(H_{5})$] $\varphi$ is a convex lower semicontinuous (l.s.c.)
function with $\varphi(x)\geq\varphi(0)=0$, for all $x\in\mathbb{R}$ and
$\mathbb{E}|\varphi(\xi)|<\infty$ (Recall that $\xi=g(\eta(T))$).
\end{itemize}
Let us introduce the following notations:
\[
\begin{array}
[c]{l}
Dom~\varphi=\{u\in\mathbb{R}:\varphi(u)<\infty\},\medskip\\
\partial\varphi(u)=\{u^{\ast}\in\mathbb{R}:u^{\ast}(v-u)+\varphi(u)\leq
\varphi(v), \text{ for all } v\in\mathbb{R}\},\medskip\\
Dom(\partial\varphi)=\{u\in\mathbb{R}:\partial\varphi(u)\neq\emptyset
\},\medskip\\
(u,u^{\ast})\in\partial\varphi\Leftrightarrow u\in Dom(\partial\varphi
),u^{\ast}\in\partial\varphi(u).
\end{array}
\]
We know that the multivalued subdifferential operator $\partial\varphi$ is a
monotone operator, i.e.,
\[
(u^{\ast}-v^{\ast})(u-v)\geq0,\;\text{ for all }(u,u^{\ast}),(v,v^{\ast})\in \partial\varphi.
\]
Now, we give the definition of the solution for BSVI (\ref{fBSVI}).
\begin{definition}
A triple $(Y,Z,U)$ is a solution for BSVI (\ref{fBSVI}), if:
\[
\begin{array}
[c]{l}
(a_{1})~  Y, ~U\in\mathcal{\bar{V}}_{T}^{H}\text{ and } Z\in\mathcal{\bar{V}}_{T}^{2H-1/2}\,,\medskip\\
(a_{2})~ (Y(t),U(t))\in\partial\varphi,~dP\otimes dt~a.e.\text{ on
}\Omega\times[0,T],\medskip\\
(a_{3})~ Y(t)+{\displaystyle\int_{t}^{T}}U(s)ds=\xi+{\displaystyle\int
_{t}^{T}}f\left(  s,\eta(s),Y(s),Z(s)\right)ds-{\displaystyle\int_{t}^{T}
}Z(s)\delta B^{H}(s),~a.s., ~ t\in(0,T].
\end{array}
\]
\end{definition}
In this section, our objective is to show the following existence result:
\begin{theorem}\label{existence of fractional BSVI}
Let the assumptions $(H_{1})$-$(H_{5})$ be satisfied. There exists a solution of BSVI (\ref{fBSVI}).
\end{theorem}
\subsection{A priori estimates}
We consider the penalized BSDE by using the Moreau$-$Yosida approximation of $\varphi$:
\begin{equation}\label{penalized fBSVI}
Y^{\varepsilon}(t)+\int_{t}^{T}\nabla\varphi_{\varepsilon}(Y^{\varepsilon
}(s))ds=\xi+\int_{t}^{T}f\left(  s,\eta(s),Y^{\varepsilon}(s),Z^{\varepsilon
}(s)\right)  ds-\int_{t}^{T}Z^{\varepsilon}(s)\delta B^{H}\left(  s\right).
\end{equation}
Recall that the regularization $\varphi_{\varepsilon}$ of $\varphi$ is defined by:
\[
\varphi_{\varepsilon}(u):=\inf\left\{  \dfrac{1}{2\varepsilon}|u-v|^{2}+\varphi(v):v\in\mathbb{R}\right\},~u\in\mathbb{R},~\varepsilon>0.
\]
It is well known that $\varphi_{\varepsilon}$ is a convex function of class $C^{1}$ on $\mathbb{R}$ and
its gradient $\nabla\varphi_{\varepsilon}$ is a Lipschitz function with Lipschitz constant $1/\varepsilon$. Let
\[
J_{\varepsilon}u=u-\varepsilon\nabla\varphi_{\varepsilon}(u),~u\in\mathbb{R}.
\]
For all $u,v\in\mathbb{R}$ and $\varepsilon,\delta>0$, the following properties hold true (see \cite{B-76} and \cite{PR-98}).
\begin{equation*}\label{property of subdifferential operator}
\begin{array}
[c]{ll}
\left(  a\right)   & \varphi_{\varepsilon}(u)=\dfrac{\varepsilon}{2}|\nabla\varphi_{\varepsilon}(u)|^{2}+\varphi(J_{\varepsilon}u),\medskip\\
\left(  b\right)   & \left\vert J_{\varepsilon}u-J_{\varepsilon}v\right\vert\leq|u-v|,\medskip\\
\left(  c\right)   & \nabla\varphi_{\varepsilon}(u)\in\partial\varphi(J_{\varepsilon}u),\medskip\\
\left(  d\right)   & 0\leq\varphi_{\varepsilon}(u)\leq u\nabla\varphi_{\varepsilon}(u),\medskip\\
\left(  e\right)   & \left(  \nabla\varphi_{\varepsilon}(u)-\nabla\varphi_{\delta}(v)\right)  (u-v)\geq-\left(  \varepsilon+\delta\right)
\nabla\varphi_{\varepsilon}(u)\nabla\varphi_{\delta}(v).
\end{array}
\end{equation*}
\begin{theorem}\label{Theorem 1 for fBSVI}
Let the assumptions $(H_{1})$-$(H_{5})$ be satisfied. Then, for all $\varepsilon>0$, the penalized BSDE (\ref{penalized fBSVI}) has a solution $(Y^{\varepsilon},Z^{\varepsilon})\in\mathcal{\bar{V}}_{T}^{1/2}\times\mathcal{\bar{V}}_{T}^{H}$ such that,  for $t\in(0,T]$,
\begin{equation}\label{equation for penelized Y and Z}
\begin{array}
[c]{l}
\mathbb{E}|Y^{\varepsilon}(t)|^{2}+2\mathbb{E}{\displaystyle\int_{t}^{T}}\dfrac{\hat{\sigma}(s)}{\sigma(s)}| Z^{\varepsilon}(s)|^{2}ds
=\mathbb{E}|\xi|^{2}
+2\mathbb{E}{\displaystyle\int_{t}^{T}}Y^{\varepsilon}(s)f(s,\eta(s),Y^{\varepsilon}(s),Z^{\varepsilon}(s))ds\medskip\\
\qquad\qquad\qquad\qquad\qquad\qquad
-2\mathbb{E}{\displaystyle\int_{t}^{T}}Y^{\varepsilon}(s)\nabla\varphi_{\varepsilon}(Y^{\varepsilon}(s))ds.
\end{array}
\end{equation}
\end{theorem}
\begin{proof}
In order to use Theorem \ref{uniqueness theorem}, we mollify $\nabla\varphi_{\varepsilon}$ in a standard way:
\[
\left(\nabla\varphi_{\varepsilon}\right)^{\alpha}(x):={\displaystyle\int_{\mathbb{R}}}\nabla\varphi_{\varepsilon}(x-\alpha u)\lambda(u)du,~x\in\mathbb{R}, \text{ where } \lambda(u)=\dfrac{1}{\sqrt{2\pi}}e^{-\frac{u^{2}}{2}},~u\in\mathbb{R}.
\]
Considering that $\varphi_{\varepsilon}$ is convex and $\nabla\varphi_{\varepsilon}$ is Lipschitz continuous with Lipschitz constant $1/\varepsilon$, $\left(\nabla\varphi_{\varepsilon}\right)^{\alpha}$ has the following properties for $x_{1},x_{2}\in\mathbb{R}$ and $\alpha,\alpha_{1},\alpha_{2}>0$:
\begin{itemize}
\item[$(i)$]$\left(\nabla\varphi_{\varepsilon}\right)^{\alpha}$ belongs to $C^{1}_{pol}(\mathbb{R})$, and is convex;
\item[$(ii)$]$\left|\left(\nabla\varphi_{\varepsilon}\right)^{\alpha_{1}}(x_{1})-\left(\nabla\varphi_{\varepsilon}\right)^{\alpha_{2}}(x_{2})\right|
\leq\dfrac{1}{\varepsilon}|x_{1}-x_{2}|+\dfrac{1}{\varepsilon}\sqrt{\dfrac{2}{\pi}}|\alpha_{1}-\alpha_{2}|.$
\end{itemize}
Now, we consider the following mollified BSDE
\begin{equation}\label{mollified penalized fBSVI}
Y^{\varepsilon,\alpha}(t)+\int_{t}^{T}\left(\nabla\varphi_{\varepsilon}\right)^{\alpha}(Y^{\varepsilon,\alpha}(s))ds=\xi+\int_{t}^{T}f\left(  s,\eta(s),Y^{\varepsilon,\alpha}(s),Z^{\varepsilon,\alpha}(s)\right)ds-\int_{t}^{T}Z^{\varepsilon,\alpha}(s)\delta B^{H}\left(  s\right).
\end{equation}
From Theorem \ref{uniqueness theorem}, we obtain that (\ref{mollified penalized fBSVI}) admits a unique solution
$\left(Y^{\varepsilon,\alpha},Z^{\varepsilon,\alpha}\right)$ in $\mathcal{S}_{f,\varepsilon,\alpha}:=\mathcal{S}_{f-\left(\nabla\varphi_{\varepsilon}\right)^{\alpha}}\subset\mathcal{\bar{V}}_{T}^{1/2}\times
\mathcal{\bar{V}}_{T}^{H}$,
This solution $\left(Y^{\varepsilon,\alpha},Z^{\varepsilon,\alpha}\right)$  can be approximated by the sequence $(Y^{k,\varepsilon,\alpha},Z^{k,\varepsilon,\alpha})\in\mathcal{V}_{T}\times\mathcal{V}_{T}$, $k\ge0$ constructed by the following method:
Define $(Y^{k,\varepsilon,\alpha},Z^{k,\varepsilon,\alpha})$, $k\ge0$ recursively: $Y^{0,\varepsilon,\alpha}=\chi(t,\eta(t))$,
$Z^{0,\varepsilon,\alpha}=\psi(t,\eta(t))$ for $\chi,\psi\in C_{pol}^{1,3}([0,T]\times\mathbb{R})$ with
$\dfrac{\partial \chi}{\partial t},\dfrac{\partial \psi}{\partial t}\in C_{pol}^{0,1}([0,T]\times\mathbb{R})$, and let $(Y^{k+1,\varepsilon,\alpha},Z^{k+1,\varepsilon,\alpha})\in\mathcal{V}_{T}\times\mathcal{V}_{T}$ be the unique solution of the BSDE
\begin{equation}\label{approximate BSDE 1}
\begin{array}
[c]{l}
Y^{k+1,\varepsilon,\alpha}(t)+{\displaystyle\int_{t}^{T}}\left(\nabla\varphi_{\varepsilon}\right)^{\alpha}(Y^{k,\varepsilon,\alpha}(s))ds=\xi+
{\displaystyle\int_{t}^{T}}f\left(  s,\eta(s),Y^{k,\varepsilon,\alpha}(s),Z^{k,\varepsilon,\alpha}(s)\right)  ds\medskip\\
\qquad\qquad\qquad\qquad\qquad
-{\displaystyle\int_{t}^{T}}Z^{k+1,\varepsilon,\alpha}(s)\delta B^{H}(s),~t\in[0,T].
\end{array}
\end{equation}
Similar to (\ref{convergent property}), we have
\begin{equation}\label{convergence property 2}
\begin{array}
[c]{ll}
\lim\limits_{k\rightarrow\infty}\mathbb{E}|Y^{k,\varepsilon,\alpha}(t)-Y^{\varepsilon,\alpha}(t)|^{2}=0,~t\in[0,T],\text{ and } \medskip\\
\lim\limits_{k\rightarrow\infty}\mathbb{E}{\displaystyle\int_{0}^{T}}|Y^{k,\varepsilon,\alpha}(s)-Y^{\varepsilon,\alpha}(s)|^{2}ds
+\mathbb{E}{\displaystyle\int_{0}^{T}}s^{2H-1}|Z^{k,\varepsilon,\alpha}(s)-Z^{\varepsilon,\alpha}(s)|^{2}ds=0.
\end{array}
\end{equation}
Moreover, analogously to (\ref{equation for Y}), we show that for arbitrary $\rho>0$,
\begin{equation}\label{mollify equation for Y}
\begin{array}
[c]{l}
\mathbb{E}|Y^{\varepsilon,\alpha}(t)|^{2}+2\mathbb{E}{\displaystyle\int_{t}^{T}}{\displaystyle\dfrac{\hat{\sigma}(s)}{\sigma(s)}}
|Z^{\varepsilon,\alpha}(s)|^{2}ds=\mathbb{E}|\xi|^{2}
-2\mathbb{E}{\displaystyle\int_{t}^{T}}Y^{\varepsilon,\alpha}(s)\left(\nabla\varphi_{\varepsilon}\right)^{\alpha}(Y^{\varepsilon,\alpha}(s))ds\medskip\\
\qquad\qquad\qquad\qquad
+2\mathbb{E}{\displaystyle\int_{t}^{T}}Y^{\varepsilon,\alpha}(s)f(s,\eta(s),Y^{\varepsilon,\alpha}(s),Z^{\varepsilon,\alpha}(s))ds,
~t\in[\rho,T],
\end{array}
\end{equation}
and
\[
\begin{array}
[c]{l}
\mathbb{E}|\Delta Y^{\varepsilon,\alpha_{1},\alpha_{2}}(t)|^{2}+2\mathbb{E}{\displaystyle\int_{t}^{T}}{\displaystyle\dfrac{\hat{\sigma}(s)}{\sigma(s)}}
|\Delta Z^{\varepsilon,\alpha_{1},\alpha_{2}}(s)|^{2}ds
=2\mathbb{E}{\displaystyle\int_{t}^{T}}\Delta Y^{\varepsilon,\alpha_{1},\alpha_{2}}(s)\Delta f^{\varepsilon,\alpha_{1},\alpha_{2}}(s)ds\medskip\\
-2\mathbb{E}{\displaystyle\int_{t}^{T}}\Delta Y^{\varepsilon,\alpha_{1},\alpha_{2}}(s)\left(
\left(\nabla\varphi_{\varepsilon}\right)^{\alpha_{1}}(Y^{\varepsilon,\alpha_{1}}(s))
-\left(\nabla\varphi_{\varepsilon}\right)^{\alpha_{2}}(Y^{\varepsilon,\alpha_{2}}(s))\right)ds,
~t\in[\rho,T],
\end{array}
\]
where $\Delta Y^{\varepsilon,\alpha_{1},\alpha_{2}}(s)=Y^{\varepsilon,\alpha_{1}}(s)-Y^{\varepsilon,\alpha_{2}}(s)$,
~$\Delta Z^{\varepsilon,\alpha_{1},\alpha_{2}}(s)=Z^{\varepsilon,\alpha_{1}}(s)-Z^{\varepsilon,\alpha_{2}}(s)$ and\\
$\Delta f^{\varepsilon,\alpha_{1},\alpha_{2}}(s)=f(s,\eta(s),Y^{\varepsilon,\alpha_{1}}(s),Z^{\varepsilon,\alpha_{1}}(s))
-f(s,\eta(s),Y^{\varepsilon,\alpha_{2}}(s),Z^{\varepsilon,\alpha_{2}}(s)),~s\in[0,T]$.\\
Then
\[
\begin{array}
[c]{l}
\mathbb{E}|\Delta Y^{\varepsilon,\alpha_{1},\alpha_{2}}(t)|^{2}+2\mathbb{E}{\displaystyle\int_{t}^{T}}{\displaystyle\dfrac{\hat{\sigma}(s)}{\sigma(s)}}
|\Delta Z^{\varepsilon,\alpha_{1},\alpha_{2}}(s)|^{2}ds
\leq\mathbb{E}{\displaystyle\int_{t}^{T}}(2L+L^{2}Ms^{1-2H})|\Delta Y^{\varepsilon,\alpha_{1},\alpha_{2}}(s)|^{2}ds\medskip\\
+\mathbb{E}{\displaystyle\int_{t}^{T}}\dfrac{1}{M}s^{2H-1}|\Delta Z^{\varepsilon,\alpha_{1},\alpha_{2}}(s)|^{2}ds
+\dfrac{3}{\varepsilon}\mathbb{E}{\displaystyle\int_{t}^{T}}|\Delta Y^{\varepsilon,\alpha_{1},\alpha_{2}}(s)|^{2}ds+\dfrac{2T}{\pi}|\alpha_{1}-\alpha_{2}|^{2},
~t\in[\rho,T],
\end{array}
\]
(by using the property for $\left(\nabla\varphi_{\varepsilon}\right)^{\alpha}$) where $M$ is the constant given by Remark \ref{remark for sigma}. Then, using (\ref{properties for sigma}) we obtain
\[
\begin{array}
[c]{l}
\mathbb{E}|\Delta Y^{\varepsilon,\alpha_{1},\alpha_{2}}(t)|^{2}+\dfrac{1}{M}\mathbb{E}{\displaystyle\int_{t}^{T}}s^{2H-1}|\Delta Z^{\varepsilon,\alpha_{1},\alpha_{2}}(s)|^{2}ds\medskip\\
\leq\dfrac{2T}{\pi}|\alpha_{1}-\alpha_{2}|^{2}+\mathbb{E}{\displaystyle\int_{t}^{T}}(2L+\dfrac{3}{\varepsilon}+L^{2}Ms^{1-2H})|\Delta Y^{\varepsilon,\alpha_{1},\alpha_{2}}(s)|^{2}ds,
~t\in[\rho,T],
\end{array}
\]
and Gronwall's inequality yields that
\[
\begin{array}
[c]{l}
\mathbb{E}|\Delta Y^{\varepsilon,\alpha_{1},\alpha_{2}}(t)|^{2}+\dfrac{1}{M}\mathbb{E}{\displaystyle\int_{t}^{T}}s^{2H-1}|\Delta Z^{\varepsilon,\alpha_{1},\alpha_{2}}(s)|^{2}ds\leq M_{\varepsilon,L,T}|\alpha_{1}-\alpha_{2}|^{2},
~t\in[\rho,T],
\end{array}
\]
where $M_{\varepsilon,L,T}$ is a constant depending only on $\varepsilon,L,T$ but independent of $\rho>0$. Consequently, taking into account the arbitrariness of $\rho>0$, there exists a couple of processes $(Y^{\varepsilon},Z^{\varepsilon})$ with $Y^{\varepsilon}\mathbf{1}_{[\rho,T]}\in \bar{\mathcal{V}}_{T}^{1/2}$, $Z^{\varepsilon}\mathbf{1}_{[\rho,T]}\in \bar{\mathcal{V}}_{T}^{H}$ for all $\rho>0$, such that
\begin{equation}\label{equation 6}
\begin{array}
[c]{l}
\lim\limits_{\alpha\rightarrow0}\mathbb{E}|Y^{\varepsilon,\alpha}(t)-Y^{\varepsilon}(t)|^{2}=0, \text{ for all }t\in[\rho,T],\medskip\\
\lim\limits_{\alpha\rightarrow0}\mathbb{E}{\displaystyle\int_{t}^{T}}|Y^{\varepsilon,\alpha}(s)-Y^{\varepsilon}(s)|^{2}ds
+\mathbb{E}{\displaystyle\int_{t}^{T}}s^{2H-1}|Z^{\varepsilon,\alpha}(s)-Z^{\varepsilon}(s)|^{2}ds=0, \text{ for all } t\in[\rho,T],
\end{array}
\end{equation}
Now let $\alpha\rightarrow0$, and by using (\ref{mollified penalized fBSVI}) and a similar discussion as in Theorem \ref{main result Hu&Peng}, we obtain that $Z^{\varepsilon}\mathbf{1}_{[t,T]}\in Dom(\delta)$, $t\in(0,T]$, and
\[
Y^{\varepsilon}(t)+\int_{t}^{T}\nabla\varphi_{\varepsilon}(Y^{\varepsilon}(s))ds
=\xi+\int_{t}^{T}f\left(  s,\eta(s),Y^{\varepsilon}(s),Z^{\varepsilon}(s)\right)ds-\int_{t}^{T}Z^{\varepsilon}(s)\delta B^{H}(s),~t\in(0,T].
\]
Moreover, taking $\alpha\rightarrow0$ in (\ref{mollify equation for Y}) yields that
\[
\begin{array}
[c]{l}
\mathbb{E}|Y^{\varepsilon}(t)|^{2}+2\mathbb{E}{\displaystyle\int_{t}^{T}}\dfrac{\hat{\sigma}(s)}{\sigma(s)}| Z^{\varepsilon}(s)|^{2}ds
=\mathbb{E}|\xi|^{2}
+2\mathbb{E}{\displaystyle\int_{t}^{T}}Y^{\varepsilon}(s)f(s,\eta(s),Y^{\varepsilon}(s),Z^{\varepsilon}(s))ds\medskip\\
\qquad\qquad\qquad\qquad\qquad\qquad
-2\mathbb{E}{\displaystyle\int_{t}^{T}}Y^{\varepsilon}(s)\nabla\varphi_{\varepsilon}(Y^{\varepsilon}(s))ds,
~t\in(0,T].
\end{array}
\]
\hfill
\end{proof}

The next three propositions provide a priori estimates for the sequence $(Y^{\varepsilon},Z^{\varepsilon})$, $\varepsilon>0$.

\begin{proposition}\label{bound of Y and Z}
Let the assumptions $(H_{1})$-$(H_{5})$ be
satisfied.  Let $(Y^{\varepsilon},Z^{\varepsilon})$ be the solution constructed in the proof of Theorem \ref{Theorem 1 for fBSVI}.
Then there exists a positive constant $C$ independent of $\varepsilon>0$, such that, for all $t\in[0,T]$,
\[
\mathbb{E}|Y^{\varepsilon}(t)|^{2}+\mathbb{E}\int_{t}^{T}s^{2H-1}%
|Z^{\varepsilon}(s)|^{2}ds\leq C~\Gamma_{1}(T),
\]
where $\Gamma_{1}(T)=\mathbb{E}\big[|\xi|^{2}+{\displaystyle\int_{0}^{T}}|\eta(s)|^{2}ds+{\displaystyle\int_{0}^{T}}|f(s,0,0,0)|^{2}ds\big]$.
\end{proposition}

\begin{proof}
From (\ref{equation for penelized Y and Z}), (\ref{properties for sigma}-$b$) and $u\nabla\varphi_{\varepsilon}(u)\geq0$, for all $u\in\mathbb{R}$,
we have, for $t\in(0,T]$,
\[
\mathbb{E}\left\vert Y^{\varepsilon}(t)\right\vert ^{2}+\dfrac{2}{M}\mathbb{E}{\int_{t}^{T}}s^{2H-1}\left\vert Z^{\varepsilon}(s)\right\vert^{2}ds
\leq\mathbb{E}\left\vert \xi\right\vert ^{2}+2{\int_{t}^{T}}
\mathbb{E}\left[  Y^{\varepsilon}(s)f\left(s,\eta(s),Y^{\varepsilon}(s),Z^{\varepsilon}(s)\right)\right]ds.
\]
On the other hand, from assumption $(H_{3})$ and Schwartz's inequality, we obtain
\[
\begin{array}
[c]{l}
2Y^{\varepsilon}(s)f(s,\eta(s),Y^{\varepsilon}(s),Z^{\varepsilon}
(s))\medskip\\
\quad\leq2\left\vert Y^{\varepsilon}(s)\right\vert \big(\left\vert f\left(s,0,0,0\right)\right\vert
+L\left\vert \eta(s)\right\vert +L\left\vert Y^{\varepsilon}(s)\right\vert +L\left\vert Z^{\varepsilon}(s)\right\vert \big)\medskip\\
\quad\leq\left\vert f\left( s,0,0,0\right)  \right\vert ^{2}+
\left\vert \eta(s)\right\vert ^{2}+\left(1+L^{2}+2L+L^{2}M\dfrac{1}{s^{2H-1}}\right)\left\vert Y^{\varepsilon}(s)\right\vert ^{2}
+\dfrac{1}{M}s^{2H-1}\left\vert Z^{\varepsilon}(s)\right\vert ^{2}.
\end{array}
\]
Then,
\[
\begin{array}
[c]{r}
\mathbb{E}\left\vert Y^{\varepsilon}(t)\right\vert ^{2}
+\dfrac{1}{M}\mathbb{E}\displaystyle{\int_{t}^{T}}s^{2H-1}\left\vert Z^{\varepsilon}(s)\right\vert^{2}ds
\leq\mathbb{E}|\xi|^{2}+{\displaystyle\int_{t}^{T}}\mathbb{E}\left\vert f\left(s,0,0,0\right)\right\vert ^{2}ds
+{\displaystyle\int_{t}^{T}}\mathbb{E}\left\vert \eta(s)\right\vert ^{2}ds\medskip\\
+{\displaystyle\int_{t}^{T}}\left(  1+L^{2}+2L+L^{2}M\dfrac{1}{s^{2H-1}}\right)\mathbb{E}\left\vert Y^{\varepsilon}(s)\right\vert ^{2}ds.
\end{array}
\]
Therefore, by Gronwall's inequality, we deduce
that
\[
\begin{array}
[c]{l}
\mathbb{E}\left\vert Y^{\varepsilon}(t)\right\vert ^{2}
+\dfrac{1}{M}\mathbb{E}\displaystyle{\int_{t}^{T}}s^{2H-1}\left\vert Z^{\varepsilon}(s)\right\vert ^{2}ds\medskip\\
\quad\leq\Gamma_{1}(T)\exp\bigg[\left(  1+L^{2}+2L\right)  \left(
T-t\right)  +L^{2}M\dfrac{T^{2-2H}-t^{2-2H}}{2-2H}\bigg],
\end{array}
\]
which completes the proof.\hfill
\end{proof}

\begin{proposition}\label{bound of approxim phi}
Let the assumptions $(H_{1})$-$(H_{5})$ be satisfied. Then there exists a positive constant $C$ such that, for all $t\in[0,T]$,
\[
\begin{array}
[c]{rl}
(i) & \mathbb{E}{\displaystyle\int_{t}^{T}}s^{2H-1}\left\vert \nabla\varphi_{\varepsilon}\left(Y^{\varepsilon}(s)\right)\right\vert^{2}ds
\leq C~\Gamma_{2}(T),\medskip\\
(ii) & t^{2H-1}\mathbb{E}\left[\varphi\left( J_{\varepsilon}\left( Y^{\varepsilon}(t)\right)\right)\right]\leq C~\Gamma_{2}(T),\medskip\\
(iii) & t^{2H-1}\mathbb{E}\left[\left\vert Y^{\varepsilon}(t)-J_{\varepsilon}\left(Y^{\varepsilon}(t)\right)\right\vert ^{2}\right]
\leq\varepsilon C~\Gamma_{2}(T),
\end{array}
\]
where $\Gamma_{2}(T)=\mathbb{E}\big[\left\vert \xi\right\vert ^{2}
+\varphi(\xi)+{\displaystyle\int_{0}^{T}}|\eta(s)|^{2}ds+{\displaystyle\int_{0}^{T}}\left\vert f(s,0,0,0)\right\vert ^{2}ds\big]$.
\end{proposition}

In order to obtain the above proposition, it is essential to use the following
fractional stochastic subdifferential inequality:

\begin{lemma}
Let $\psi:\mathbb{R}\rightarrow\mathbb{R}_{+}$ be a convex $C^{1}$ function which derivative $\nabla\psi$ is a Lipschitz function
(with Lipschitz constant denoted by $K$). Then, for all $t\in(0,T]$, $P$-a.s.
\begin{equation}\label{fractional subdiff ineq}
\begin{array}
[c]{l}
t^{2H-1}\mathbb{E}\left[  \psi\left(  Y^{\varepsilon}(t)\right)  \right]
+\mathbb{E}{\displaystyle\int_{t}^{T}}s^{2H-1}\nabla\psi\left( Y^{\varepsilon}(s)\right)\nabla \varphi_{\varepsilon}( Y^{\varepsilon}(s))ds \medskip\\
\leq T^{2H-1}\mathbb{E}\left[  \psi(\xi)\right]+\mathbb{E}
{\displaystyle\int_{t}^{T}}s^{2H-1}\nabla\psi\left(  Y^{\varepsilon}(s) \right)f(s,\eta(s),Y^{\varepsilon}(s),Z^{\varepsilon}(s))ds.
\end{array}
\end{equation}
where $(Y^{\varepsilon},Z^{\varepsilon})$ is the solution constructed in the proof of Theorem \ref{Theorem 1 for fBSVI}.
\end{lemma}
\begin{proof}
We first show that
\[
\begin{array}
[c]{l}
t^{2H-1}\mathbb{E}\left[  \psi\left(  Y^{k+1,\varepsilon,\alpha}\left(  t\right)  \right)  \right]
+\mathbb{E}{\displaystyle\int_{t}^{T}}s^{2H-1}\nabla\psi\left(  Y^{k+1,\varepsilon,\alpha}(s) \right)(\nabla \varphi_{\varepsilon})^{\alpha}( Y^{k,\varepsilon,\alpha}(s))ds \medskip\\
\leq T^{2H-1}\mathbb{E}\left[  \psi\left(\xi\right)  \right]
+\mathbb{E}{\displaystyle\int_{t}^{T}}s^{2H-1}\nabla\psi\left(  Y^{k+1,\varepsilon,\alpha}(s) \right)f(s,\eta(s),Y^{k,\varepsilon,\alpha}(s),Z^{k,\varepsilon,\alpha}(s))ds ,
\end{array}
\]
where $(Y^{k,\varepsilon,\alpha},Z^{k,\varepsilon,\alpha})\in \mathcal{V}_{T}\times \mathcal{V}_{T}$ is defined through (\ref{approximate BSDE 1}).
We mollify the function $\psi$ by setting, for $\theta>0$, $\psi^{\theta}(x):={\displaystyle\int_{\mathbb{R}}}\psi(x-\theta u)\lambda(u)du,~x\in\mathbb{R}$, where $\lambda(u)=\dfrac{1}{\sqrt{2\pi}}e^{-\frac{u^{2}}{2}},~u\in\mathbb{R}$.
From the convexity of $\psi$, it follows that $\psi^{\theta}$ is convex. Moreover $\psi^{\theta}\ge0$.
The generalized It\^{o} formula [see (\ref{1 general Ito's formula for the divergence integral}) in Remark \ref{remark for genelized ito formula}] yields
\begin{equation}\label{Ito formula from Nualart book}
\begin{array}
[c]{l}
T^{2H-1}\psi^{\theta}(\xi)=t^{2H-1}\psi^{\theta}(Y^{k+1,\varepsilon,\alpha}(t))
+(2H-1){\displaystyle\int_{t}^{T}}s^{2H-2}\psi^{\theta}(Y^{k+1,\varepsilon,\alpha}(s))ds\medskip\\
\qquad-{\displaystyle\int_{t}^{T}}s^{2H-1}\nabla\psi^{\theta}(Y^{k+1,\varepsilon,\alpha}(s))f\left(  s,\eta(s),Y^{k,\varepsilon,\alpha}(s),Z^{k,\varepsilon,\alpha}(s)\right) ds\medskip\\
\qquad+{\displaystyle\int_{t}^{T}}s^{2H-1}\nabla\psi^{\theta}(Y^{k+1,\varepsilon,\alpha}(s))
\left(\nabla\varphi_{\varepsilon}\right)^{\alpha}(Y^{k,\varepsilon,\alpha}(s))ds\medskip\\
\qquad+{\displaystyle\int_{t}^{T}}s^{2H-1}\nabla\psi^{\theta}(Y^{k+1,\varepsilon,\alpha}(s))Z^{k+1,\varepsilon,\alpha}(s)\delta B^{H}(s)\medskip\\
\qquad+{\displaystyle\int_{t}^{T}}s^{2H-1}D^{2}_{xx}\psi^{\theta}(Y^{k+1,\varepsilon,\alpha}(s))Z^{k+1,\varepsilon,\alpha}(s)
\mathbb{D}^{H}_{s}Y^{k+1,\varepsilon,\alpha}(s)ds.
\end{array}
\end{equation}

Now taking the expectation in (\ref{Ito formula from Nualart book}), by using $\psi^{\theta}\ge0$, the convexity of $\psi^{\theta}$ and the fact that $\mathbb{D}^{H}_{s}Y^{k+1,\varepsilon,\alpha}(s)=\dfrac{\hat{\sigma}(s)}{\sigma(s)}Z^{k+1,\varepsilon,\alpha}(s)$, we have
\begin{equation}\label{fractional subdiff ineq 3}
\begin{array}
[c]{l}
t^{2H-1}\mathbb{E}\left[\psi^{\theta}(Y^{k+1,\varepsilon,\alpha}(t))\right]
+\mathbb{E}{\displaystyle\int_{t}^{T}}s^{2H-1}\nabla\psi^{\theta}(Y^{k+1,\varepsilon,\alpha}(s))
\left(\nabla\varphi_{\varepsilon}\right)^{\alpha}(Y^{k,\varepsilon,\alpha}(s))ds\medskip\\
\leq T^{2H-1}\mathbb{E}\psi^{\theta}(\xi)
+\mathbb{E}{\displaystyle\int_{t}^{T}}s^{2H-1}\nabla\psi^{\theta}(Y^{k+1,\varepsilon,\alpha}(s))f\left(  s,\eta(s),Y^{k,\varepsilon,\alpha}(s),Z^{k,\varepsilon,\alpha}(s)\right) ds.
\end{array}
\end{equation}
Considering that
\[
\begin{array}
[c]{l}
\left|\nabla\psi^{\theta}(x)-\nabla\psi(x)\right|
=\left|\nabla{\displaystyle\int_{\mathbb{R}}}\psi(x-\theta u)\lambda(u)du-\nabla\psi(x)\right|\medskip\\
\leq{\displaystyle\int_{\mathbb{R}}}\left|\nabla\psi(x-\theta u)-\nabla\psi(x)\right|\lambda(u)du
\leq \sqrt{\dfrac{2}{\pi}}K|\theta|,
\end{array}
\]
we have
\[
\begin{array}
[c]{l}
\mathbb{E}\Bigg|{\displaystyle\int_{t}^{T}}s^{2H-1}\nabla\psi^{\theta}(Y^{k+1,\varepsilon,\alpha}(s))
\left(\nabla\varphi_{\varepsilon}\right)^{\alpha}(Y^{k,\varepsilon,\alpha}(s))ds\medskip\\
\qquad\qquad-{\displaystyle\int_{t}^{T}}s^{2H-1}\nabla\psi(Y^{k+1,\varepsilon,\alpha}(s))
\left(\nabla\varphi_{\varepsilon}\right)^{\alpha}(Y^{k,\varepsilon,\alpha}(s))ds\Bigg|\medskip\\
\leq T^{2H-1}\sqrt{\dfrac{2}{\pi}}K|\theta|\mathbb{E}{\displaystyle\int_{t}^{T}}
\left|\left(\nabla\varphi_{\varepsilon}\right)^{\alpha}(Y^{k,\varepsilon,\alpha}(s))\right|ds\rightarrow0, \text{ as }\theta\rightarrow0.
\end{array}
\]
Similarly, we get
\[
\begin{array}
[c]{l}
\mathbb{E}{\displaystyle\int_{t}^{T}}s^{2H-1}\nabla\psi^{\theta}(Y^{k+1,\varepsilon,\alpha}(s))
f\left(  s,\eta(s),Y^{k,\varepsilon,\alpha}(s),Z^{k,\varepsilon,\alpha}(s)\right) ds\medskip\\
\rightarrow\mathbb{E}{\displaystyle\int_{t}^{T}}s^{2H-1}\nabla\psi(Y^{k+1,\varepsilon,\alpha}(s))
f\left(  s,\eta(s),Y^{k,\varepsilon,\alpha}(s),Z^{k,\varepsilon,\alpha}(s)\right) ds
\rightarrow0, \text{ as }\theta\rightarrow0.
\end{array}
\]
Moreover, using Fatou's Lemma (recalling that $\psi\ge0$), we obtain
\[
\mathbb{E}\left[\psi(Y^{k+1,\varepsilon,\alpha}(t))\right]
=\mathbb{E}\left[\liminf\limits_{\theta\rightarrow0}\psi^{\theta}(Y^{k+1,\varepsilon,\alpha}(t))\right]
\leq\liminf\limits_{\theta\rightarrow0}\mathbb{E}\left[\psi^{\theta}(Y^{k+1,\varepsilon,\alpha}(t))\right].
\]
On the other hand, we know that $\psi$ is quadratic growth; therefore there exists a suitable constant $C$, such that
\[
|\psi^{\theta}(x)|\leq{\displaystyle\int_{\mathbb{R}}}|\psi(x-\theta u)|\lambda(u)du\leq C(1+x^{2}+\theta^{2}).
\]
From  $(H_{4})$ and (\ref{Lp estimates}), it follows $\sup\limits_{\theta\leq 1}\mathbb{E}\left[|\psi^{\theta}(\xi)|\right]^{2}<\infty$, which implies that $\{\psi^{\theta}(\xi)\}_{\theta\leq1}$ is uniformly integrable. Then, considering that $\psi^{\theta}(\xi)\xrightarrow[P-a.s.]{\theta\rightarrow0}\psi(\xi)$, we have
$\mathbb{E}[\psi^{\theta}(\xi)]\rightarrow\mathbb{E}[\psi(\xi)]$.
Consequently, letting $\theta\rightarrow0$ in (\ref{fractional subdiff ineq 3}) we have
\begin{equation}\label{fractional subdiff ineq 4}
\begin{array}
[c]{l}
t^{2H-1}\mathbb{E}\left[\psi(Y^{k+1,\varepsilon,\alpha}(t))\right]
+\mathbb{E}{\displaystyle\int_{t}^{T}}s^{2H-1}\nabla\psi(Y^{k+1,\varepsilon,\alpha}(s))
\left(\nabla\varphi_{\varepsilon}\right)^{\alpha}(Y^{k,\varepsilon,\alpha}(s))ds\medskip\\
\leq T^{2H-1}\mathbb{E}\left[\psi(\xi)\right]
+\mathbb{E}{\displaystyle\int_{t}^{T}}s^{2H-1}\nabla\psi(Y^{k+1,\varepsilon,\alpha}(s))f\left(  s,\eta(s),Y^{k,\varepsilon,\alpha}(s),Z^{k,\varepsilon,\alpha}(s)\right) ds.
\end{array}
\end{equation}
Recalling that [see (\ref{convergence property 2})]
\[
\begin{array}
[c]{ll}
\lim\limits_{k\rightarrow\infty}\mathbb{E}|Y^{k,\varepsilon,\alpha}(t)-Y^{\varepsilon,\alpha}(t)|^{2}=0,~t\in[0,T],\text{ and } \medskip\\
\lim\limits_{k\rightarrow\infty}\mathbb{E}{\displaystyle\int_{0}^{T}}|Y^{k,\varepsilon,\alpha}(s)-Y^{\varepsilon,\alpha}(s)|^{2}ds
+\mathbb{E}{\displaystyle\int_{0}^{T}}s^{2H-1}|Z^{k,\varepsilon,\alpha}(s)-Z^{\varepsilon,\alpha}(s)|^{2}ds=0,
\end{array}
\]
it follows that $Y^{k,\varepsilon,\alpha}(t)\xrightarrow[P-a.s.]{k\rightarrow\infty} Y^{\varepsilon,\alpha}(t)$, for all $t\in[0,T]$ and then also $\psi(Y^{k,\varepsilon,\alpha}(t))\xrightarrow[P-a.s.]{k\rightarrow\infty}\psi(Y^{\varepsilon,\alpha}(t))$, for all $t\in[0,T]$.
Thus, using Fatou's Lemma once again (recalling that $\psi\ge0$), we obtain
\[
\mathbb{E}\left[\psi(Y^{\varepsilon,\alpha}(t))\right]
\leq\liminf\limits_{k\rightarrow\infty}\mathbb{E}\left[\psi(Y^{k+1,\varepsilon,\alpha}(t))\right].
\]
Considering that $\nabla\psi$ and $\left(\nabla\varphi_{\varepsilon}\right)^{\alpha}$ are Lipschitz
with the Lipschitz constant $K$ and $1/\varepsilon$, respectively, we get
\[
\begin{array}
[c]{l}
\mathbb{E}\Bigg|{\displaystyle\int_{t}^{T}}s^{2H-1}\nabla\psi(Y^{k+1,\varepsilon,\alpha}(s))
\left(\nabla\varphi_{\varepsilon}\right)^{\alpha}(Y^{k,\varepsilon,\alpha}(s))ds\medskip\\
\qquad\qquad-{\displaystyle\int_{t}^{T}}s^{2H-1}\nabla\psi(Y^{\varepsilon,\alpha}(s))
\left(\nabla\varphi_{\varepsilon}\right)^{\alpha}(Y^{\varepsilon,\alpha}(s))ds\Bigg|\medskip\\
\leq T^{2H-1}
\mathbb{E}{\displaystyle\int_{t}^{T}}\left|\nabla\psi(Y^{k+1,\varepsilon,\alpha}(s))-\nabla\psi(Y^{\varepsilon,\alpha}(s))\right|
\left|\left(\nabla\varphi_{\varepsilon}\right)^{\alpha}(Y^{k,\varepsilon,\alpha}(s))\right|ds\medskip\\
\qquad\qquad+T^{2H-1}\mathbb{E}{\displaystyle\int_{t}^{T}}\left|\nabla\psi(Y^{\varepsilon,\alpha}(s))\right|
\left|\left(\nabla\varphi_{\varepsilon}\right)^{\alpha}(Y^{k,\varepsilon,\alpha}(s))
-\left(\nabla\varphi_{\varepsilon}\right)^{\alpha}(Y^{\varepsilon,\alpha}(s))\right|ds\medskip\\
\leq
KT^{2H-1}\mathbb{E}{\displaystyle\int_{t}^{T}}\left|Y^{k+1,\varepsilon,\alpha}(s)-Y^{\varepsilon,\alpha}(s)\right|
\left|\left(\nabla\varphi_{\varepsilon}\right)^{\alpha}(Y^{k,\varepsilon,\alpha}(s))\right|ds\medskip\\
\qquad\qquad+\dfrac{1}{\varepsilon}T^{2H-1}\mathbb{E}{\displaystyle\int_{t}^{T}}\left|\nabla\psi(Y^{\varepsilon,\alpha}(s))\right|
\left|Y^{k,\varepsilon,\alpha}(s)-Y^{\varepsilon,\alpha}(s)\right|ds\rightarrow0, ~\text{ as }k\rightarrow\infty.
\end{array}
\]
Indeed, using that the functions $\nabla\psi$ and $\left(\nabla\varphi_{\varepsilon}\right)^{\alpha}$ are of linear growth, $\mathbb{E}{\displaystyle\int_{t}^{T}}|Y^{\varepsilon,\alpha}(s)|^{2}ds
+\mathbb{E}{\displaystyle\int_{t}^{T}}|Y^{k,\varepsilon,\alpha}(s)|^{2}ds\leq C$, $k\ge1$, and $\lim\limits_{k\rightarrow\infty}\mathbb{E}{\displaystyle\int_{0}^{T}}|Y^{k,\varepsilon,\alpha}(s)-Y^{\varepsilon,\alpha}(s)|^{2}ds=0$,
we can obtain the above convergence with the help of H\"{o}lder inequality. Similarly, we show
\[
\begin{array}
[c]{l}
\mathbb{E}{\displaystyle\int_{t}^{T}}s^{2H-1}\nabla\psi(Y^{k+1,\varepsilon,\alpha}(s))
f\left(  s,\eta(s),Y^{k,\varepsilon,\alpha}(s),Z^{k,\varepsilon,\alpha}(s)\right) ds\medskip\\
\rightarrow\mathbb{E}{\displaystyle\int_{t}^{T}}s^{2H-1}\nabla\psi(Y^{\varepsilon,\alpha}(s))
f\left(  s,\eta(s),Y^{\varepsilon,\alpha}(s),Z^{\varepsilon,\alpha}(s)\right) ds
\rightarrow0, \text{ as }k\rightarrow\infty.
\end{array}
\]
Consequently, letting $k\rightarrow\infty$ in (\ref{fractional subdiff ineq 4}) yields that
\begin{equation}\label{fractional subdiff ineq 5}
\begin{array}
[c]{l}
t^{2H-1}\mathbb{E}\left[\psi(Y^{\varepsilon,\alpha}(t))\right]
+\mathbb{E}{\displaystyle\int_{t}^{T}}s^{2H-1}\nabla\psi(Y^{\varepsilon,\alpha}(s))
\left(\nabla\varphi_{\varepsilon}\right)^{\alpha}(Y^{\varepsilon,\alpha}(s))ds\medskip\\
\leq T^{2H-1}\mathbb{E}\psi(\xi)
+\mathbb{E}{\displaystyle\int_{t}^{T}}s^{2H-1}\nabla\psi(Y^{\varepsilon,\alpha}(s))
f\left(s,\eta(s),Y^{\varepsilon,\alpha}(s),Z^{\varepsilon,\alpha}(s)\right)ds,~t\in[0,T].
\end{array}
\end{equation}
A similar argument allows to take the limit $\alpha\rightarrow0$ in (\ref{fractional subdiff ineq 5}) (using (\ref{equation 6})), it follows
\[
\begin{array}
[c]{l}
t^{2H-1}\mathbb{E}\left[\psi\left(Y^{\varepsilon}(t)\right)\right]
+\mathbb{E}{\displaystyle\int_{t}^{T}}s^{2H-1}\nabla\psi\left(Y^{\varepsilon}(s) \right)\nabla \varphi_{\varepsilon}( Y^{\varepsilon}(s))ds \medskip\\
\leq T^{2H-1}\mathbb{E}\left[\psi(\xi)\right]
+\mathbb{E}{\displaystyle\int_{t}^{T}}s^{2H-1}\nabla\psi\left(Y^{\varepsilon}(s) \right)f(s,\eta(s),Y^{\varepsilon}(s),Z^{\varepsilon}(s))ds, ~t\in(0,T],
\end{array}
\]
and the statement is proven.
\hfill
\end{proof}
\medskip

Now, we are able to give the proof of Proposition \ref{bound of approxim phi}.\medskip

\begin{proof}[Proof of Proposition \ref{bound of approxim phi}]
We consider $\psi(x)=\varphi_{\varepsilon}(x)$, $x\in\mathbb{R}$, and applying (\ref{fractional subdiff ineq}) we have
\[
\begin{array}
[c]{l}
t^{2H-1}\mathbb{E}\left[  \varphi_{\varepsilon}( Y^{\varepsilon}(t))\right]
+\mathbb{E}{\displaystyle\int_{t}^{T}}s^{2H-1}\left|\nabla\varphi_{\varepsilon}(Y^{\varepsilon}(s))\right|^{2}ds\medskip\\
\leq T^{2H-1}\mathbb{E}\left[  \varphi_{\varepsilon}(\xi) \right]
+\mathbb{E}{\displaystyle\int_{t}^{T}}s^{2H-1}\nabla\varphi_{\varepsilon}(Y^{\varepsilon}(s))
f(s,\eta(s),Y^{\varepsilon}(s),Z^{\varepsilon}(s))ds,~s\in[0,T].
\end{array}
\]
Since $0\leq\varphi_{\varepsilon}(u)\leq\varphi(u)$, $u\in\mathbb{R}$ we obtain
\[
\begin{array}
[c]{l}
t^{2H-1}\mathbb{E}\left[\varphi_{\varepsilon}\left(Y^{\varepsilon}(t)\right)\right]
+\mathbb{E}{\displaystyle\int_{t}^{T}}s^{2H-1}\left\vert \nabla\varphi_{\varepsilon}\left(Y^{\varepsilon}(s)\right)\right\vert^{2}ds\medskip\\
\quad\leq T^{2H-1}\mathbb{E}\left[\varphi(\xi)\right]
+\mathbb{E}{\displaystyle\int_{t}^{T}}\left[\dfrac{1}{2}s^{2H-1}\left\vert\nabla\varphi_{\varepsilon}\left(Y^{\varepsilon}(s)\right)\right\vert ^{2}
+\dfrac{1}{2}s^{2H-1}\left\vert f\left(s,\eta(s),Y^{\varepsilon}(s),Z^{\varepsilon}(s)\right)\right\vert^{2}\right]ds\medskip\\
\quad\leq T^{2H-1}\mathbb{E}\left[\varphi(\xi)\right]
+\dfrac{1}{2}\mathbb{E}{\displaystyle\int_{t}^{T}}s^{2H-1}\left\vert\nabla\varphi_{\varepsilon}\left(Y^{\varepsilon}(s)\right)\right\vert^{2}ds\medskip\\
\quad\quad+2\mathbb{E}{\displaystyle\int_{t}^{T}}s^{2H-1}\big(\left\vert f\left(s,0,0,0\right)\right\vert ^{2}
+L^{2}\left\vert \eta(s)\right\vert^{2}+L^{2}\left\vert Y^{\varepsilon}(s)\right\vert ^{2}+L^{2}\left\vert Z^{\varepsilon}(s)\right\vert^{2}\big)ds.
\end{array}
\]
Considering Proposition \ref{bound of Y and Z}, we see that
\[
t^{2H-1}\mathbb{E}\left[\varphi_{\varepsilon}\left(Y^{\varepsilon}(t)\right)\right]
+\mathbb{E}{\int_{t}^{T}}s^{2H-1}\left\vert \nabla\varphi_{\varepsilon}\left(Y^{\varepsilon}(s)\right)\right\vert ^{2}ds\leq C~\Gamma_{2}(T).
\]
Therefore, since $\varphi(J_{\varepsilon}u)\leq \varphi_{\varepsilon}(u)$, $u\in\mathbb{R}$, we have proven $(i)$ and $(ii)$ of the proposition.

Finally, in order to obtain $(iii)$, it is suffices to remark that
$\left\vert u-J_{\varepsilon}(u)\right\vert^{2}=|\nabla\varphi_{\varepsilon}(u)|^{2}\leq2\varepsilon\varphi_{\varepsilon}(u)$, $u\in\mathbb{R}$.
\hfill
\end{proof}

\begin{proposition}\label{Cauchy sequence}
Let the assumptions $(H_{1})$-$(H_{5})$ be satisfied. Then  there exists a positive constant $C$ such that, for all $\varepsilon,\delta>0$,
\[
\sup\limits_{s\in[0,T]}s^{2H-1}\mathbb{E}\left|Y^{\varepsilon}(s)-Y^{\delta}(s)\right|^{2}ds
+{\int_{0}^{T}}s^{2\left(2H-1\right)}\mathbb{E}\left|Z^{\varepsilon}(s)-Z^{\delta}(s)\right|^{2}ds\leq(\varepsilon+\delta)C~\Gamma_{2}(T).
\]

\end{proposition}

\begin{proof}
Similarly to (\ref{equation for penelized Y and Z}), we have, for $t\in(0,T]$,
\begin{equation}\label{equation which will be used}
\begin{array}
[c]{l}
\mathbb{E}|\Delta Y(t)|^{2}+2\mathbb{E}{\displaystyle\int_{t}^{T}}\dfrac{\hat{\sigma}(s)}{\sigma(s)}|\Delta Z(s)|^{2}ds
=-2\mathbb{E}{\displaystyle\int_{t}^{T}}\Delta Y(s)[\nabla\varphi_{\varepsilon}(Y^{\varepsilon}(s))-\nabla\varphi_{\delta}(Y^{\delta}(s))]ds\medskip\\
\qquad+2\mathbb{E}{\displaystyle\int_{t}^{T}}\Delta Y(s)
[f(s,\eta(s),Y^{\varepsilon}(s),Z^{\varepsilon}(s))-f(s,\eta(s),Y^{\delta}(s),Z^{\delta}(s))]ds
\end{array}
\end{equation}
where $\Delta Y\left(  t\right)  =Y^{\varepsilon}(t)-Y^{\delta}(t)$ and
$\Delta Z\left(  t\right)  =Z^{\varepsilon}(t)-Z^{\delta}(t)$.  Then from
\[
d\left(s^{2H-1}\mathbb{E}\left[|\Delta Y(s)^{2}\right]\right)
=s^{2H-1}d\mathbb{E}\left[|\Delta Y(s)^{2}\right]+(2H-1)s^{2H-2}\mathbb{E}\left[|\Delta Y(s)^{2}\right]
\]
as well as (\ref{equation which will be used}), we deduce that for $t\in(0,T]$,
\begin{equation}\label{ineq Cauchy sequence}
\begin{array}
[c]{l}
t^{2H-1}\mathbb{E}|\Delta Y(t)|^{2}
+\mathbb{E}{\displaystyle\int_{t}^{T}}\left(2H-1\right) s^{2H-2}|\Delta Y(s)|^{2}ds
+\mathbb{E}{\displaystyle\int_{t}^{T}2}\dfrac{\hat{\sigma}(s)}{\sigma(s)}s^{2H-1}|\Delta Z(s)|^{2}ds\medskip\\
=-2\mathbb{E}{\displaystyle\int_{t}^{T}}
s^{2H-1}\Delta Y(s)[\nabla\varphi_{\varepsilon}(Y^{\varepsilon}(s))-\nabla\varphi_{\delta}(Y^{\delta}(s))]ds\medskip\\
\quad+2\mathbb{E}{\displaystyle\int_{t}^{T}}s^{2H-1}\Delta Y(s)
[f(s,\eta(s),Y^{\varepsilon}(s),Z^{\varepsilon}(s))-f(s,\eta(s),Y^{\delta}(s),Z^{\delta}(s))]ds.
\end{array}
\end{equation}
Since
\[
\begin{array}
[c]{l}
2s^{2H-1}|\Delta Y\left(  s\right)| |f(s,\eta(s),Y^{\varepsilon}
(s),Z^{\varepsilon}(s))-f(s,\eta(s),Y^{\delta}(s),Z^{\delta}(s))|\medskip\\
\quad\leq2Ls^{2H-1}|\Delta Y\left(  s\right)  |^{2}+2Ls^{2H-1}|\Delta Y\left(
s\right)  |~|\Delta Z\left(  s\right)  |\medskip\\
\quad\leq\Big(2L+L^{2}M\dfrac{1}{s^{2H-1}}\Big)s^{2H-1}|\Delta Y\left(
s\right)  |^{2}+\dfrac{1}{M}s^{2\left(  2H-1\right)  }|\Delta Z\left(
s\right)  |^{2},
\end{array}
\]
(\ref{ineq Cauchy sequence}) yields [Recall(\ref{properties for sigma}-$b$)]
\begin{equation}\label{ineq Cauchy sequence 2}
\begin{array}
[c]{l}
t^{2H-1}\mathbb{E}|\Delta Y(t)|^{2}+\dfrac{2}{M}\mathbb{E}{\displaystyle\int_{t}^{T}}s^{2\left(  2H-1\right)  }|\Delta Z(s)|^{2}ds
+\mathbb{E}{\displaystyle\int_{t}^{T}}\left( 2H-1\right)  s^{2H-2}|\Delta Y(s)|^{2}ds\medskip\\
\quad\leq\mathbb{E}{\displaystyle\int_{t}^{T}}\Big(2L+L^{2}M\dfrac{1}{s^{2H-1}}\Big)s^{2H-1}|\Delta Y(s)|^{2}ds
+\dfrac{1}{M}\mathbb{E}{\displaystyle\int_{t}^{T}}s^{2(2H-1)}|\Delta Z(s)|^{2}ds\medskip\\
\quad\quad-2\mathbb{E}{\displaystyle\int_{t}^{T}}s^{2H-1}\Delta Y(s)
[\nabla\varphi_{\varepsilon}(Y^{\varepsilon}(s))-\nabla\varphi_{\delta}(Y^{\delta}(s))]ds.
\end{array}
\end{equation}
By using the following inequality
\[
\left(\nabla\varphi_{\varepsilon}(u)-\nabla\varphi_{\delta}(v)\right)
(u-v)\geq-\left(\varepsilon+\delta\right)  \nabla\varphi_{\varepsilon
}(u)\nabla\varphi_{\delta}(v).
\]
and Proposition \ref{bound of approxim phi} $(i)$ as well as Gronwall's inequality, we conclude from (\ref{ineq Cauchy sequence 2}) that there exists
$C>0$ such that%
\[%
\begin{array}
[c]{l}
t^{2H-1}\mathbb{E}|\Delta Y(t)|^{2}+\mathbb{E}{\displaystyle\int_{t}^{T}}s^{2(2H-1)}|\Delta Z(s)|^{2}ds\medskip\\
\quad\leq C\left(\varepsilon+\delta\right)  \mathbb{E}{\displaystyle\int_{t}^{T}}s^{2H-1}|\nabla\varphi_{\varepsilon}(Y^{\varepsilon}(s))|
|\nabla\varphi_{\delta}(Y^{\delta}(s))|ds\medskip\\
\quad\leq C\left(\varepsilon+\delta\right)\mathbb{E}{\displaystyle\int_{t}^{T}}\left(  s^{2H-1}|\nabla\varphi_{\varepsilon}(Y^{\varepsilon}(s))|^{2}
+s^{2H-1}|\nabla\varphi_{\delta}(Y^{\delta}(s))|^{2}\right)ds\leq C(\varepsilon+\delta)\Gamma_{2}(T),
\end{array}
\]
and the proof is complete.\hfill
\end{proof}

\subsection{Proof of the Existence of the solution}

\begin{proof}[Proof of Theorem \ref{existence of fractional BSVI}]
For arbitrary $\rho>0$, by Proposition \ref{Cauchy sequence}, there exist
$(Y,Z)\in\mathcal{\bar{V}}_{T}^{H}\times\mathcal{\bar{V}}_{T}^{2H-1/2}$ such that
\begin{equation}\label{equation 9}
\sup\limits_{s\in[0,T]}s^{2H-1}\mathbb{E}|Y^{\varepsilon}(s)-Y(s)|^{2}ds\rightarrow0
\text{ and }{\int_{0}^{T}}s^{2(2H-1)}\mathbb{E}|Z^{\varepsilon}(s)-Z(s)|^{2}ds\rightarrow0,
\end{equation}
From Proposition \ref{bound of approxim phi} $(iii)$, we deduce
that
\begin{equation}\label{equation 10}
\begin{array}
[c]{l}
\lim\limits_{\varepsilon\rightarrow0}\mathbb{E}\left\vert Y^{\varepsilon}(t)-J_{\varepsilon}\left(  Y^{\varepsilon}(t)\right)\right\vert ^{2}=0
\text{ for all } t\in[0,T],\medskip\\
\lim\limits_{\varepsilon\rightarrow0}J_{\varepsilon}(Y^{\varepsilon})=Y \text{ in }\mathcal{\bar{V}}_{T}^{H}.
\end{array}
\end{equation}
For each $\varepsilon>0$, let $U^{\varepsilon}(t)=\nabla\varphi_{\varepsilon}\left(Y^{\varepsilon}(t)\right)$, $t\in[0,T]$.
The process $U^{\varepsilon}$ belongs to the space $\mathcal{\bar{V}}_{T}^{H}$ (see Lemma \ref{property miu bar} in the Appendix).
From Proposition \ref{bound of approxim phi} $(i)$, we obtain that
\[
\left\vert \left\vert U^{\varepsilon}\right\vert \right\vert_{H}^{2}=\mathbb{E}{\int_{0}^{T}}s^{2H-1}|U^{\varepsilon}(s)|^{2}ds
\leq C\Gamma_{2}(T),~\varepsilon>0.
\]
Hence, there exists a subsequence $\varepsilon_{n}\rightarrow0$ and a process $U\in\mathcal{\bar{V}}_{T}^{H}$ such that
\begin{equation}\label{equation 11}
U^{\varepsilon_{n}}\xrightarrow[\;\;\;\;\;\;\;\;\;\;\;\;\;]{\varepsilon_{n}\rightarrow0}U,
\text{ weakly in the Hilbert space }\mathcal{\bar{V}}_{T}^{H}.
\end{equation}
Consequently,
\[
\mathbb{E}{\int_{0}^{T}}s^{2H-1}\left\vert U(s)\right\vert^{2}ds
\leq\liminf\limits_{\varepsilon_{n}\rightarrow0}\mathbb{E}{\int_{0}^{T}}s^{2H-1}\left\vert U^{\varepsilon_{n}}(s)\right\vert ^{2}ds
\leq C~\Gamma_{2}(T).
\]
From Eq. (\ref{penalized fBSVI}) we have
\[
\begin{array}
[c]{l}
{\displaystyle\int_{t}^{T}}Z^{\varepsilon_{n}}(s)\delta B^{H}(s)
=-Y^{\varepsilon_{n}}(t)-{\displaystyle\int_{t}^{T}}U^{\varepsilon_{n}}(s)ds+\xi
+{\displaystyle\int_{t}^{T}}f\left(  s,\eta(s),Y^{\varepsilon_{n}}(s),Z^{\varepsilon_{n}}(s)\right)ds\medskip\\
\qquad\qquad\qquad\qquad :=\theta^{\varepsilon_{n}}(t),~t\in(0,T],~n\ge1.
\end{array}
\]
On the other hand, since $Z^{\varepsilon_{n}}\mathbf{1}_{[\rho,T]}$ converges to $Z\mathbf{1}_{[\rho,T]}$ in $L^{2}(\Omega,\mathcal{F},P;\mathcal{H})$ for all $\rho>0$ and $Z^{\varepsilon_{n}}I_{[\rho,T]}\in Dom(\delta)$, then we apply Definition \ref{divergence operator} and we obtain
for all $F\in\mathcal{P}_{T}$,
\[
\begin{array}
[c]{l}
\mathbb{E}\left( \left\langle D_{\cdot}^{H}F,Z(\cdot)\mathbf{1}_{[\rho,T]}(\cdot)\right\rangle _{T}\right)
=\lim\limits_{n\rightarrow\infty}\mathbb{E}\left(\left\langle D_{\cdot}^{H}F,Z^{\varepsilon_{n}}(\cdot)\mathbf{1}_{[\rho,T]}(\cdot)\right\rangle _{T}\right) \medskip\\
=\lim\limits_{n\rightarrow\infty}\mathbb{E}\left( F{{\displaystyle\int_{\rho}^{T}}}Z^{\varepsilon_{n}}(s)(s)\delta B^{H}(s)\right)
=\lim\limits_{n\rightarrow\infty}\mathbb{E}(F~\theta^{\varepsilon_{n}}(\rho))=\mathbb{E}(F~\theta(\rho)),
\end{array}
\]
where it follows from (\ref{equation 9}) and (\ref{equation 11}) that
\[
\theta(t)=-Y(t)-{\displaystyle\int_{t}^{T}}U(s)ds+\xi
+{\displaystyle\int_{t}^{T}}f\left(  s,\eta(s),Y(s),Z(s)\right)ds
\]
is the weak limit in $L^{2}(\Omega,\mathcal{F},P)$ of $\theta^{\varepsilon_{n}}(t)$ as $n\rightarrow\infty$, $t\in(0,T]$.
From the definition of the divergence operator, it follows that $Z\mathbf{1}_{[\rho,T]}\in Dom(\delta)$ and
$\delta(Z\mathbf{1}_{[\rho,T]})=\theta(\rho)$, $P$-a.s. Consequently, since $\rho>0$ is arbitrarily chosen, we have
\begin{equation}\label{equation 12}
Y(t)+{\displaystyle\int_{t}^{T}}U(s)ds=\xi
+{\displaystyle\int_{t}^{T}}f\left( s,\eta(s),Y(s),Z(s)\right)ds-{\displaystyle\int_{t}^{T}}Z(s)\delta B^{H}(s), \text{ for all } t\in(0,T].
\end{equation}
Moreover, since $U^{\varepsilon}(t)\in\partial\varphi(J_{\varepsilon}(Y^{\varepsilon}(t))$, for $t\in[0,T]$, we have
\[
U^{\varepsilon}(t)(V(t)-J_{\varepsilon}(Y^{\varepsilon}(t)))+\varphi(J_{\varepsilon}(Y^{\varepsilon}(t)))
\leq\varphi(V(t)), \text{ for all }V\in\mathcal{\bar{V}}_{T}^{H}, ~t\in[0,T],
\]
and we deduce that for all $A\times[a,b]\subset\Omega\times[0,T]$, $A\in\mathcal{F}$,
\[
\begin{array}
[c]{l}
\mathbb{E}\Big({\displaystyle\int_{a}^{b}}s^{2H-1}\mathbf{1}_{A}U^{\varepsilon}(t)\big(V(t)-J_{\varepsilon}(Y^{\varepsilon}(t))\big)dt\Big)
+\mathbb{E}\Big({\displaystyle\int_{a}^{b}}s^{2H-1}\mathbf{1}_{A}\varphi(J_{\varepsilon}(Y^{\varepsilon}(t)))dt\Big)\medskip\\
\quad\leq\mathbb{E}\Big({\displaystyle\int_{a}^{b}}s^{2H-1}\mathbf{1}_{A}\varphi(V(t))dt\Big).
\end{array}
\]
Considering that $\varphi$ is a proper convex l.s.c. function; hence (\ref{equation 10}) and (\ref{equation 11}) yield that
\[
\begin{array}
[c]{l}
\mathbb{E}\Big({\displaystyle\int_{a}^{b}}s^{2H-1}\mathbf{1}_{A}U(t)(V(t)-Y(t))dt\Big)
+\mathbb{E}\Big({\displaystyle\int_{a}^{b}}s^{2H-1}\mathbf{1}_{A}\varphi(Y(t))dt\Big)\medskip\\
\quad
\leq\mathbb{E}\Big({\displaystyle\int_{a}^{b}}s^{2H-1}\mathbf{1}_{A}\varphi(V(t))dt\Big),  \text{ for all }  A\times[a,b]\subset\Omega\times[0,T].
\end{array}
\]
Therefore,
\[
U(t)(V(t)-Y(t))+\varphi(Y(t))\leq\varphi(V(t))~dP\otimes dt~a.e.\text{ on }\Omega\times[0,T],
\]
which means that
\[
(Y(t),U(t))\in\partial\varphi,~ dP\otimes dt~a.e.\text{ on }\Omega\times[0,T].
\]
This together with (\ref{equation 12}) complete the proof.
\hfill
\end{proof}
\medskip

\textbf{ Acknowledgement}\quad  The authors wish to express their thanks
to Rainer Buckdahn, Yaozhong Hu, Shige Peng, and Aurel R\u{a}\c{s}canu for their useful
suggestions and discussions.
\newpage

\noindent\textbf{Appendix}

\begin{theorem}\label{general Ito formula for the divergence integral}
Let $\psi$ be a function of class $C^{1,2}([0,T]\times\mathbb{R})$. Assume that $u$ is a process in
$\mathcal{V}_{T}$ and $f\in C_{pol}^{0,1}([0,T]\times\mathbb{R})$. Let
\[
X_{t}=X_{0}+{\displaystyle\int_{0}^{t}}f(s,\eta(s))ds+{\displaystyle\int_{0}^{t}}u_{s}\delta B^{H}(s), ~s\in[0,T]
\]
Then for all $t\in[0,T]$, the following formula holds
\begin{equation}\label{1 general Ito's formula for the divergence integral}
\begin{array}
[c]{l}
\psi(t,X_{t})=\psi(0,X_{0})+{\displaystyle\int_{0}^{t}}{\displaystyle\dfrac{\partial }{\partial s}}\psi(s,X_{s})ds
+{\displaystyle\int_{0}^{t}}{\displaystyle\dfrac{\partial }{\partial x}}\psi(s,X_{s})f(s,X_{s})ds\medskip\\
\qquad\qquad+{\displaystyle\int_{0}^{t}}{\displaystyle\dfrac{\partial }{\partial x}}\psi(s,X_{s})u_{s}\delta B^{H}(s)
+{\displaystyle\int_{0}^{t}}{\displaystyle\dfrac{\partial^{2}}{\partial x^{2}}}\psi(s,X_{s})u_{s}\mathbb{D}^{H}_{s}X_{s}ds.
\end{array}
\end{equation}
\end{theorem}
\begin{remark}\label{remark for genelized ito formula}
Since $u\in\mathcal{V}_{T}$, we know that $u$ is adapted. Then, due to the definition of $\mathbb{D}^{H}_{s}$ [see (\ref{another derivative})], we have
\[
\begin{array}
[c]{l}
\mathbb{D}^{H}_{s}X_{s}
={\displaystyle\int_{0}^{s}}\phi(s-r)\left({\displaystyle\int_{0}^{s}}D_{r}f(\theta,\eta(\theta))d\theta\right)dr
+{\displaystyle\int_{0}^{s}}\left({\displaystyle\int_{0}^{\theta}}\phi(s-r)D_{r}u_{\theta} dr\right)\delta B^{H}(\theta)\medskip\\
\qquad\qquad+{\displaystyle\int_{0}^{s}}\phi(s-\theta)u_{\theta}d\theta\medskip\\
={\displaystyle\int_{0}^{T}}\phi(s-r)\left({\displaystyle\int_{0}^{s}}D_{r}f(\theta,\eta(\theta))d\theta\right)dr
+{\displaystyle\int_{0}^{T}}\phi(s-r)\left({\displaystyle\int_{0}^{s}}D_{r}u_{\theta}\delta B^{H}(\theta)\right) dr\medskip\\
\qquad\qquad+{\displaystyle\int_{0}^{s}}\phi(s-\theta)u_{\theta}d\theta .
\end{array}
\]
\end{remark}
\begin{proof}[Proof of Theorem \ref{general Ito formula for the divergence integral}]
We follow the similar discussion as in the proof of Theorem 8 \cite{AN-03}. Here, we only give the sketch of the proof.

First, we mention that since $u\in \mathcal{V}_{T}$,  we have $\mathbb{E}|u_{s}|^{2}+\mathbb{E}|D^{H}_{\tau}u_{s}|^{2}+\mathbb{E}|D^{H}_{\tau_{1}}D^{H}_{\tau_{2}}u_{s}|^{2}\leq C$, for all $s,\tau,\tau_{1},\tau_{2}$, where $C$ is a suitable constant.

Similar to the discussion as in the proof of Theorem 8 \cite{AN-03}, we can assume that $\psi$, $\dfrac{\partial \psi}{\partial t}$, $\dfrac{\partial \psi}{\partial x}$, $\dfrac{\partial^{2} \psi}{\partial x^{2}}$ are bounded.
Set $t_{i}=\dfrac{it}{n}$, $0\leq i\leq n$. Then
\[
\begin{array}
[c]{l}
\psi(t,X_{t})-\psi(0,X_{0})
=\sum\limits_{i=0}\limits^{n-1}\left[\psi(t_{i+1},X_{t_{i+1}})-\psi(t_{i},X_{t_{i+1}})+\psi(t_{i},X_{t_{i+1}})-\psi(t_{i},X_{t_{i}})\right]\medskip\\
=\sum\limits_{i=0}\limits^{n-1}\left[\dfrac{\partial}{\partial t}\psi(\bar{t}_{i},X_{t_{i+1}})(t_{i+1}-t_{i})+
\dfrac{\partial}{\partial x}\psi(t_{i},X_{t_{i}})(X_{t_{i+1}}-X_{t_{i}})\right]\medskip\\
\qquad\qquad\qquad\qquad\qquad\qquad
+\dfrac{1}{2}\sum\limits_{i=0}\limits^{n-1}\dfrac{\partial^{2}}{\partial x^{2}}\psi(t_{i},\bar{X}_{t_{i}})(X_{t_{i+1}}-X_{t_{i}})^{2}\medskip\\
=\sum\limits_{i=0}\limits^{n-1}\left[\dfrac{\partial}{\partial t}\psi(\bar{t}_{i},X_{t_{i+1}})(t_{i+1}-t_{i})+
\dfrac{\partial}{\partial x}\psi(t_{i},X_{t_{i}}){\displaystyle\int_{t_{i}}^{t_{i+1}}}f(s,\eta(s))ds\right]\medskip\\
\qquad\qquad
\sum\limits_{i=0}\limits^{n-1}\dfrac{\partial}{\partial x}\psi(t_{i},X_{t_{i}}){\displaystyle\int_{t_{i}}^{t_{i+1}}}u_{s}\delta B^{H}(s)+\dfrac{1}{2}\sum\limits_{i=0}\limits^{n-1}\dfrac{\partial^{2}}{\partial x^{2}}\psi(t_{i},\bar{X}_{t_{i}})(X_{t_{i+1}}-X_{t_{i}})^{2}
\end{array}
\]
where $\bar{t}_{i}\in[t_{i},t_{i+1}]$ and $\bar{X}_{t_{i}}$ denotes a random intermediate point between $X_{t_{i}}$ and $X_{t_{i+1}}$.
Since
\[
\dfrac{\partial}{\partial x}\psi(t_{i},X_{t_{i}}){\displaystyle\int_{t_{i}}^{t_{i+1}}}u_{s}\delta B^{H}(s)
={\displaystyle\int_{t_{i}}^{t_{i+1}}}\dfrac{\partial}{\partial x}\psi(t_{i},X_{t_{i}}) u_{s}\delta B^{H}(s)
+\langle D^{H}(\dfrac{\partial}{\partial x}\psi(t_{i},X_{t_{i}})), u\mathbf{1}_{[t_{i},t_{i+1}]}\rangle_{T},
\]
[for $\langle\cdot,\cdot\rangle_{T}$, see (\ref{inner product})].
Observe that from our assumption, $\dfrac{\partial}{\partial x}\psi(t_{i},X_{t_{i}}) u\in \mathbb{D}^{1,2}(|\mathcal{H}|)$ and all the terms in the above inequality are square integrable. Moreover,
\[
\begin{array}
[c]{l}
\langle D^{H}(\dfrac{\partial}{\partial x}\psi(t_{i},X_{t_{i}})), u\mathbf{1}_{t_{i},t_{i+1}}\rangle_{T}
=\langle \dfrac{\partial^{2}}{\partial x^{2}}\psi(t_{i},X_{t_{i}}){\displaystyle\int_{0}^{t_{i}}}D^{H}f(\theta,\eta(\theta))d\theta,
u\mathbf{1}_{[t_{i},t_{i+1}]}\rangle_{T}\medskip\\
\qquad+\langle \dfrac{\partial^{2}}{\partial x^{2}}\psi(t_{i},X_{t_{i}})u\mathbf{1}_{[0,t_{i}]}, u\mathbf{1}_{[t_{i},t_{i+1}]}\rangle_{T}
+\langle \dfrac{\partial^{2}}{\partial x^{2}}\psi(t_{i},X_{t_{i}}){\displaystyle\int_{0}^{t_{i}}}D^{H}u_{\theta}\delta B^{H}(\theta),
u\mathbf{1}_{[t_{i},t_{i+1}]}\rangle_{T}.
\end{array}
\]
Now we use the following several steps to proof our theorem.

\noindent\textbf{Step 1}
The term
\[
\dfrac{1}{2}\sum\limits_{i=0}\limits^{n-1}\dfrac{\partial^{2}}{\partial x^{2}}\psi(t_{i},\bar{X}_{t_{i}})(X_{t_{i+1}}-X_{t_{i}})^{2}
\]
converges to $0$ in $L^{1}(\Omega,\mathcal{F},P)$ as $n\rightarrow\infty$. In fact,
\[
\begin{array}
[c]{l}
\left|\dfrac{1}{2}\sum\limits_{i=0}\limits^{n-1}\dfrac{\partial^{2}}{\partial x^{2}}\psi(t_{i},\bar{X}_{t_{i}})(X_{t_{i+1}}-X_{t_{i}})^{2}\right|
\medskip\\
\leq\sum\limits_{i=0}\limits^{n-1}\left|\dfrac{\partial^{2}}{\partial x^{2}}\psi(t_{i},\bar{X}_{t_{i}})\right|
\left({\displaystyle\int_{t_{i}}^{t_{i+1}}}f(\theta,\eta(\theta))ds\right)^{2}
+\sum\limits_{i=0}\limits^{n-1}\left|\dfrac{\partial^{2}}{\partial x^{2}}\psi(t_{i},\bar{X}_{t_{i}})\right|
\left({\displaystyle\int_{t_{i}}^{t_{i+1}}}u_{\theta}\delta B^{H}(\theta)\right)^{2}.
\end{array}
\]
Then our result holds because of Proposition 7 \cite{AN-03} as well as
\[
\mathbb{E}\sum\limits_{i=0}\limits^{n-1}\left({\displaystyle\int_{t_{i}}^{t_{i+1}}}f(\theta,\eta(\theta))ds\right)^{2}
\leq \dfrac{t}{n}{\displaystyle\int_{0}^{t}}\mathbb{E}|f(\theta,\eta(\theta))|^{2}ds\rightarrow0, \text{ as }n\rightarrow\infty.
\]

\noindent\textbf{Step 2}
The term
\[
\sum\limits_{i=0}\limits^{n-1}\dfrac{\partial}{\partial t}\psi(\bar{t}_{i},X_{t_{i+1}})(t_{i+1}-t_{i})
\rightarrow{\displaystyle\int_{0}^{t}}\dfrac{\partial}{\partial t}\psi(s,X_{s})ds, \text{ in }L^{1}(\Omega,\mathcal{F},P)
\text{ as } n\rightarrow\infty,
\]
by the dominate convergence theorem and the continuity of $\dfrac{\partial\psi}{\partial t}$.

\noindent\textbf{Step 3}
The term
\[
\begin{array}
[c]{l}
\sum\limits_{i=0}\limits^{n-1}\dfrac{\partial}{\partial x}\psi(t_{i},X_{t_{i}}){\displaystyle\int_{t_{i}}^{t_{i+1}}}f(s,\eta(s))ds\medskip\\
\rightarrow{\displaystyle\int_{0}^{t}}\dfrac{\partial}{\partial x}\psi(s,X_{s})f(s,\eta(s))ds, \text{ in }L^{1}(\Omega,\mathcal{F},P)
\text{ as } n\rightarrow\infty,
\end{array}
\]
by the dominate convergence theorem and the continuity of $\dfrac{\partial\psi}{\partial x}$.

\noindent\textbf{Step 4}
The term
\[
\begin{array}
[c]{l}
\sum\limits_{i=0}\limits^{n-1}\langle \dfrac{\partial^{2}}{\partial x^{2}}\psi(t_{i},X_{t_{i}}){\displaystyle\int_{0}^{t_{i}}}D^{H}f(\theta,\eta(\theta))d\theta,
u\mathbf{1}_{[t_{i},t_{i+1}]}\rangle_{T}\medskip\\
\rightarrow {\displaystyle\int_{0}^{t}}{\displaystyle\dfrac{\partial^{2}}{\partial x^{2}}}\psi(s,X_{s})u_{s}\left({\displaystyle\int_{0}^{T}}
\phi(s-r)\left({\displaystyle\int_{0}^{s}}D_{r}f(\theta,\eta(\theta))d\theta \right)dr\right)ds
\end{array}
\]
in $L^{1}(\Omega,\mathcal{F},P)$ as $n\rightarrow\infty$. Indeed
\[
\begin{array}
[c]{l}
\langle {\displaystyle\int_{0}^{t_{i}}}D^{H}f(\theta,\eta(\theta))d\theta,
u\mathbf{1}_{[t_{i},t_{i+1}]}\rangle_{T}
=
{\displaystyle\int_{0}^{T}}{\displaystyle\int_{t_{i}}^{t_{i+1}}}\left({\displaystyle\int_{0}^{t_{i}}}D^{H}_{\mu}f(\theta,\eta(\theta))d\theta\right) u_{s}\phi(s-\mu)dsd\mu.
\end{array}
\]
Then
\[
\begin{array}
[c]{l}
\mathbb{E}\Bigg[\Big|\sum\limits_{i=0}\limits^{n-1}\langle \dfrac{\partial^{2}}{\partial x^{2}}\psi(t_{i},X_{t_{i}}){\displaystyle\int_{0}^{t_{i}}}D^{H}f(\theta,\eta(\theta))d\theta,
u\mathbf{1}_{[t_{i},t_{i+1}]}\rangle_{T}\medskip\\
\qquad-{\displaystyle\int_{0}^{t}}{\displaystyle\dfrac{\partial^{2}}{\partial x^{2}}}\psi(s,X_{s})u_{s}\left({\displaystyle\int_{0}^{T}}
\phi(s-r)\left({\displaystyle\int_{0}^{s}}D_{r}^{H}f(\theta,\eta(\theta))d\theta \right)dr\right)ds\Big|\Bigg]\medskip\\
=\mathbb{E}\Bigg[\Big|\sum\limits_{i=0}\limits^{n-1}{\displaystyle\int_{t_{i}}^{t_{i+1}}}\dfrac{\partial^{2}}{\partial x^{2}}\psi(t_{i},X_{t_{i}})u_{s}\left({\displaystyle\int_{0}^{T}}\phi(s-\mu)
\left({\displaystyle\int_{0}^{t_{i}}}D^{H}_{\mu}f(\theta,\eta(\theta))d\theta\right) d\mu\right)ds\medskip\\
\qquad-\sum\limits_{i=0}\limits^{n-1}{\displaystyle\int_{t_{i}}^{t_{i+1}}}{\displaystyle\dfrac{\partial^{2}}{\partial x^{2}}}\psi(s,X_{s})u_{s}\left({\displaystyle\int_{0}^{T}}
\phi(s-r)\left({\displaystyle\int_{0}^{s}}D_{r}^{H}f(\theta,\eta(\theta))d\theta \right)dr\right)ds\Big|\Bigg]
\medskip\\
\rightarrow0,\text{ as } n\rightarrow\infty
\end{array}
\]
by the dominate convergence theorem and the continuity of $\dfrac{\partial^{2}\psi}{\partial x^{2}}$. Observing that
\[
\begin{array}
[c]{l}
\sum\limits_{i=0}\limits^{n-1}\left|{\displaystyle\int_{t_{i}}^{t_{i+1}}}\dfrac{\partial^{2}}{\partial x^{2}}\psi(t_{i},X_{t_{i}})u_{s}\left({\displaystyle\int_{0}^{T}}\phi(s-\mu)
\left({\displaystyle\int_{0}^{t_{i}}}D^{H}_{\mu}f(\theta,\eta(\theta))d\theta\right) d\mu\right)ds\right|\medskip\\
\leq M{\displaystyle\int_{0}^{t}}|u_{s}|\left({\displaystyle\int_{0}^{T}}
\phi(s-r)\left({\displaystyle\int_{0}^{T}}|D_{r}^{H}f(\theta,\eta(\theta))|d\theta \right)dr\right)ds
\end{array}
\]
and
\[
\mathbb{E}\left[{\displaystyle\int_{0}^{t}}|u_{s}|\left({\displaystyle\int_{0}^{T}}
\phi(s-r)\left({\displaystyle\int_{0}^{T}}|D_{r}^{H}f(\theta,\eta(\theta))|d\theta \right)dr\right)ds\right]<\infty
\]

\noindent\textbf{Step 5} Analogously to the Step 4, we get
\[
\begin{array}
[c]{l}
\sum\limits_{i=0}\limits^{n-1}\langle \dfrac{\partial^{2}}{\partial x^{2}}\psi(t_{i},X_{t_{i}})u\mathbf{1}_{[0,t_{i}]}, u\mathbf{1}_{[t_{i},t_{i+1}]}\rangle_{T}
\rightarrow {\displaystyle\int_{0}^{t}}{\displaystyle\dfrac{\partial^{2}}{\partial x^{2}}}\psi(s,X_{s})u_{s}\left({\displaystyle\int_{0}^{s}}
u_{\theta}\phi(s-\theta)d\theta \right)ds,
\end{array}
\]
in $L^{1}(\Omega,\mathcal{F},P)$ as $n\rightarrow\infty$.

\noindent\textbf{Step 6}
The term
\[
\begin{array}
[c]{l}
\sum\limits_{i=0}\limits^{n-1}\langle \dfrac{\partial^{2}}{\partial x^{2}}\psi(t_{i},X_{t_{i}})
{\displaystyle\int_{0}^{t_{i}}}D^{H}u_{\theta}\delta B^{H}(\theta),
u\mathbf{1}_{[t_{i},t_{i+1}]}\rangle_{T}\medskip\\
\rightarrow {\displaystyle\int_{0}^{t}}{\displaystyle\dfrac{\partial^{2}}{\partial x^{2}}}\psi(s,X_{s})u_{s}\left({\displaystyle\int_{0}^{T}}
\phi(s-r)\left({\displaystyle\int_{0}^{s}}D^{H}_{r}u_{\theta}\delta B^{H}(\theta) \right)dr\right)ds
\end{array}
\]
in $L^{1}(\Omega,\mathcal{F},P)$, as $n\rightarrow\infty$. Indeed,
we can adapt the discussion of \emph{step 3} in the proof of Theorem 8 \cite{AN-03}, by using  $\mathbb{E}|u_{s}|^{2}+\mathbb{E}|D^{H}_{\tau}u_{s}|^{2}+\mathbb{E}|D^{H}_{\tau_{1}}D^{H}_{\tau_{2}}u_{s}|^{2}\leq C$, for all $s,\tau,\tau_{1},\tau_{2}$, where $C$ is a suitable constant.

\noindent\textbf{Step 7}
The term
\[
\sum\limits_{i=0}\limits^{n-1}\dfrac{\partial}{\partial x}\psi(t_{i},X_{t_{i}})u\mathbf{1}_{[t_{i},t_{i+1}]}
\rightarrow \dfrac{\partial}{\partial x}\psi(\cdot,X_{\cdot})u_{\cdot}
~\text{ in } L^{1}(\Omega,\mathcal{F},P;|\mathcal{H}|).
\]
In fact
\[
\begin{array}
[c]{l}
\mathbb{E}\left|\left|\sum\limits_{i=0}\limits^{n-1}\dfrac{\partial}{\partial x}\psi(t_{i},X_{t_{i}})u\mathbf{1}_{[t_{i},t_{i+1}]}
-\dfrac{\partial}{\partial x}\psi(\cdot,X_{\cdot})u_{\cdot}\right|\right|_{|\mathcal{H}|}^{2}\medskip\\
=\mathbb{E}\sum\limits_{i=0}\limits^{n-1}\sum\limits_{j=0}\limits^{n-1}
{\displaystyle\int_{t_{i}}^{t_{i+1}}}{\displaystyle\int_{t_{j}}^{t_{j+1}}}\left|\left(\dfrac{\partial}{\partial x}\psi(t_{i},X_{t_{i}})
-\dfrac{\partial}{\partial x}\psi(s,X_{s})\right)u_{s}\right|\medskip\\
\qquad\qquad\qquad\qquad\times\left|\left(\dfrac{\partial}{\partial x}\psi(t_{i},X_{t_{i}})
-\dfrac{\partial}{\partial x}\psi(r,X_{r})\right)u_{r}\right|\phi(s-r)dsdr,
\end{array}
\]
which converges to $0$ by the dominated convergence theorem and the continuity of $\dfrac{\partial\psi}{\partial x}$. Note that
for $u\in\mathcal{V}_{T}$ and $\dfrac{\partial\psi}{\partial x}$, $\dfrac{\partial^{2}\psi}{\partial x^{2}}$ bounded, we have $\dfrac{\partial\psi}{\partial x}u\in\mathbb{D}^{1,2}(|\mathcal{H}|)\subset Dom(\delta)$.
Consequently, for $F\in\mathcal{P}_{T}$, we have
\[
\lim\limits_{n\rightarrow\infty}\mathbb{E}\left[F\sum\limits_{i=0}\limits^{n-1}{\displaystyle\int_{t_{i}}^{t_{i+1}}}
\dfrac{\partial}{\partial x}\psi(t_{i},X_{t_{i}})u_{s}\delta B^{H}(s)\right]
=\mathbb{E}\left[F{\displaystyle\int_{0}^{t}}
\dfrac{\partial}{\partial x}\psi(s,X_{s})u_{s}\delta B^{H}(s)\right]
\]
On the other hand, from steps 1$-$6, we know that $\sum\limits_{i=0}\limits^{n-1}{\displaystyle\int_{t_{i}}^{t_{i+1}}}
\dfrac{\partial}{\partial x}\psi(t_{i},X_{t_{i}})u_{s}\delta B^{H}(s)$ converges in $L^{1}(\Omega,\mathcal{F},P)$ to
\[
\begin{array}
[c]{l}
\psi(t,X_{t})-\psi(0,X_{0})+{\displaystyle\int_{0}^{t}}{\displaystyle\dfrac{\partial }{\partial s}}\psi(s,X_{s})ds
+{\displaystyle\int_{0}^{t}}{\displaystyle\dfrac{\partial }{\partial x}}\psi(s,X_{s})f(s,X_{s})ds\medskip\\
\qquad\qquad+H(2H-1){\displaystyle\int_{0}^{t}}{\displaystyle\dfrac{\partial^{2}}{\partial x^{2}}}\psi(s,X_{s})u_{s}\left({\displaystyle\int_{0}^{T}}
|s-r|^{2H-2}\left({\displaystyle\int_{0}^{s}}D_{r}f(\theta,\eta(\theta))d\theta \right)dr\right)ds\medskip\\
\qquad\qquad+H(2H-1){\displaystyle\int_{0}^{t}}{\displaystyle\dfrac{\partial^{2}}{\partial x^{2}}}\psi(s,X_{s})u_{s}\left({\displaystyle\int_{0}^{T}}
|s-r|^{2H-2}\left({\displaystyle\int_{0}^{s}}D_{r}u_{\theta}\delta B^{H}(\theta)\right)dr\right)ds\medskip\\
\qquad\qquad+H(2H-1){\displaystyle\int_{0}^{t}}{\displaystyle\dfrac{\partial^{2}}{\partial x^{2}}}\psi(s,X_{s})u_{s}\left({\displaystyle\int_{0}^{s}}
u_{\theta}|s-\theta|^{2H-2}d\theta \right)ds,
\end{array}
\]
as $n\rightarrow\infty$, which allows to complete the proof. \hfill
\end{proof}
\medskip

In particular, we have the following corollary.
\begin{corollary}\label{particular for general ito formula}
Let $f:[0,T]\rightarrow\mathbb{R}$ and $g:[0,T]\rightarrow\mathbb{R}$ be deterministic continuous  functions.
If
\[
X_{t}=X_{0}+\int_{0}^{t}g_{s}ds+\int_{0}^{t}f_{s}\delta B^{H}(s),~t\in[0,T],
\]
and $\psi\in C^{1,2}([0,T]\times\mathbb{R})$, then we have
\[
\begin{array}
[c]{l}
\psi(t,X_{t})=\psi(0,X_{0})+{\displaystyle\int_{0}^{t}}\dfrac{\partial}{\partial s} \psi(s,X_{s})ds
+{\displaystyle\int_{0}^{t}}\dfrac{\partial }{\partial x}\psi(s,X_{s})dX_{s}\medskip\\
\qquad\qquad\qquad
+\dfrac{1}{2}{\displaystyle\int_{0}^{t}}\dfrac{\partial^{2}}{\partial x^{2}}\psi(s,X_{s})~\left( \dfrac{d}{ds} \Vert f\Vert_{s}^{2}\right)ds,\;t\in[0,T].
\end{array}
\]
\end{corollary}
\begin{proof}
Since $f,g$ are deterministic function, they satisfy the condition of
Theorem \ref{general Ito formula for the divergence integral}. Then from (\ref{1 general Ito's formula for the divergence integral}) and Remark \ref{remark for genelized ito formula} we have
\[
\begin{array}
[c]{l}
\psi(t,X_{t})=\psi(0,X_{0})+{\displaystyle\int_{0}^{t}}{\displaystyle\dfrac{\partial }{\partial s}}\psi(s,X_{s})ds
+{\displaystyle\int_{0}^{t}}{\displaystyle\dfrac{\partial }{\partial x}}\psi(s,X_{s})dX_{s}\medskip\\
\qquad\qquad+{\displaystyle\int_{0}^{t}}{\displaystyle\dfrac{\partial^{2}}{\partial x^{2}}}\psi(s,X_{s})f_{s}\left({\displaystyle\int_{0}^{s}}
f_{\theta}\phi(s-\theta)d\theta \right)ds.
\end{array}
\]
On the other hand
\[
\begin{array}
[c]{l}
\dfrac{d}{ds} \Vert f\Vert_{s}^{2}=\dfrac{d}{ds}{\displaystyle\int_{0}^{s}}{\displaystyle\int_{0}^{s}}\phi(u-v)f_{u}f_{v}dudv
=2f_{s}{\displaystyle\int_{0}^{s}}\phi(u-s)f_{u}du,
\end{array}
\]
which completes our proof. \hfill
\end{proof}
\begin{lemma}
\label{property miu bar}If $Y\in\mathcal{\bar{V}}_{T}^{\alpha}$ and $\psi$ is
a Lipschitz function, then $\psi(Y)\in\mathcal{\bar{V}}_{T}
^{\alpha}$.
\end{lemma}

\begin{proof}
We remark first that, since $Y\in\mathcal{\bar{V}}_{T}^{\alpha}$, there exists
a sequence $\left\{Y_{n}\right\}_{n=1}^{\infty}\subset\mathcal{V}_{T}$ such that
$\lim\limits_{n\rightarrow\infty}{\displaystyle\int_{0}^{T}}t^{2\alpha
-1}\mathbb{E}\left\vert Y_{n}\left(  t\right)  -Y\left(  t\right)  \right\vert^{2}dt=0.$
Let us set
\[
\psi^{\delta}(x)={\int_{\mathbb{R}}}\psi(x-\delta u)\rho(u)du,~x\in\mathbb{R},
\]
where $\delta>0$ and $\rho(u)=\dfrac{1}{\sqrt{2\pi}}e^{-\frac{u^{2}}{2}},~u\in\mathbb{R}$. We
know that
\[
|\psi^{\delta}(x)-\psi(x)|~\leq{\int_{\mathbb{R}}}\big|\psi(x-\delta u)-\psi(x)\big|\rho(u)du\leq{\int_{\mathbb{R}}}K\delta|u|\rho(u)du
\leq K\delta\sqrt{2/\pi},
\]
where $K$ is a Lipschitz constant of $\psi$. Now, we have
\[
\begin{array}
[c]{l}
{\displaystyle\int_{0}^{T}}t^{2\alpha-1}\mathbb{E}|\psi^{\delta}(Y_{n}(t))-\psi(Y(t))|^{2}dt\medskip\\
\quad\leq2{\displaystyle\int_{0}^{T}}t^{2\alpha-1}\mathbb{E}\big|\psi^{\delta}(Y_{n}(t))-\psi(Y_{n}(t))\big|^{2}dt
+2{\displaystyle\int_{0}^{T}}t^{2\alpha-1}\mathbb{E}\big|\psi(Y_{n}(t))-\psi(Y(t))\big|^{2}dt\medskip\\
\quad\leq\dfrac{4}{\pi}{\displaystyle\int_{0}^{T}}K^{2}\delta^{2}t^{2\alpha-1}dt
+2K^{2}{\displaystyle\int_{0}^{T}}t^{2\alpha-1}\mathbb{E}|Y_{n}(t)-Y(t)|^{2}dt.
\end{array}
\]
Consequently,
\[
{\int_{0}^{T}}t^{2\alpha-1}\mathbb{E}|\psi^{\delta}(Y_{n}(t))-\psi(Y(t))|^{2}dt\rightarrow0, \text{ for } n\rightarrow \infty,\;\delta\rightarrow0,
\]
and the proof is completed by showing that $\psi^{\delta}(Y_{n})\in\mathcal{V}_{T}$, for all $n$ and $\delta>0$.

First, it is obvious that $\psi^{\delta}\in C^{\infty}$ and%
\[
|\psi^{\delta}(x)-\psi^{\delta}(y)|\leq{\displaystyle\int_{\mathbb{R}}}|\psi(x-\delta u)-\psi(y-\delta u)|\rho(u)du\leq K|x-y|,
\]
which implies that $\left|\dfrac{d}{dx}\psi^{\delta}(x)\right|\leq K$. Moreover, a straightforward computation shows that the derivatives $\dfrac{d^{2}}{dx^{2}}\psi^{\delta}(x)$  and $\dfrac{d^{3}}{dx^{3}}\psi^{\delta}(x)$ have polynomial growth.
Recalling that $Y_{n}$ belongs to $\mathcal{V}_{T}~$, we complete the proof.\hfill
\end{proof}


\medskip

\begin{thebibliography}{99}


\bibitem {AMN1-00}\textsc{E. Al\`{o}s, O. Mazet, D. Nualart}, \emph{Stochastic
calculus with respect to fractional Brownian motion with Hurst parameter
lesser than }$\frac{1}{2}$. Stoch. Process. Appl. 86,
121-139 (2000)

\bibitem {AMN-01}\textsc{E. Al\`{o}s, O. Mazet, D. Nualart}, \emph{Stochastic
calculus with respect to Gaussian processes}. Ann. Probab. 29, 766-801 (2001)

\bibitem {AN-03}\textsc{E. Al\`{o}s, D. Nualart}, \emph{Stochastic integration
with respect to the fractional Brownian motion}. Stochastics 75,
129-152 (2003)

\bibitem {B-76}\textsc{V. Barbu}, \emph{Nonlinear Semigroups and Differential
Equations in Banach Spaces}, Noordhoff International Publishing, Leiden (1976)

\bibitem {B-05}\textsc{Bender, C.}, \emph{Explicit solutions of a class of linear fractional BSDEs}. Syst. Control Lett. 54, 671-680 (2005)

\bibitem {BHOS-02}\textsc{Biagini, F., Hu, Y., \O ksendal, B., Sulem, A.},
\emph{A stochastic maximal principle for processes driven by fractional Brownian motion}. Stoch. Process. Appl. 100, 233-253 (2002)

\bibitem {BHOZ-06}\textsc{F. Biagini, Y. Hu, B. \O ksendal, T. Zhang},
\emph{Stochastic Calculus for Fractional Brownian Motion and Applications}.
Springer, London (2006)

\bibitem {DU-98}\textsc{L. Decreusefond, A.S. \"{U}st\"{u}nel},
\emph{Stochastic analysis of the fractional Brownian motion}. Potential Anal.
10, 177-214 (1998)

\bibitem {DHP-00}\textsc{T.E. Duncan,Y. Hu, B. Pasik-Duncan}, \emph{Stochastic
calculus for fractional Brownian motion I, Theory}. SIAM J. Control Optim. 38, 582-612 (2000)

\bibitem {H-05}\textsc{Y. Hu}, \emph{Integral transformations and anticipative
calculus for fractional Brownian motion}. Mem. Amer. Math. Soc. 175, (2005) no. 825, viii+127 pp.


\bibitem {HO-03}\textsc{Y. Hu, B. \O ksendal}, \emph{Fractional white noise
calculus and applications to finance}. Infin. Dimens. Anal. Quantum Probab.
Relat. Top. 6, 1-32 (2003)

\bibitem {HP-09}\textsc{Y. Hu, S. Peng}, \emph{Backward stochastic
differential equation driven by fractional Brownian motion}. SIAM J. Control
Optim. 48, 1675-1700 (2009)

\bibitem {MPY-94}\textsc{Ma, J., Protter, P., Yong, J.}, \emph{Solving forward$-$backward stochastic differential equations explicitly$-$a four step scheme}. Probab. Theor. Relat. Fields 98, 339-359 (1994)

\bibitem {MMV-01}\textsc{J. Memin, Y. Mishura, E. Valkeila},
\emph{Inequalities for the moments of Wiener integrals with respect to
fractional Brownian motions}. Stat. Probab. Lett. 55, 421-430 (2001)

\bibitem {M-07}\textsc{Y. Mishura},
\emph{Stochastic Calculus for Fractional Brownian Motion and Related Processes}. Springer, Berlin (2007)

\bibitem {N-95}\textsc{D. Nualart}, \emph{The Malliavin Calculus and Related
Topics} (2nd ed.). Springer, Berlin (2006)

\bibitem {NP-88}\textsc{D. Nualart, E. Pardoux}, \emph{Stochastic calculus
with anticipating integrands}. Probab. Theor. Relat. Fields 78,
535-581 (1988)

\bibitem {PP-92}\textsc{E. Pardoux, S. Peng}, \emph{Backward stochastic
differential equations and quasilinear parabolic partial differential
equations}. In: Stochastic PDE and Their Applications, B.L. Rozovski, R.B.
Sowers eds., LNCIS 176, pp. 200-217. Springer, (1992)

\bibitem {PR-98}\textsc{E. Pardoux, A. R\u{a}\c{s}canu}, \emph{Backward
stochastic differential equations with subdifferential operators and related
variational inequalities}. Stoch. Process. Appl. 76, 191-215 (1998)

\bibitem {PR-11}\textsc{E. Pardoux, A. R\u{a}\c{s}canu}, \emph{Stochastic differential equations, Backward SDEs, Partial differential equations}, to appear.

\bibitem {Y-36}\textsc{L.C. Young}, \emph{An inequality of the H\"{o}der type
connected with Stieltjes integration}. Acta Math. 67, 251-282 (1936)
\end{thebibliography}
\end{document}